%% file: main.tex
\documentclass[a4paper, 10pt]{article}

\input{parts/header.tex}

\begin{document}

\input{parts/title_abstract_etc.tex}
\newpage
\input{parts/introduction.tex}

\input{parts/content.tex}

\newpage
\appendix
\input{parts/technical_proofs_introduction.tex}

\input{parts/technical_proofs_quasi_metric.tex}

\input{parts/technical_proofs_function_spaces.tex}

\input{parts/appendix.tex}

\printbibliography

\end{document}

%% file: parts/header.tex
\usepackage[utf8]{inputenc}
\usepackage[english]{babel}
\usepackage[T1]{fontenc}
\usepackage{lmodern}

\usepackage{amsmath}
\usepackage{amsthm}
\usepackage{xcolor}
\usepackage{enumitem}
\usepackage{multicol}
\usepackage[toc,page]{appendix}
\usepackage{aligned-overset} %
\usepackage{mathtools}
\usepackage{csquotes}

\usepackage{authblk}

\usepackage{zref-clever}
\zcsetup{cap}
\newcommand{\cref}[1]{\zcref{#1}}
\newcommand{\Cref}[1]{\zcref[S]{#1}}

\numberwithin{equation}{section}

\NewDocumentCommand{\newzctheorem}{mO{#1}m}{
  \newtheorem{#1}[sharedtheoremcounter]{#3}
    \AddToHook{env/#1/begin}{%
      \zcsetup{countertype={sharedtheoremcounter=#2}}}
}

\theoremstyle{plain}
\newzctheorem{theorem}{Theorem}
\newzctheorem{lemma}{Lemma}
\newzctheorem{corollary}{Corollary}
\newzctheorem{proposition}{Proposition}
\theoremstyle{definition}
\newzctheorem{definition}{Definition}
\newzctheorem{example}{Example}
\theoremstyle{remark}
\newzctheorem{remark}{Remark}
\newtheorem*{remark*}{Remark}

\MakeOuterQuote{"}

\usepackage[bookmarks=true]{hyperref}

\hypersetup{
	bookmarksnumbered,
	plainpages=false,
	colorlinks=true, %
	linktoc=all,     %
	linktocpage,
	linkcolor=red!70!black,  %
	citecolor=green!80!black,
	filecolor=magenta,
	hidelinks,
	urlcolor=magenta,
	breaklinks,
	unicode=true,
	hypertexnames=false,
}
\usepackage{mathrsfs}
\usepackage{bm}					   %
\usepackage{dsfont}

\usepackage{amssymb}

\usepackage{graphicx}
\usepackage{tikz}

\usepackage[style=alphabetic-verb, %
			sorting=nyt,
			url = false, 
			giveninits=true, %
			eprint = true, %
			isbn= false,
			backend = biber,
			maxbibnames=99,
			useprefix=true]{biblatex}%

\DeclareSourcemap{
  \maps[datatype=bibtex]{
    \map{
      \step[fieldset=language, null]
    }
  }
}

\numberwithin{equation}{section}
\usepackage[a4paper,
hmargin=3cm,bottom=4cm,top=3.5cm,footskip=3\baselineskip
]{geometry}
\allowdisplaybreaks		%

\renewcommand{\subset}{\subseteq}
\renewcommand{\supset}{\supseteq}

\newcommand{\supp}{\operatorname{supp}}
\newcommand{\eps}{\varepsilon}
\newcommand{\dif}{\mathrm{d}}
\newcommand{\one}{\mathds{1}}
\newcommand{\indicator}{\mathds{1}}

\DeclareMathOperator{\dist}{dist}
\DeclareMathOperator{\Lip}{Lip}

\newcommand{\abs}[1]{\left\lvert#1\right\rvert}
\newcommand{\norm}[1]{\left\lVert#1\right\rVert}

\newcommand{\Borel}{\mathcal{B}}

\newcommand{\fcc}{\eta} %
\newcommand{\ConditionZetaText}{Condition~\eqref{condition_zeta}}
\newcommand{\ConditionZetaTextAlphaC}[1]{Condition \allowbreak $\bigl(\text{\ref{condition_zeta}}#1\bigr)$}

\newcommand{\cn}{\mathcal{N}}   %
\newcommand{\extcn}{\mathcal{N}_{\mathrm{ext}}}   %
\newcommand{\en}{e}   %
\newcommand{\umd}{\overline{\operatorname{dim}}_{\mathrm{M},\exp}} %
\newcommand{\lmd}{\underline{\operatorname{dim}}_{\mathrm{M},\exp}} %
\newcommand{\md}{\operatorname{dim}_{\mathrm{M}, \exp}} %
\newcommand{\hdm}{\mathcal{H}_{\exp}} %
\newcommand{\hausd}{\operatorname{dim}_{H, \exp}} %
\newcommand{\metr}{\varrho}
\newcommand{\num}{\#}
\newcommand{\ball}{B}
\newcommand{\sigmacap}{\Cap}
\newcommand{\SC}{\mathscr{C}} %

\newcommand{\om}[1]{#1^{\ast}} %
\newcommand{\CR}{t} %
\newcommand{\AR}{t} %

\newcommand{\dis}{\mathcal{D}'} %
\newcommand{\td}{\mathscr{S}'} %
\newcommand{\isA}[1][]{\ifthenelse{\equal{#1}{}}{A_{p,q}^{s}}{(A#1)_{p#1,q#1}^{s#1}}} %
\newcommand{\embeds}{\hookrightarrow}

\newcommand{\scrit}{s_{\ast}}

\newcommand{\linspan}{\operatorname{span}}

\makeatletter
\newcommand*{\transpose}{{\mathpalette\@transpose{}}}
\newcommand*{\@transpose}[2]{\raisebox{\depth}{$\m@th#1\intercal$}}
\makeatother

\newcommand{\RR}{\mathbb{R}}
\newcommand{\R}{\RR}
\newcommand{\NN}{\mathbb{N}}
\newcommand{\N}{\NN}
\newcommand{\ZZ}{\mathbb{Z}}
\newcommand{\Z}{\ZZ}
\newcommand{\CC}{\mathbb{C}}

\newcommand{\PP}{\mathbb{P}}
\newcommand{\XX}{\mathbb{X}}
\newcommand{\X}{\XX}
\newcommand{\YY}{\mathbb{Y}}
\newcommand{\Y}{\YY}
\newcommand{\FF}{\mathbb{F}}

\newcommand{\CalB}{\mathcal{B}}
\newcommand{\CalG}{\mathcal{G}}
\newcommand{\CalO}{\mathcal{O}}
\newcommand{\CalA}{\mathcal{A}}
\newcommand{\CalH}{\mathcal{H}}
\newcommand{\CalR}{\mathcal{R}}
\newcommand{\CalN}{\mathcal{N}}
\newcommand{\CalW}{\mathcal{W}}
\newcommand{\CalP}{\mathcal{P}}
\newcommand{\CalZ}{\mathcal{Z}}
\newcommand{\CalNN}{\mathcal{NN}}
\newcommand{\diam}{\operatorname{diam}}

\renewcommand{\emptyset}{\varnothing}

%% file: parts/title_abstract_etc.tex
\renewcommand\Authfont{\sffamily}
\renewcommand\Affilfont{\sffamily\small}

\title{\vspace{-1cm}Maximally Spread Out Measures and Implications\\{} for Phase Transitions in Approximation Theory}

\author[$\dagger$]{Hannes Matt}
\author[$\ddagger$]{Erwin Riegler}
\author[$\dagger$]{Felix Voigtlaender\thanks{FV acknowledges support by the German Science Foundation (DFG)
in the context of the Emmy Noether junior research group VO 2594/1-1.
FV and HM acknowledge support by the Hightech Agenda Bavaria.}}

\affil[$\dagger$]{Mathematical Institute for Machine Learning and Data Science (MIDS),\authorcr
Catholic University of Eichstätt–Ingolstadt,
Germany}
\affil[$\,$]{\texttt{\{hannes.matt, felix.voigtlaender\}@ku.de}
\vspace{.5em}}

\affil[$\ddagger$]{ETH Z{\"u}rich, Z{\"u}rich, Switzerland}
\affil[$\,$]{\texttt{eriegler@mins.ee.ethz.ch}}

\date{\vspace*{-0.9cm}}

\maketitle

\begin{abstract}
We establish the existence of a "maximally spread out" Borel probability measure on a totally bounded subset $\SC$
of a (quasi)-Banach space $\XX$ under two mild conditions:
\emph{(i)} a growth condition on the covering numbers $\cn(\SC, \eps)$ of $\SC$,
and \emph{(ii)} a technical topological condition that is in particular satisfied
whenever $\SC \subset \XX$ is closed, bounded, and convex.

More formally, condition \emph{(i)} requires that the so-called
\emph{lower power-exponential Minkowski dimension} of $\SC$, i.e.,
\vspace*{-0.3cm}
\[
  s_\ast
  := \lmd(\SC)
  := \liminf_{\eps \downarrow 0} \frac{\log_2 \log_2 \cn(\SC, \eps)}{\log_2(1/\eps)}
  \tag{$\ast$}
\]
satisfies $s_\ast > 0$.
Under these conditions, we construct a Borel probability measure $\mu$ on $\XX$ that is \emph{critical} for $\SC$,
or maximally spread out, meaning that
the associated outer measure $\mu^\ast$ satisfies $\mu^\ast (\XX \setminus \SC)= 0$
and furthermore satisfies for every $0 < s < s_\ast$ the small-ball condition
\[
  \mu^\ast (\ball(x,r)) \leq \exp\bigl(- c(s) \cdot (1/r)^s\bigr)
  \quad \text{ for all } x \in \X \text{ and } 0 < r < r_0 (s)
  ,
  \tag{$\dagger$}
\]
where $\ball(x,r)$ is the ball around $x$ of radius $r$.

The existence of such a critical measure in particular implies
that the so-called \emph{power-exponential Hausdorff dimension}
$\hausd (\SC)$ of $\SC$ introduced in [J.~Topol.~Anal.~4(2):203--235, 2012]
satisfies $\hausd(\SC) = \lmd(\SC)$.

Such a critical measure gives rise to a phase transition
regarding lossy compression of elements of $\SC$,
under the further condition that the $\liminf$ in $(\ast)$ exists as an actual limit.
Indeed, in this case the dimension $s_\ast$ in $(\ast)$ describes the optimal compression rate for $\SC$,
meaning that for every $\tau < 1/s_\ast$, the elements of $\SC$ can be encoded using $n$ bits
and recovered up to error $\mathcal{O}(n^{-\tau})$,
whereas we show that for every fixed encoding/decoding scheme 
the set of elements $x \in \SC$ that can be encoded at rate $\mathcal{O}(n^{-\tau})$ for some $\tau> 1/s_\ast$ forms a $\mu$-null set.
Similar phase transitions occur when considering approximation-theoretic
properties of $\SC$, such as non-linear $n$-term approximation using a given dictionary,
or approximation using neural networks with complexity controlled by $n \in \NN$.

Previous work constructed critical measures for unit balls of certain Besov spaces,
considered as subsets of $L^2$, by utilizing the description of the Besov norm via wavelet bases.
In contrast, our construction is completely general.
In particular, our results are basis-independent and apply to balls of Besov- or Triebel-Lizorkin spaces,
considered as a subset of another such space, for the full range of parameters
and also for spaces of dominating mixed smoothness, as long as the embedding is compact.
Thus, we greatly generalize, unify, and simplify earlier results.

\end{abstract}


\vspace{0.2cm}


%% file: parts/introduction.tex
\section{Introduction, main results, and applications}
\label{sec:intro}

For subsets $\SC \subset \R^d$, it is well-known that the
\emph{Hausdorff dimension} and the \emph{Minkowski dimension} of $\SC$
agree for sufficiently regular sets; see e.g.\ \cite[Theorem 5.7]{mattilaGeometrySetsMeasures1995},
but not for arbitrary sets; see e.g.\ \cite[p.77]{mattilaGeometrySetsMeasures1995} for a counter-example.
In this paper, we derive a version of this result adapted to the case
of subsets $\SC$ of infinite-dimensional spaces,
where instead of the "finite-dimensional polynomial scaling behavior"
$\cn(\SC, \eps) \approx (1/\eps)^{s_\ast}$, the covering numbers $\cn(\SC,\eps)$
more frequently scale ``power-exponentially'' like
$\cn(\SC,\eps) \approx \exp(C \cdot (1/\eps)^{s_\ast})$ for some $s_\ast = s_\ast (\SC) > 0$.

More precisely, let $\SC$ be a subset of a (quasi)-Banach space $\X$
and assume that the \emph{lower power-exponential Minkowski dimension} of $\SC$,
defined in terms of the covering numbers $\cn(\SC, \eps)$ of $\SC$ as
\begin{equation}
  \lmd (\SC)
  := \liminf_{\eps \downarrow 0}
       \frac{\log_2 \log_2 \cn(\SC, \eps)}{\log_2 (1/\eps)}
  \in [0,\infty],
  \label{eq:LowerMinkowskiDimension}
\end{equation}
satisfies $\lmd(\SC) > 0$.
Roughly speaking, this means that $s = \lmd(\SC)$ is the largest number $s > 0$ such that
for each $0\le \sigma <s$, we have $\cn(\SC, \eps) \geq \exp(C_\sigma \cdot (1/\eps)^{\sigma})$
as $\eps \to 0$, for some constant $C_\sigma > 0$.

Under certain regularity conditions on the set $\SC$ (see \cref{def:condition_zeta} below),
which, in particular, include the case when $\SC$ is closed, bounded, and convex,
we prove the existence of a "maximally spread out" probability measure $\mu$
on $\XX$ that is concentrated on $\SC$.
Here, the measure is "maximally spread out" in the sense that
the outer measure $\om{\mu}$ associated to
$\mu$ satisfies for every ${0 < \sigma < \lmd(\SC)}$
the \emph{small-ball condition}\footnote{
  Strictly speaking, we require \Cref{eq:SmallBallCondition} to hold with
  the outer measure $\mu^\ast (\ball(x,r))$
  instead of $\mu (\ball(x,r))$ on the left-hand side,
  to account for the fact that in a quasi-Banach space,
  balls are not necessarily Borel measurable.
  See \Cref{sub:QuasiNormedSpacesOuterMeasures} for more details on this.
}
\begin{equation}
    \om{\mu} \bigl(\ball(x,r)\bigr) \leq \exp\bigl( - c(\sigma) \cdot (1/r)^\sigma \bigr)
  \quad \text{for all } x \in \X \text{ and } 0 < r < r_0 (\sigma)
  \label{eq:SmallBallCondition}
\end{equation}
for certain $c(\sigma), r_0 (\sigma) > 0$.
The small-ball condition formalizes the notion of a ``spread-out'' measure,
since it guarantees that the measure cannot concentrate on small balls.
Regarding ``maximality'', we will see 
in \cref{cor:SmallBallCondition_Dominated_by_LowerMinkowskiDim}
that no measure
with $\mu^\ast (\SC) > 0$ can satisfy the small-ball condition with parameter $s > \lmd(\SC)$.
The measure $\mu$ can be seen as an analog of a "Frostman measure",
see \cite[Chapter 3]{bishopFractalsProbabilityAnalysis2017}, adapted to the modified
(i.e., power-exponential) scaling behavior of the covering numbers.

As we will see, the existence of such a measure has two important implications:
\begin{itemize}
  \item It implies that the so-called
        \emph{power-exponential Hausdorff dimension} $\hausd (\SC)$ of $\SC$
        --- a quantity introduced in \cite{kloecknerGeneralizationHausdorffDimension2012},
        which, in general, is relatively difficult to compute ---
        is simply given by
        \[
          \hausd (\SC) = \lmd(\SC)
          .
        \]
        This is in analogy to the finite-dimensional setting where the Hausdorff
        and Minkowski dimensions are equal in many situations (and in particular for convex sets).
        This fact and the definition of the power-exponential Hausdorff dimension
        are discussed in more detail in \cref{sub:PowerExponentialMinkowskiHausdorffDimension}
        and \cref{sub:MainResult} below.

  \item If, in addition, we assume that the power-exponential Minkowski dimension
        $\md (\SC)$ of $\SC$ exists
        (meaning that the limit inferior in \Cref{eq:LowerMinkowskiDimension} exists as a true limit;
        in which case $\md(\SC) := \lmd(\SC)$)
        and if $0 < \md (\SC) < \infty$, then the measure $\mu$ gives rise to a \emph{phase transition}
        regarding the achievable \emph{compression rates},
        as first described in \cite{grohsPhaseTransitionsRate2023}:
        \begin{itemize}
            \item There exists a coding scheme such that for every $\tau <1/\md(\SC)$, the elements of $\SC$ are described (encoded and decoded by this coding scheme) up to error $\CalO(n^{-\tau})$ using $n$ bits.

            \item For every coding scheme, the set of elements of $\SC$ that can be encoded up to error $\mathcal{O}(n^{-\tau})$ using $n$ bits (for every $n\in \NN$) for some $\tau>1/\md(\SC)$ forms a $\mu$-null set.

        \end{itemize}
        Our results thus greatly generalize, unify, and simplify the earlier results
        in \cite{grohsPhaseTransitionsRate2023}, which only showed the existence
        of such a maximally spread out measure $\mu$ for the special case where $\XX = L^2$
        and $\SC \subset L^2$ is a unit ball of certain Besov- or Sobolev spaces
        that compactly embed into $L^2$ (or $\SC$ is an image of one of these unit balls
        under a suitable sub-Lipschitz map).
\end{itemize}

Before we discuss the notions and results outlined above in more detail,
we mention that instead of requiring $\SC$ to be closed, bounded, and convex,
it suffices to impose the following weaker condition.

\begin{definition} \label{def:condition_zeta}
  Let $\XX$ be a quasi-normed space and let $\emptyset\ne\SC\subset \XX$.
  We then say that $\SC$ satisfies {\ConditionZetaText}
  if there exist constants $C, \alpha > 0$ satisfying
  \begin{equation}
    \sum_{n=1}^{\infty} \frac{x_n}{C \cdot n^\alpha}
    \in \SC
    \quad \text{whenever} \quad x_1,x_2,\dots \in \SC
    .
    \label{condition_zeta}
    \tag{$\zeta$}
  \end{equation}
  More precisely, we then say that $\SC$ satisfies \ConditionZetaTextAlphaC{(\alpha,C)}.
\end{definition}

{\ConditionZetaText} is in particular satisfied for every complete, bounded, and convex set $\SC$
(see \Cref{lem:ConvexSetsAreQuasiConvex}),
but applies more generally, in particular
whenever $\SC$ is the image of the unit ball under a continuous linear operator $T : \Y \to \X$
between quasi-Banach spaces $\X, \Y$;
see \Cref{quasi_Banach_unit_balls_are_quasi_convex,quasi-convexity_preserved}.
Conversely, {\ConditionZetaText} includes a certain notion of completeness and implies boundedness,
see \Cref{remark_condition_zeta_implies} for the latter.

\medskip{}

In the remainder of this introduction, we describe our main results in more detail.
We start with a primer on quasi-metric spaces (\Cref{sub:QuasiNormedSpacesOuterMeasures}),
focusing in particular on the issues regarding measurability that this entails.
The reader interested only in the case of genuine metric spaces can safely skip this section,
after noting that in this setting the outer measure $\om{\mu}(B)$,
which will appear throughout the technical parts of the paper, is always equal to $\mu(B)$,
where $B$ denotes a closed ball.
In \Cref{sub:PowerExponentialMinkowskiHausdorffDimension},
following \cite{kloecknerGeneralizationHausdorffDimension2012}, we discuss the concepts of the
\emph{power-exponential Minkowski dimension} and \emph{power-exponential Hausdorff dimension}
that were mentioned above.
These generalize the concepts of the classical Minkowski- and Hausdorff dimensions,
which are based on the \emph{polynomial scale} $r \mapsto r^s$, $s > 0$,
to the \emph{power-exponential scale} $r \mapsto \exp(- (1/r)^s)$, $s > 0$,
which is more useful for subsets of infinite-dimensional spaces.
The main results regarding the existence of maximally spread out measures and the consequences
for the power-exponential Hausdorff dimension are presented in \Cref{sub:MainResult},
while the implications regarding phase transitions for lossy compression and approximation theory
are discussed in \Cref{sub:PhaseTransitionConsequences}.
Finally, \Cref{sub:ImplicationsForClassicalSpaces} shows that our results are easily applicable
to a wide range of scenarios: in fact, they apply in a straightforward way
to (unit balls of) Besov- and Triebel-Lizorkin spaces,
both for the classical spaces and for spaces of dominating mixed smoothness.
The relations between our results and the existing literature are
discussed in \cref{sub:RelatedWork}.
\cref{sub:Notation} introduces additional notation and
in \cref{sub:PaperStructure}, we lay out the structure of the remaining part of the paper.

\subsection{Quasi-metric spaces and outer measures}%
\label{sub:QuasiNormedSpacesOuterMeasures}

This section gives a brief introduction to quasi-metric spaces and discusses potential issues
regarding measurability.
This section can be safely skipped by readers who are only interested in metric spaces,
in which case $\om{\mu}(\ball)= \mu(\ball)$ for all balls $\ball$.

In approximation theory and more generally in the theory of function spaces,
one frequently needs to consider \emph{quasi-normed spaces} instead of normed spaces
in order to cover all spaces of interest.
Natural examples for this are the $\ell^p$ spaces,
where a sequence ${x = (x_n)_{n \in \N} \in \ell^p}$ has better decay properties for smaller $p$,
so that one would also like to allow choosing ${p \in (0,1)}$.
In this case, $\ell^p$ is a quasi-Banach space, but no longer a Banach space.
Thus, the present subsection briefly discusses the most important properties of such
quasi-normed spaces (as well as quasi-metric spaces) that we will need,
including the resulting implications regarding the topology and measurability.

A \emph{quasi-metric} on a set $\XX$ is a map $d : \XX \times \XX \to [0,\infty)$
satisfying the usual axioms of a metric, except that the triangle inequality is replaced
by the \emph{quasi-triangle inequality}. 
The latter requires the existence of some constant $k \geq 1$, called the \emph{triangle constant}
or \emph{modulus of concavity}, satisfying
\[
  d(x,z) \leq k \cdot \bigl(d(x,y) + d(y,z)\bigr)
  \quad \text{for all } x,y,z \in \XX
  .
\]
Similarly, a \emph{quasi-norm} on a vector space $\XX$ over the field $\FF \in \{ \R, \CC \}$
is a map $\| \cdot \| : \XX \to [0,\infty)$ with the following properties:
\begin{enumerate}[label=(\roman*)]
  \item Definiteness: $\| x \| = 0$ if and only if $x = 0$;
  \item Homogeneity: $\| \alpha x \| = |\alpha| \cdot \| x \|$ for $\alpha \in \FF$ and $x \in \XX$;
  \item "quasi-triangle inequality": there exists a constant $k \geq 1$, called the
        \emph{modulus of concavity} or \emph{triangle constant} satisfying
        \[
          \| x + y \| \leq k \cdot (\| x \| + \| y \|)
          \quad \text{for all } \quad
          x,y \in \XX
          .
        \]
\end{enumerate}
In this case, a quasi-metric on $\XX$ is defined via
$d(x,y) := d_{\| \cdot \|} (x,y) := \| x - y \|$.
Unless stated otherwise, we will always interpret a quasi-normed space as a quasi-metric space
in this way.

The \emph{topology} on a quasi-metric space $(\XX, d)$ is defined
by declaring $U \subset \XX$ as open if
\[
  \forall \, x \in U \quad
  \exists \, r > 0 : \quad
  \{ y \in \XX \,\,:\,\,  d(x,y) < r\} \subset U
  .
\]
It is easy to see that this is equivalent to the condition
\[
  \forall \, x \in U \quad
  \exists \, r > 0 : \quad
  \ball(x,r) \subset U
  ,
\]
where
\[
  \ball(x, r)
  := \bigl\{ y \in \XX \,\,:\,\, d(x,y) \leq r \bigr\}
\]
is the "closed" ball of radius $r$ around $x$.
The \emph{Borel $\sigma$-algebra} on $\XX$ is then the $\sigma$-algebra generated by the
family of all open sets.
We will denote it by $\Borel(\XX)$ or by $\Borel(\XX,d)$.

It is crucial to note, however, that unlike in the setting of genuine metric spaces,
the "closed" ball $\ball(x,r)$ is \emph{not} generally closed
(in the sense of its complement being an open set).
Neither is the "open" ball $\{ y \in \XX \,\,:\,\, d(x,y) < r \}$ generally open.
This means that the quasi-metric $d$ of a quasi-metric space $\XX$
is in general \emph{not} continuous with respect to the topology generated by it.
Likewise, a quasi-norm is not necessarily continuous with respect to the topology
that it generates; in fact, it can happen that the open or closed balls
and the quasi-norm are not even Borel measurable.
For an example where this happens,
see \cite[Remark~2.1.9]{voigtlaenderEmbeddingTheoremsDecomposition2015}.
To deal with these issues, we will employ two main tools:
outer measures and the Aoki-Rolewicz theorem, which we now discuss.

In order to make sense of expressions such as $\mu(\ball(x,r))$ for a Borel measure $\mu$
on a quasi-metric space $\XX$, even in cases where the ``closed'' ball $\ball(x,r)$
might or might not be measurable, we will use the \emph{outer measure}
$\mu^\ast$ associated to $\mu$.
Generally, if $\mu : \CalA \to [0,\infty]$ is a measure defined on a $\sigma$-algebra $\CalA$
on a set $\XX$, \emph{the associated outer measure $\mu^\ast$} is defined as
\begin{equation}
  \mu^\ast (M)
  := \inf \bigl\{ \mu(A) \,\,:\,\, A \in \CalA \text{ with } A \supset M \bigr\}
  \in [0,\infty]
  \quad \text{for \emph{arbitrary} subsets} \quad
  M \subset \XX
  .
  \label{eq:OuterMeasure}
\end{equation}
It is easy to see that this outer measure has the following crucial properties:
\begin{enumerate}[label=(\roman*)]
  \item $\mu^\ast$ is monotone, meaning $\mu^\ast (M) \leq \mu^\ast \bigl(\widetilde{M}\bigr)$
        if $M \subset \widetilde{M}$.

  \item $\mu^\ast$ is $\sigma$-subadditive, meaning
        $\mu^\ast \bigl(\bigcup_{n \in \N} M_n\bigr) \leq \sum_{n=1}^{\infty} \mu^\ast (M_n)$
        for arbitrary sets $M_n \subset \XX$, $n \in \N$.

  \item $\mu^\ast$ is an extension of the measure $\mu$,
        meaning that for measurable sets $A \in \CalA$, we have $\mu^\ast (A) = \mu(A)$.
        In particular, $\mu^\ast (\emptyset) = 0$.
\end{enumerate}
In particular, if we provide bounds for $\mu^\ast (\ball(x,r))$,
then this is actually a bound for $\mu(\ball(x,r))$ in cases where the balls are measurable,
which in particular holds if $(\XX, \norm{\cdot})$ is a normed vector space
rather than a quasi-normed vector space.

The Aoki-Rolewicz theorem (see below) states that every (possibly pathological) quasi-norm
can, for many practical purposes, be replaced by an equivalent "nice" quasi-norm,
meaning a $p$-norm for some $p \in (0,1]$.
Here, two (quasi)-norms $\| \cdot \|$ and $\| \cdot \|^\ast$ on a vector space $\XX$
are called \emph{equivalent} if there exists $C \geq 1$ with
\[
  C^{-1} \cdot \| x \| \leq \| x \|^\ast \leq C \cdot \| x \|
  \qquad \text{for all } x \in \XX
  .
\]
It is easy to see that equivalent (quasi)-norms induce the same topology.

Finally, for given $p \in (0,1]$ we say that a (quasi)-norm $\| \cdot \|$ is a \emph{$p$-norm},
if it satisfies the \emph{$p$-triangle inequality}
\[
  \| x + y \|^p \leq \| x \|^p + \| y \|^p
  \qquad \text{for all } x,y \in \XX
  .
\]
In this case, one can define a metric via $d(x,y) := \| x - y \|^p$
and this metric induces the same topology as $\| \cdot \|$.
In particular, this implies that \emph{a $p$-norm is continuous with respect
to its induced topology}.

Given this terminology, we can now state the Aoki-Rolewicz theorem.
A proof can be found in \cite[Theorem~2.1.4]{voigtlaenderEmbeddingTheoremsDecomposition2015}
and \cite[Chapter~2, Theorem~1.1]{DeVoreConstructiveApproximation}.

\begin{theorem}[Aoki-Rolewicz]\label{thm:Aoki_Rolewicz}
    Let $(\XX,\norm{\cdot})$ be a quasi-normed space with modulus of concavity $k \geq 1$.
    Then there exists an equivalent quasi-norm $\norm{\cdot}_{\ast}$ on $\XX$
    such that $\norm{\cdot}_{\ast}$ is a $p$-norm with $p = \frac{1}{1 + \log_2(k)}$.
\end{theorem}

\subsection{The power-exponential Minkowski- and Hausdorff dimensions,
and the small-ball condition for measures}%
\label{sub:PowerExponentialMinkowskiHausdorffDimension}

The classical notions of Minkowski- and Hausdorff dimensions are based on
the \emph{polynomial scale} consisting of the functions $r \mapsto r^s$, $s > 0$.
For instance, the Minkowski dimension of a set $\SC$ is --- roughly speaking --- the exponent
$s$ for which the covering numbers $\cn(\SC, \eps)$ satisfy $\cn(\SC, \eps) \approx (1/\eps)^s$.
Likewise, the definition of the $s$-dimensional Hausdorff measure,
\[
  \CalH^s (M)
  = \lim_{\delta \downarrow 0} \,\,
    \inf
    \left\{ 
      \sum_{i=1}^{\infty} (\diam (M_i))^s
      \,\,:\,\,
      M_i \subset \XX \text{ for $i \in \N$, with }
      \bigcup_{i=1}^\infty M_i \!\supset\! M,
      \,\, \diam (M_i) \leq \delta
    \right\}
\]
for subsets $M\subseteq \XX$ of a metric space $\XX$,
is crucially based on using the $s$-th power of the diameter $\diam (M_i)$.

Instead of the polynomial scaling $\cn(\SC, \eps) \approx (1/\eps)^s$ that is characteristic
of finite-di\-men\-sion\-al sets, many interesting subsets of \emph{infinite}-dimensional spaces
obey the \emph{power-exponential scaling}
\[
  \cn(\SC, \eps) \approx \exp\bigl(C \cdot (1/\eps)^s\bigr)
\]
for some $s > 0$.
For our purposes, it will thus be appropriate to consider variants of the Minkowski-
and Hausdorff dimensions adapted to this scaling.
Our description of the resulting objects will mainly follow
\cite{kloecknerGeneralizationHausdorffDimension2012}.
However, since \cite{kloecknerGeneralizationHausdorffDimension2012} only considers the case of
metric spaces, whereas we work in the setting of \emph{quasi}-normed spaces,
we provide proofs for several of the required results from
\cite{kloecknerGeneralizationHausdorffDimension2012}
in the quasi-metric setting; see \Cref{sec:QuasiMetricTechnicalProofs}.

\subsubsection{The power-exponential Minkowski dimension}
\label{sub:power-exp_Minkowski_dimension}

The \emph{lower power-exponential Minkowski dimension} $\lmd$
was already introduced in \Cref{eq:LowerMinkowskiDimension}%
\footnote{
  Strictly speaking, \Cref{eq:LowerMinkowskiDimension} only applies for the case where
  $\SC$ is totally bounded and $\num \SC \geq 2$.
  The former ensures that $\cn(\SC,\eps)\in \NN$
  and the latter ensures that $\cn(\SC, \eps) \geq 2$ for $\eps$ small enough.
  In that case $\log_2(\cn(\SC,\eps))\ge1$ exists for $\eps$ small enough,
  which then ensures that $\log_2 (\log_2 (\cn(\SC, \eps))) \in [0,\infty)$ is well-defined.
  It follows that $\lmd(\SC)\ge 0$.
  In case of $\# \SC < 2$, we simply define $\lmd(\SC) := \umd(\SC) := \md(\SC) := 0$.
  In the case when $\SC$ is not totally bounded,
  we define $\lmd(\SC) := \umd(\SC) := \md(\SC) := \infty$.
}.
We will also have occasion to consider the \emph{upper power-exponential Minkowski dimension}
\[
  \umd (\SC)
  := \limsup_{\eps \downarrow 0}
       \frac{\log_2 \log_2 \cn(\SC, \eps)}{\log_2 (1/\eps)}
  \in [0,\infty]
  .
\]
We say that the \emph{power-exponential Minkowski dimension}
$\md (\SC)$ of $\SC$ exists if the lower and upper power-exponential
Minkowski dimensions agree, i.e., $\lmd (\SC) = \umd(\SC)$;
in this case, we set $\md(\SC) := \lmd(\SC)$.
We note that our terminology slightly deviates from that in
\cite[Definition~3.1]{kloecknerGeneralizationHausdorffDimension2012},
where the (lower/upper) power-ex\-po\-nen\-tial Minkowski dimensions are called
the \emph{(lower/upper) Minkowski critical parameters}.

\subsubsection{The power-exponential Hausdorff dimension}
\label{sub:power-exp_Hausdorff_dimension}

To introduce the \emph{power-exponential Hausdorff dimension}, we first define the
$f$-Hausdorff measure associated to a non-decreasing continuous function $f : [0,\infty) \to [0,\infty)$
with $f(0) = 0$.
Given a (quasi)-metric space $(\XX, d)$, this outer measure is defined as
\begin{equation*}
    \CalH^f(M) 
    := \sup_{\delta>0} \CalH^{f,\delta}(M),
\end{equation*}
for arbitrary $M\subset \XX$,
where
\begin{equation*}
    \CalH^{f,\delta}(M):=
     \inf
     \left\{ 
       \sum_{i=1}^{\infty} f(\diam(M_i))
       \,:\,
       M_i \subset \XX \text{ for $i \!\in\! \N$, with }
       \bigcup_{i=1}^\infty M_i \supset M,
       \,\, \diam (M_i) \leq \delta
     \right\}
\end{equation*}
for $\delta>0$.
Here, the diameter of a subset $M_i \subset \XX$ is defined as
\[
  \diam (M_i) := \sup_{x,y \in M_i} d(x,y)
  \quad \text{for } \emptyset \neq M_i \subset \XX
  ,
\]
and $\diam(\emptyset) := 0$.
Clearly, we have $\CalH^{f}(M) = \lim_{\delta\downarrow 0} \CalH^{f,\delta}(M)$.

Now, for $s \in (0,\infty)$, following
\cite[Section~2.2]{kloecknerGeneralizationHausdorffDimension2012}, we define
\begin{equation}
  f_s : \quad
  [0,\infty) \to [0,\infty), \quad
  f_s (r) = \begin{cases}
              \exp(- (1/r)^s), & \text{if } r > 0, \\
              0,               & \text{if } r = 0.
            \end{cases}
  \label{eq:PEHausdorffHelperFunction}
\end{equation}
Based on this, the \emph{power-exponential Hausdorff measure with exponent $s$}
is defined as
\begin{equation}
  \CalH_{\exp}^s := \CalH^{f_s}
  .
  \label{eq:PEHausdorffMeasureDefinition}
\end{equation}
Finally, we define the \emph{power-exponential Hausdorff dimension} of $M\subset \XX$ as
\begin{equation}
  \begin{aligned}
    \hausd (M)
    & := \sup \bigl\{ s \in (0,\infty) \,\,:\,\, \CalH_{\exp}^s (M) = \infty \bigr\} \\
    &\phantom{:}= \sup \bigl\{ s \in (0,\infty) \,\,:\,\, \CalH_{\exp}^s (M) > 0 \bigr\} \\
    &\phantom{:}= \inf \bigl\{ s \in (0,\infty) \,\,:\,\, \CalH_{\exp}^s (M) = 0 \bigr\} \\
    &\phantom{:}= \inf \bigl\{ s \in (0,\infty) \,\,:\,\, \CalH_{\exp}^s (M) < \infty \bigr\}
     \in [0,\infty]
     .
  \end{aligned}
  \label{eq:PEHausdorffDimensionDefinition}
\end{equation}
Here, we use the convention that $\sup \emptyset = 0$ and $\inf \emptyset = \infty$.
The fact that these four expressions all coincide
is established in \Cref{sec:proof_of_eq:PEHausdorffDimensionDefinition};
see also \cite[Lemma~2.2 and Def.~2.3]{kloecknerGeneralizationHausdorffDimension2012}.
Finally, we remark that in \cite{kloecknerGeneralizationHausdorffDimension2012}
the term \emph{critical parameter of $M$ with respect to the power-exponential scale}
is used instead of the term "power-exponential Hausdorff dimension" that we use.

\subsubsection{Relations between power-exponential Hausdorff and Minkowski dimensions}

In the case of the classical Hausdorff and Minkowski dimensions, it is well-known
that the Hausdorff dimension is always less than or equal to the lower Minkowski dimension;
see for instance \cite[Page~77]{mattilaGeometrySetsMeasures1995}.
As we will show in \Cref{sec:proof_of_eq:PEHausdorffDominatedByPELowerMinkowski}
(see \cref{lem:PEHausdorffDominatedByPEMinkowski}), the same also holds for the
power-exponential versions of these quantities, meaning that
\begin{equation}
  \hausd (\SC) \leq \lmd (\SC)
  ;
  \label{eq:PEHausdorffDominatedByPELowerMinkowski}
\end{equation}
see also \cite[Proposition~3.2]{kloecknerGeneralizationHausdorffDimension2012}.
For general sets, a converse inequality cannot be expected;
indeed, if $\SC$ is a countable set, then always $\hausd (\SC) = 0$,
but $\lmd(\SC)$ can have any value in $[0,\infty]$, as can be seen by taking $\SC$
to be a countable dense subset of a known set with the specified lower power-exponential
Minkowski dimension.
However, one of the main results of the present paper
(see \cref{thm:MainResultCriticalMeasureExistence}) will show that
\[
    \text{if}
    \quad
    \begin{array}{c}
        \XX \text{ is a quasi-Banach space and}\\
        \SC \subset \XX \text{ satisfies \ConditionZetaText}
    \end{array}
    \qquad
    \text{then}
    \quad
  \hausd (\SC) = \lmd(\SC),
\]
and this will be done by constructing a suitable measure on $\SC$.

\subsubsection{Measures satisfying a small-ball condition}
\label{sub:measures_sbc}

A useful tool for deriving lower bounds on the classical Hausdorff dimension
is (the easy direction of) the \emph{Frostman lemma}
(see \cite[Theorem 8.8]{mattilaGeometrySetsMeasures1995}), which involves measures satisfying
the decay condition $\mu(\ball(x, r)) \leq r^s$ for some $s \geq 0$.
In the present setting, the relevant classes of measures are those that satisfy
the small-ball condition, as made precise in the following \cref{def:small_ball}.
\begin{definition}
    \label{def:small_ball}
    Let $(\XX,d)$ be a quasi-metric space and let $\mu$ be a measure on some $\sigma$-algebra on $\XX$.
    Let $s\ge 0$.
    We say that $\mu$ satisfies the \emph{small-ball condition with parameter $s$}, if there exist constants $c,r_0>0$ (depending on $s$) such that
    \begin{equation}
      \mu^\ast \bigl(\ball(x,r)\bigr)
      \leq \exp\bigl( -c \cdot (1/r)^s \bigr)
      \qquad \forall \, x\in \XX \text{ and } 0 < r < r_0
      .
      \label{eq:SmallBallConditionPrecise}
    \end{equation}
\end{definition}
Using this terminology, we get the following important tool for lower-bounding the
power-exponential Hausdorff dimension of a set:

\begin{lemma}[Version of the easy direction of the Frostman lemma]
    \label{lem:FrostmanLemmaEasy}
    Let $(\XX,d)$ be a quasi-metric space and let $\SC \subset \XX$.
    Let $s \in [0,\infty)$ and suppose that $\mu$ is a measure on some $\sigma$-algebra on $\XX$ which satisfies
    the small-ball condition with parameter $s$ and also satisfies $\mu^{\ast}(\SC) > 0$.
    Then $\hausd(\SC) \geq s$.
\end{lemma}

\begin{proof}
  See \cref{sec:proof_of_lem:FrostmanLemmaEasy}.
\end{proof}

By the previous lemma and \cref{eq:PEHausdorffDominatedByPELowerMinkowski},
we see that no measure $\mu$ on $\XX$ with $\mu^\ast (\SC) > 0$
can satisfy the small-ball condition with a parameter $s > \lmd(\SC)$.
\begin{corollary}
    \label{cor:SmallBallCondition_Dominated_by_LowerMinkowskiDim}
    Let $\XX,\SC,\mu,s$ be as in \cref{lem:FrostmanLemmaEasy}.
    Then $s \le \lmd(\SC)$.
\end{corollary}

Motivated by this, we introduce a terminology for measures that satisfy
the small-ball condition for the "largest possible" range of parameters $s$
(ignoring the case $s = \lmd(\SC)$).

\begin{definition}\label{def:CriticalMeasure}
    Let $\XX$ be a quasi-metric space and let $\SC \subset \XX$.
    We say that a measure $\mu$ on $\XX$ is \emph{critical for $\SC$},
    if it satisfies $\om{\mu}(\SC)>0$ and the small-ball condition with parameter $s$
    for every $0 < s < \lmd(\SC)$.
\end{definition}

The notion of a critical measure was first introduced in
\cite[Definition~3]{grohsPhaseTransitionsRate2023},
where these were also called ``measures of critical growth''.
Our notion of a critical measure is a generalization thereof in the following sense:
Firstly, we consider subsets of quasi-metric spaces instead of Banach spaces.
Secondly, in our present article, the small-ball condition is required to hold
for all ${0 < s < \lmd(\SC) \vphantom{\sum_j^i}}$, while in \cite{grohsPhaseTransitionsRate2023},
the small-ball condition is required to hold for all $0 <s < \umd(\SC) \vphantom{\sum_j^i}$.
For all the specific sets $\SC$ considered in \cite{grohsPhaseTransitionsRate2023},
it holds that $\lmd(\SC) = \umd(\SC) \vphantom{\sum_j^i}$, and so both ranges
for the parameter $s$ are equal.
In view of 
\cref{cor:SmallBallCondition_Dominated_by_LowerMinkowskiDim},
our definition is the natural generalization to cases where $\lmd(\SC)< \umd(\SC)$ might hold.
For an example of a convex compact subset where the upper and lower dimensions differ,
see \cref{sec:example_lmd<umd}.

\begin{remark}\label{rem:CriticalMeasuresImpliesHausdorffEqualMinkowski}
  If $\lmd(\SC) > 0$ and if there exists a measure on $\XX$ that is critical for $\SC$, then
  \cref{def:CriticalMeasure} and
  \cref{lem:FrostmanLemmaEasy} show that for all $s \in (0,\lmd(\SC))$
  we have $s\le \hausd(\SC)$.
  Hence $\lmd(\SC)\le \hausd(\SC)$.
  In combination with \cref{eq:PEHausdorffDominatedByPELowerMinkowski},
  this then implies
  \[
    \hausd(\SC) = \lmd(\SC)
    .
  \]
  If $\lmd(\SC)=0$, then $0\le \hausd(\SC) \le  \lmd(\SC)=0$ by
  \cref{eq:PEHausdorffDominatedByPELowerMinkowski}, and so $\hausd(\SC)=\lmd(\SC)$.
\end{remark}

\subsection{Main result}%
\label{sub:MainResult}

The main result of the paper
shows the existence of a critical probability measure for any subset $\SC$
of a quasi-Banach space that satisfies some mild regularity condition
(namely, {\ConditionZetaText}).
Precisely, we show the following:

\begin{theorem}\label{thm:MainResultCriticalMeasureExistence}
  Let $\XX$ be a quasi-Banach space.
  Let $\emptyset\ne \SC \subset \XX$ satisfy {\ConditionZetaText}.
  Then there exists a Borel probability measure $\mu$ on $\XX$ that is critical
  for $\SC$ with $\om{\mu}(\XX\setminus\SC)=0$ and $\om{\mu}(\SC)=1$.
  In particular, the power-exponential Hausdorff dimension of $\SC$
  coincides with the power-exponential lower Minkowski dimension,
  \[
    \hausd(\SC) = \lmd(\SC)
    .
  \]
\end{theorem}

\begin{proof}
  See \cref{equivalence_lower_minkowski_existence_of_measure} and its proof.
  The "in particular" part follows from \cref{rem:CriticalMeasuresImpliesHausdorffEqualMinkowski}.
\end{proof}

The theorem thus implies that for any set satisfying some mild regularity conditions,
the power-exponential Hausdorff dimension --- which is often difficult to determine ---
coincides with the much more tractable power-exponential lower Minkowski dimension.
The latter is based on the asymptotic growth of the covering numbers,
which is well understood for many infinite-dimensional sets of interest;
see for instance the examples in \cref{sub:ImplicationsForClassicalSpaces}.

We expect the construction of the measure, which is described in detail in
\cref{sec:MeasureConstruction}, to be of independent interest.

\subsection{Consequences regarding phase transitions for lossy compression
and (non-linear) approximation}%
\label{sub:PhaseTransitionConsequences}

In addition to simplifying the computation of the power-exponential Hausdorff dimension,
critical measures are of interest for lossy compression and in approximation theory,
where they allow to quantify the size of subsets of "well approximable" elements.
As we will explain in this section, critical measures give rise to a \emph{phase transition}:
Roughly speaking, every compression/approximation rate below a critical threshold is achievable,
whereas the elements that can be encoded/approximated at a rate
above the threshold form a null set with respect to the critical measure.
We remark that this general idea already appeared in \cite{grohsPhaseTransitionsRate2023}.
Our main contribution here is to systematically determine general conditions
under which results as in \cite{grohsPhaseTransitionsRate2023} hold,
whereas \cite{grohsPhaseTransitionsRate2023} focused on the specific setting
of Besov- or Sobolev spaces as subsets of $L^2(\Omega)$.

Our discussion is structured as follows:
In \cref{sub:IntroRateDistortionTheory}, we recall the findings
from \cite{grohsPhaseTransitionsRate2023} in our notation and our level of generality,
and analyze how the quantities studied there
relate to the quantities considered in the present paper,
such as the power-exponential Minkowski dimension.
In \cref{sub:IntroNonlinearApproximation}, we introduce
a general criterion pertaining to methods of non-linear approximation.
We then show that whenever this criterion is satisfied,
a phase transition occurs for the considered approximation method.
This criterion in particular applies to the setting of approximation
by neural networks which was studied in \cite[Theorem~3]{grohsPhaseTransitionsRate2023},
but also to the setting of non-linear $n$-term approximation using a dictionary.

\subsubsection{Critical measures and lossy compression}
\label{sub:IntroRateDistortionTheory}

The goal in lossy compression is to encode the elements $x$ of a "signal class" $\SC$ using $n$ bits,
in such a way that the error ("distortion") incurred by first encoding and then decoding
a signal is as small as possible.
The "size" of the signal class is then reflected in the decay of this distortion,
as the number $n$ of bits tends to $\infty$.

Formally, let $(\XX, d)$ be a quasi-metric space, and let $\SC \subset \XX$.
An \emph{$n$-bit codec} (or an \emph{$n$-bit encoder-decoder pair}) for $\SC$ is a pair
$(E_n, D_n)$ consisting of the encoder $E_n : \SC \to \{ 0,1 \}^n$ and the decoder
$D_n : \{ 0,1 \}^n \to \XX$.
The \emph{(worst-case) distortion} of $(E_n, D_n)$ is defined as
\[
  \delta (E_n, D_n)
  := \sup_{x \in \SC} d\bigl(x, D_n(E_n(x))\bigr)
  \in [0,\infty]
  .
\]
We note that $\delta(E_n,D_n) < \infty$ if and only if $\SC$ is bounded.

A \emph{codec sequence} for $\SC$ is a sequence $( (E_n,D_n) )_{n \in \N}$,
where $(E_n, D_n)$ is an $n$-bit codec for $\SC$.
We say that such a codec sequence \emph{achieves compression rate $\CR \in [0,\infty)$},
if there exists $C > 0$ such that
\[
  \delta (E_n, D_n)
  \leq C \cdot n^{-\CR}
  \qquad \forall \, n \in \N
  .
\]
Finally, we define the \emph{optimal compression rate} of $\SC \subset \XX$ as
\begin{equation}
  \CR^\ast (\SC, \XX)
  := \sup
     \bigl\{
       \CR \in [0,\infty)
       \,\,:\,\,
       \text{$\exists$ codec sequence for $\SC$ achieving compression rate $\CR$}
     \bigr\}
  .
  \label{eq:OptimalCompressionRate}
\end{equation}

As we will now show in \cref{lem:CompressionRateVSMinkowskiDimension},
the optimal compression rate of $\SC$ is closely related to the
power-exponential Minkowski dimension of $\SC$.
This result is probably folklore, but since we could not locate a reference, we provide the proof. 

\begin{lemma}\label{lem:CompressionRateVSMinkowskiDimension}
  Let $(\XX, d)$ be a quasi-metric space and let $\emptyset \neq \SC \subset \XX$.
  Then
  \[
    \CR^\ast (\SC, \XX)
    = \bigl(\umd(\SC)\bigr)^{-1}
    .
  \]
  Here, $\SC$ is equipped with the quasi-metric $d$,
  and we use the interpretations $\infty^{-1} = 0$ and $0^{-1} = \infty$.
\end{lemma}
\begin{proof}
 See \cref{sec:proof_of_lem:CompressionRateVSMinkowskiDimension}.
\end{proof}

By definition, no codec sequence can achieve a distortion of $\CalO(n^{-\CR^\ast (\SC, \XX)-\eps})$ for some $\eps>0$, uniformly over all elements of $\SC$.
However, for a given codec sequence $( (E_n, D_n) )_{n \in \N}$,
there might be elements $x \in \SC$ that can be encoded and recovered at a strictly better rate.
For instance, the unit ball of $W^{k,\infty} (\Omega)$ also contains (a multiple of)
the unit ball of $W^{k+1, \infty} (\Omega)$, and these more regular functions can be encoded
and recovered at a faster rate, given a suitable codec sequence.
It is thus of interest to quantify the "size" of the set of elements that can be
encoded and recovered at a "better than optimal" rate.
As we will see, if there exists a critical measure $\mu$ on $\XX$
that is critical for $\SC$ and if $\SC$ is sufficiently
``nice'' (meaning that $\md(\SC)$ exists),
then this set of elements will be a $\mu$-null set.

To make this precise, given a fixed codec sequence $\mathscr{S} = ( (E_n, D_n) )_{n \in \N}$,
and any $\CR \in [0,\infty)$, we define the set of elements that can be encoded and recovered
at rate $\CalO(n^{-\CR})$ as
\begin{equation}
  \mathcal{R}^\CR (\mathscr{S}, \SC, \XX)
  := \Bigl\{
       x \in \SC
       \,\,:\,\,
       \sup_{n \in \N}
       \Bigl(
         n^\CR \cdot d\bigl(x, D_n(E_n(x))\bigr)
       \Bigr)
       < \infty
     \Bigr\}
  .
  \label{eq:ElementsWithGivenCompressionRate}
\end{equation}

We then have the following result, which is a (minor) generalization of
\cite[Theorem~4]{grohsPhaseTransitionsRate2023} to the present setting
of subsets of quasi-metric spaces, instead of the subsets of Banach spaces
considered in \cite{grohsPhaseTransitionsRate2023}.

\begin{theorem}\label{thm:RateDistortionTheoryPhaseTransition}
  Let $(\XX, d)$ be a quasi-metric space.
  Let $\emptyset \ne \SC \subset \XX$ be bounded.
  Then, writing $\CR^\ast := \bigl(\umd(\SC)\bigr)^{-1}$, the following holds:
  \begin{enumerate}[label=(\roman*)]
    \item \emph{(A suitable codec sequence achieves every rate below $\CR^\ast$)}
          There exists a codec sequence ${\mathscr{S}^\ast = ( (E_n^\ast, D_n^\ast) )_{n \in \N}}$
          for $\SC$ such that for every $0 \leq \CR < \CR^\ast$,
          there exists $C(\CR) > 0$ satisfying
          \[
            \delta (E_n^\ast, D_n^\ast) \leq C(\CR) \cdot n^{-\CR}
            \qquad \forall \, n \in \N
            .
          \]
          In particular, $\mathcal{R}^\CR (\mathscr{S}^\ast, \SC, \XX) = \SC$.
  \end{enumerate}
  Now, assume that $\mu$ is a probability measure on $\XX$ that is critical for $\SC$,
  and set
  \[
    \CR^\sharp
    := \bigl(\lmd(\SC)\bigr)^{-1}
    .
  \]
  Then the following hold:
  \begin{enumerate}[resume,label=(\roman*)]
    \item \emph{(quantitative bound relating approximation error and number of bits)}
          Let $\CR \in (\CR^\sharp, \infty)$.
          Then there exist $c = c(\CR) > 0$ and $\eps_0 = \eps_0 (\CR) > 0$
          such that for every $n \in \N$ and every $n$-bit codec $(E_n, D_n)$ for $\SC$,
          we have
          \[
            \mu^\ast
            \Bigl(
              \bigl\{
                x \in \SC
                \,\,:\,\,
                d\bigl(x, D_n(E_n(x))\bigr) \leq \eps
              \bigr\}
            \Bigr)
            \leq 2^{n - c \cdot \eps^{-1/\CR}}
            \qquad \forall \, \eps \in (0, \eps_0)
            .
          \]

    \item \emph{(any rate above $\CR^\sharp$ is only achievable on a $\mu$-null set)}
          For every $t \in (\CR^\sharp, \infty)$ and every codec sequence
          $\mathscr{S} = \bigl( (E_n, D_n) \bigr)_{n \in \N}$,
          the set $\mathcal{R}^\CR (\mathscr{S}, \SC, \XX)$ is a $\mu$-null set:
          \[
            \mu^\ast \Bigl(\mathcal{R}^\CR (\mathscr{S}, \SC, \XX)\Bigr)
            = 0
            \quad \text{ for every $\CR\in (\CR^{\sharp},\infty)$.}
          \]
          In fact, even
          \[
            \mu^\ast
            \biggl(\,\,
              \bigcup_{\CR \in (\CR^\sharp, \infty)}
                \mathcal{R}^\CR (\mathscr{S}, \SC, \XX)
            \biggr)
            = 0
            .
          \]
  \end{enumerate}
\end{theorem}

\begin{remark}\label{rem:SharpPhaseTransition}
  The most interesting case of the above theorem occurs if the power-exponential Minkowski dimension
  of $\SC$ exists, meaning that $\CR^\ast = \CR^{\sharp}$.
  In this case, one gets a \emph{sharp phase transition}, meaning that every rate
  \emph{below} $\CR^\ast$ is achievable on all of $\SC$, but every rate above $\CR^\ast$
  can only be achieved on a $\mu$-null set.
  
  It should be noted that the theorem makes no claim for the case $\CR = \CR^\ast$.
\end{remark}

\begin{proof}
  See \cref{sec:proof_of_thm:RateDistortionTheoryPhaseTransition}.
  We remark that the proof is essentially the same as in \cite{grohsPhaseTransitionsRate2023},
  so no real originality is claimed.
  The (relatively short) proof is provided to make the paper self-contained.
  The differences of our result compared to the one in \cite{grohsPhaseTransitionsRate2023}
  are the following:
  \begin{itemize}
      \item We handle the setting of \emph{quasi}-metric spaces instead of Banach spaces;

      \item We clearly distinguish between the upper and the lower Minkowski dimension.
            In \cite{grohsPhaseTransitionsRate2023}, a critical measure is required
            to ``match'' the rate $(\CR^{\ast})^{-1}= \umd(\SC)$. 
            In view of 
            \cref{cor:SmallBallCondition_Dominated_by_LowerMinkowskiDim},
            this implies that
            $\umd(\SC)\le \lmd(\SC)$, so that $\md(\SC)$ exists.
            For an example of a compact convex set for which the upper and lower
            generalized Minkowski dimension differ, see \cref{sec:example_lmd<umd}.

      \item We show that one ``ultimate'' codec sequence $\mathscr{S}^{\ast}$ in (i)
            can be chosen that attains any rate $\CR\in [0,\CR^{\ast})$.
            In contrast, in \cite{grohsPhaseTransitionsRate2023},
            a potentially different codec sequence is used for every $\CR\in [0,\CR^{\ast})$.
            \qedhere
  \end{itemize}
\end{proof}

\subsubsection{Critical measures and (non-linear) approximation}
\label{sub:IntroNonlinearApproximation}

Given a subset $\SC$ of a (quasi)-Banach space $\XX$, approximation theory is concerned
with approximating the elements of $\SC$ via "elementary objects" $g \in \CalA_n$
of a given complexity $n \in \N$,
and with determining the decay of the approximation error as $n \to \infty$.
Examples include:
\begin{enumerate}[label=(\roman*)]
  \item Approximating by polynomials (or trigonometric polynomials) of degree at most $n \in \N$.
        This is a case of \emph{linear} approximation, in the sense that
        the sets $\CalA_n$, $n \in \N$ in this case form a nested family of subspaces.
        A fundamental result related to this form of approximation is \emph{Jackson's inequality}, which
        states that $2\pi \Z^d$-periodic functions $f \in C^r_{2 \pi} (\R^d)$
        can be approximated by trigonometric polynomials
        up to error $\CalO(n^{-r})$ in $L^\infty$; see for instance
        \cite[Chapter~6, Section~3, Theorem~6]{LorentzApproximationOfFunctions}.
        Here, it should be noted that the space of trigonometric polynomials of degree at most $n$
        in $d$ variables has dimension $\CalO(n^d)$,
        meaning that $\CalO(n^{-r}) = \CalO\bigl( [\dim(\CalA_n)]^{-r/d}]\bigr)$.

  \item Approximating by neural networks of "complexity" at most $n \in \N$.
        Here, the "complexity" could refer to several quantities, such as
        the width and/or the depth of the networks,
        the number of non-zero network weights,
        but also to the magnitude of the network coefficients.

        This is a case of \emph{non-linear approximation}, since the "neural network sets"
        $\CalA_n = \mathcal{NN}_n$ are not subspaces;
        see \cite{PetersenEtAlTopologicalPropertiesOfNNSets}.
        For quantitative approximation results using neural networks, we refer to
        \cite{MhaskarNNsForOptimalApproximation,YarotskyErrorBounds,YarotskyPhaseDiagram,
        PetersenVoigtlaenderOptimalApproximationUsingNNs,
        ShenYangZhangOptimalApproximationRateWidthDepth,schneiderNonlocalTechniquesAnalysis2026,
        NNApproximationSobolev}
        and the references therein.

  \item Approximating by linear combinations of at most $n$ elements
        of a given system ${\CalG \!=\! (g_i)_{i \in I}}$.
        This is also called \emph{non-linear $n$-term approximation} in the literature;
        see, e.g., \cite{DahmenCompressedSensing}.
        Hence, in this case we would essentially have
        \begin{equation}
          \CalA_n
          = \Sigma_n (\CalG)
          := \bigcup_{I_0 \subset I, \, \# I_0 \leq n}
               \mathrm{span} \bigl\{ g_i \,\,:\,\, i \in I_0 \bigr\}
          .
          \label{eq:NonlinearNTermApproximation}
        \end{equation}
        This is again a case of non-linear approximation (i.e., the sets $\Sigma_n (\CalG)$
        do not form subspaces), which is particularly well studied
        due to its relation with compressive sensing
        and the design of "optimally adapted dictionaries"; see \cite{DahmenCompressedSensing}.
\end{enumerate}

In many such cases, one chooses the sets $\CalA_n$ such that each element
$g \in \CalA_n$ can be described by $\CalO(n)$ real parameters
(plus potentially $\CalO(n)$ discrete parameters).
This is for instance the case for $\CalA_n = \Sigma_n (\CalG)$
and for $\CalA_n = \mathcal{NN}_n$, if one requires that networks in $\mathcal{NN}_n$
have at most $n$ non-zero weights.
For $\CalA_n = \Sigma_n (\CalG)$, the $n$ discrete parameters specify the set $I_0:= \{i_1,\dots,i_n\}$
in \Cref{eq:NonlinearNTermApproximation}, and the $n$ real parameters specify the coefficients
$c_i$ of the linear combination $\sum_{i\in I_0} c_i \, g_i$.
For $\CalA_n := \mathcal{NN}_n$, the $3n$ discrete parameters describe the ``positions''
of the non-zero weights (which two neurons in which layer are connected),
and the $n$ real parameters encode the values of the $n$ non-zero network weights.

In such a case, one could expect for ``nice'' sets $\SC$ that the (optimal) approximation rate
scales like
\begin{equation}
  \sup_{x \in \SC} \, \dist(x, \CalA_n)
  \asymp n^{-\AR^\ast}
  ,
  \label{eq:ApproximationWithFamiliesTargetStatement}
\end{equation}
where $\AR^\ast = \AR^\ast (\SC, \XX)$ is as in \Cref{eq:OptimalCompressionRate}.
In the remainder of this section, we will discuss a wide class of "well-behaved"
families of sets $\CalA_n$ where this is indeed true (up to fine-print).
Crucially, however, such a statement cannot be expected to hold for \emph{completely general}
families of sets $\CalA_n$ in which $\CalA_n$ can be described by $n$ real parameters,
because of phenomena such as space-filling curves that allow to encode $n$ real parameters
using just one real parameter.
We briefly discuss several examples that illustrate what can go wrong,
even in seemingly natural cases.
\begin{itemize}
  \item For non-linear $n$-term approximation, if one chooses $\CalG = (g_\ell)_{\ell \in \N}$
        to be a countable \emph{dense} subset of the "ambient space" $\XX$,
        it is easy to see that $\dist(x, \CalA_n) \leq \dist(x, \CalA_1) = 0$ for all $x \in \XX$,
        meaning that a statement as in \Cref{eq:ApproximationWithFamiliesTargetStatement}
        does not hold.

        The pathology here is that one needs to use elements "far out" in the dictionary
        to approximate well.
        We will rule out such a behavior by imposing a condition of "polynomial search depth";
        see \Cref{prop:PolynomiallyBoundedNTermApproximationControlledComplexity} below.

    \item For approximation by neural networks,
          \cite[Theorem 4]{MaiorovPinkusLowerBoundsForMLPApproximation}
          shows, based on the Kol\-mo\-go\-rov--Arnold representation theorem,
          that there exists a (pathological but) smooth, sigmoidal activation function
          such that \emph{every} continuous function on $[0,1]^d$
          can be approximated up to \emph{arbitrarily small error} using neural networks
          \emph{of a fixed size}.
          In this case, the issue is that the weights and biases used for the approximation
          grow in an uncontrolled fashion.

  \item Even for the widely used ReLU activation function,
        a statement as in \Cref{eq:ApproximationWithFamiliesTargetStatement}
        fails, if one does not impose further assumptions.
        Indeed, the phenomenon of ``superconvergence'' for ReLU networks, discovered in \cite{pmlr-v75-yarotsky18a}
        and further refined in
        \cite{YarotskyPhaseDiagram,ShenYangZhangOptimalApproximationRateWidthDepth},
        states, roughly speaking, 
        that one can approximate functions $f \in C^r ([0,1]^d)$ up to error $\CalO(n^{-2r/d})$
        (instead of the expected $\CalO(n^{-r/d})$) using ReLU networks with $n$ non-zero weights,
        if one neither restricts the depth nor the weight magnitude of the approximating networks.
        It was recently shown in \cite[Lemma~3.4]{OuBoelcskeiCoveringNumbersForDeepReLUNNs}
        for the case of Lipschitz functions in dimension one
        that this even holds for networks with all weights in $[-1,1]$,
        meaning that the phenomenon is driven by the network depth,
        rather than the magnitude of the network weights.

        In this case, the issue is that the $n$ real parameters in $[-1,1]$
        defining the network have to be discretized "extremely finely" in order
        for the resulting discretized network to be close to the original one.
\end{itemize}

Being properly warned of what can go wrong, we now introduce a large class of
approximating families that are "well-behaved".

\begin{definition}\label{def:ControlledComplexity}
  Let $(\XX, d)$ be a quasi-metric space, and for each $n \in \N$
  let $\emptyset \neq \CalA_n \subset \XX$.
  We say that the family $(\CalA_n)_{n \in \N}$ is \emph{of controlled complexity in $\XX$}, if
  for every $\alpha > 0$, every $x \in \X$ and every $R > 0$,
  there exist $n_0 \in \N$ and a function $\fcc : \N \to [0,\infty)$ such that
  \begin{equation}
    \forall \, \delta > 0 : \qquad \sup_{n \in \N} \, \frac{\fcc(n)}{n^{1 + \delta}} < \infty
    \label{eq:ZetaGrowthCondition}
  \end{equation}
  and
  \begin{equation}
    \cn\bigl(\ball(x,R) \cap \CalA_n, \,\, n^{-\alpha}\bigr)
    \leq \exp(\fcc(n))
    \qquad \forall \, n \in \N_{\geq n_0}
    .
    \label{eq:ControlledComplexity}
  \end{equation}
\end{definition}

\begin{remark*}
  For showing that a given family is of controlled complexity, one usually chooses
  \[
    \fcc (n) = C \cdot n \cdot (\ln(e n))^\mu
  \]
  for some $\mu \ge 0$.
  It is straightforward to see that this satisfies Condition \eqref{eq:ZetaGrowthCondition}.
\end{remark*}

We now show that for families of controlled complexity,
our target statement \eqref{eq:ApproximationWithFamiliesTargetStatement} essentially holds,
at least for the case where the power-exponential Minkowski dimension of $\SC$
exists and satisfies $\md(\SC) \in (0,\infty)$.

\begin{theorem}\label{thm:PhaseTransitionForFamiliesWithControlledComplexity}
  Let $(\XX, d)$ be a quasi-metric space and let $\emptyset \neq \SC \subset \XX$ be bounded.
  Then, writing $\AR^\ast := (\umd(\SC))^{-1}$, the following holds:
  \begin{enumerate}[label=(\roman*)]
    \item (A suitable family of controlled complexity achieves every rate below $\AR^\ast$)
          There exists a family $(\CalA_n^\ast)_{n \in \N}$ of subsets
          $\emptyset \neq \CalA_n^\ast \subset \XX$
          of controlled complexity in $\XX$ such that for every $0 \leq \AR < \AR^\ast$,
          there exists $C(\AR) > 0$ satisfying
          \[
            \sup_{x \in \SC} \,\, \dist(x, \CalA_n^\ast) \leq C(\AR) \cdot n^{-\AR}
            \qquad \forall \, n \in \N.
          \]

    \item (No family of controlled complexity can achieve a rate above $\AR^\ast$
          uniformly over $\SC$)
          Let $(\CalA_n)_{n \in \N}$ be a family of controlled complexity in $\XX$.
          If
          \begin{equation}
            \sup_{x \in \SC} \,\, \dist(x, \CalA_n) \leq C(\AR) \cdot n^{-\AR}
            \qquad \forall \, n \in \N
            \label{eq:FamilyApproximationRate}
          \end{equation}
          holds for some $\AR \geq 0$ and some constant $C(\AR) > 0$, then $\AR \leq \AR^\ast$.
  \end{enumerate}
  Now, set $\AR^\sharp := \bigl( \lmd (\SC) \bigr)^{-1}$ and
  assume that $\mu$ is a probability measure on $\XX$ that is critical for $\SC$.
  Then the following holds:
  \begin{enumerate}[label=(\roman*),resume]
    \item (Any rate above $\AR^\sharp$ is only achievable on a $\mu$-null set)
          For every $\AR > \AR^\sharp$ and every family $(\CalA_n)_{n \in \N}$
          of subsets $\emptyset \neq \CalA_n \subset \XX$ of controlled complexity in $\XX$,
          the set
          \[
            \CalR^\AR \bigl( (\CalA_n)_{n \in \N}, \SC, \XX\bigr)
            := \Bigl\{
                 x \in \SC
                 \,\,:\,\,
                 \sup_{n \in \N}
                 \bigl(
                   n^\AR \cdot \dist (x, \CalA_n)
                 \bigr)
                 < \infty
               \Bigr\}
          \]
          is a $\mu$-null set, i.e.,
          $\mu^\ast \bigl( \CalR^\AR \bigl( (\CalA_n)_{n \in \N}, \SC, \XX\bigr) \bigr) = 0$.
          In fact, we even have
          \[
            \mu^\ast
            \biggl(\,\,
              \bigcup_{\AR \in (\AR^\sharp, \infty)}
                \CalR^\AR \bigl( (\CalA_n)_{n \in \N}, \SC, \XX\bigr)
            \biggr)
            = 0
            .
          \]
  \end{enumerate}
\end{theorem}

\begin{remark*}
  The most interesting case of the theorem occurs if $\md(\SC) \in (0,\infty)$ exists.
  In that case, we get a true phase transition, since then $\AR^\ast = \AR^{\sharp}$.
\end{remark*}

\begin{proof}
  See \Cref{sec:PhaseTransitionForControlledComplexityFamiliesProof}.
\end{proof}

In view of the above result, the following definition makes sense.

\begin{definition}\label{def:OptimalFamilies}
  Let $(\XX, d)$ be a quasi-metric space, let $\emptyset \neq \SC \subset \XX$ be bounded,
  and set $\AR^\ast := (\umd(\SC))^{-1}$.
  A family $(\CalA_n)_{n \in \N}$ of subsets $\emptyset \neq \CalA_n \subset \XX$ of controlled complexity
  is called \emph{optimal for $\SC$ in the class of families of controlled complexity in $\XX$},
  if \eqref{eq:FamilyApproximationRate} holds for all $0 \leq \AR < \AR^\ast$.
\end{definition}

To indicate the generality of the notion of a family of controlled complexity
(see \Cref{def:ControlledComplexity}) and thus the generality and utility of
\Cref{thm:PhaseTransitionForFamiliesWithControlledComplexity}, in the following we show
that the family $\CalA_n = \CalNN_n (g, \R^d)$ corresponding to neural network approximation
(with suitable restrictions on the number of weights, the magnitude of the individual weights,
and the network depth)
and the family $\CalA_n = \Sigma_n (\CalG, P)$ corresponding to non-linear $n$-term
approximation using the dictionary $\CalG$ (with ``polynomially bounded search depth'')
are both of controlled complexity.
Furthermore, for each case, we identify natural sets $\SC$ for which the given family
is optimal in the sense of \cref{def:OptimalFamilies}.

We start with the case of neural networks.

\begin{definition}\label{def:NeuralNetworkSet}
  The weights of a neural network are a tuple of the form
  \[
    \CalW = \bigl( (A_1, b_1), \dots, (A_L, b_L)\bigr)
  \]
  of matrices $A_\ell \in \R^{N_\ell \times N_{\ell-1}}$ and bias vectors $b_\ell \in \R^{N_\ell}$,
  where $N_\ell \in \N$ specifies the width of the $\ell$-th layer of the network.
  We define the \emph{input- and output dimensions} of $\CalW$ as
  $d_{\mathrm{in}}(\CalW) := N_0$ and $d_{\mathrm{out}}(\CalW) := N_L$.

  The associated \emph{neural network function} using the ReLU activation function
  $\varrho : \R \to \R$, $\varrho(x) = \max \{ 0, x \}$ is given by
  \[
    R_\varrho (\CalW) : \quad \R^{d_{\mathrm{in}}(\CalW)} \to \R^{d_{\mathrm{out}} (\CalW)},
    \quad
    R_\varrho (\CalW)
    := T_L \circ (\varrho \circ T_{L-1}) \circ \cdots \circ (\varrho \circ T_1)
    ,
  \]
  where $T_\ell \, x = A_\ell \, x + b_\ell$, and where $\varrho$ acts componentwise on vectors,
  i.e.,
  \[
    \varrho \bigl( (x_1,\dots,x_m)\bigr)
    = \bigl(\varrho(x_1), \dots, \varrho(x_m)\bigr)
    .
  \]
  In this paper, we measure the ``complexity'' of
  $\CalW$ in terms of the \emph{number of non-zero weights}
  \[
      \norm{\CalW}_{\ell^0}
    := \sum_{j=1}^{L} \bigl(\| A_j \|_{\ell^0} + \| b_j \|_{\ell^0}\bigr)
    ,
  \]
  in terms of the \emph{network depth} $L(\CalW) := L$,
  and in terms of the \emph{weight magnitude}
  \[
      \| \CalW \|_{\ell^{\infty}}
    := \max_{1 \leq j \leq L} \bigl(\| A_j \|_{\ell^\infty} + \| b_j \|_{\ell^\infty}\bigr)
    .
  \]
  Here, for a matrix (or vector) $A$, the quantity $\| A \|_{\ell^0}$
  counts the total number of non-zero entries of $A$,
  and $\| A \|_{\ell^\infty} := \max_{i,j} |A_{i,j}|$.

  Now, given $d,g \in \N$, we define for $n \in \N$ the set of ReLU networks $\R^d \to \R$
  with at most $n$ non-zero weights, weight magnitude "of moderate growth",
  and number of layers growing at most poly-logarithmically with $n$ as
  \[
    \CalNN_n (g, \R^d)
    := \Bigl\{
         R_\varrho (\CalW)
         \,\,:\,\,
         \begin{array}{l}
           \CalW \text{ neural network weights with }
           d_{\mathrm{in}}(\CalW) = d,
           d_{\mathrm{out}}(\CalW) = 1, \\
           \norm{\CalW}_{\ell^0} \leq n, \,\,
           \norm{\CalW}_{\ell^{\infty}} \leq g \cdot n^{g},
           \text{ and }
           L(\CalW) \leq (\ln(e n))^g
         \end{array}
       \Bigr\}
    .
  \]
\end{definition}

\begin{proposition}\label{prop:NNApproximationOfControlledComplexity}
  Let $d,g \in \N$, let $\Omega \subset \R^d$ be measurable and bounded, and let $p \in (0,\infty]$.
  Then the sequence $(\CalA_n)_{n \in \N}$ of the sets $\CalA_n := \CalNN_n (g, \R^d)$
  introduced in \Cref{def:NeuralNetworkSet} and considered as subsets of $\XX:=L^p(\Omega)$,
  is of controlled complexity in $L^p(\Omega)$.
\end{proposition}

\begin{proof}
  See \Cref{sec:NNControlledComplexityProofs}.
\end{proof}

The neural network sets defined above are optimal for the approximation of functions in
$C^r ([0,1]^d)$, as shown in the following result.
We remark that very similar observations are already contained in
\cite{ElbraechterDNNApproximationTheory} and \cite{grohsPhaseTransitionsRate2023}.
We state the result here mainly to show that it can be easily treated by our methods,
and since we do get a phase transition in $L^p$, which is not part of the results shown in
\cite{ElbraechterDNNApproximationTheory}.
A similar phase transition is contained in \cite{grohsPhaseTransitionsRate2023}
for the case when the approximation error is measured in $L^2$,
whereas we can cover the case of general $p \in [1,\infty]$.

\begin{proposition}\label{prop:NNsOptimalForCr}
  Let $d \in \N$, $p \in [1,\infty]$, $m \in \N_0$ and $\alpha \in (0,1]$ and set $r := m + \alpha$.
  We say that $f : [0,1]^d \to \R$ belongs to $C^{m,\alpha}([0,1]^d)$
  if $f$ is $m$ times continuously differentiable with
  \[
    \| f \|_{C^{m,\alpha}}
    := \max
       \Bigl\{ 
         \| f \|_{L^\infty},
         \quad
         \max_{\beta \in \N_0^d, |\beta| = m} \,\,
           \sup_{x,y \in [0,1]^d, x \neq y} \,\,
             \frac{|\partial^\beta f (x) - \partial^\beta f(y)|}{\norm{x-y}^\alpha}
       \Bigr\}
    < \infty
    .
  \]
  Then, as a subset of $\XX = L^p ([0,1]^d)$, the set
  \[
    \SC
    := \bigl\{ f \in C^{m,\alpha}([0,1]^d) \,\,:\,\, \| f \|_{C^{m,\alpha}} \leq 1 \bigr\}
  \]
  satisfies $\md(\SC) = \frac{d}{r}$,
  there exists a Borel probability measure on $\XX = L^p ([0,1]^d)$
  that is critical for $\SC$,
  and the family of neural networks
  \(
    (\CalA_n)_{n \in \N} = \bigl(\CalNN_n (3,\R^d)\bigr)_{n \in \N}
  \)
  defined in \Cref{prop:NNApproximationOfControlledComplexity} (with $g = 3$)
  is optimal for $\SC$ in the class of families of controlled complexity in $L^p([0,1]^d)$.
  In particular, a phase transition as in
  \Cref{thm:PhaseTransitionForFamiliesWithControlledComplexity} occurs.
\end{proposition}

\begin{proof}
  See \Cref{sec:NNsOptimalForCrProof}.
\end{proof}

Next, we consider the setting of non-linear $n$-term approximation with
polynomially bounded search depth.
The precise definition and result are as follows.

\begin{proposition}\label{prop:PolynomiallyBoundedNTermApproximationControlledComplexity}
  Let $\XX$ be a quasi-normed space and $\CalG = \{g_n\}_{n \in \N} \subseteq \XX$ arbitrary.
  Moreover, let $P \in \R[X]$ be a polynomial.
  For $n \in \N$, define
  \[
    \CalA_n
    := \Sigma_n(\CalG, P)
    := \bigcup_{\substack{I_0 \subseteq \{k \in \N: k \leq P(n)\}, \\ \# I_0 \leq n}}
         \linspan \bigl\{ g_k \,\,:\,\, k \in I_0 \bigr\}
  \]
  to be the set of $n$-term linear combinations of $\CalG$ with search depth bounded by $P$.
  Here, we use the convention $\linspan \emptyset = \{0\}$.
  Then the family $(\CalA_n)_{n \in \N}$ is of controlled complexity in $\XX$.
\end{proposition}

\begin{proof}
  See \Cref{sec:NonlinearNTermControlledComplexityProofs}.
\end{proof}

As our final result in this section, we show that our results can --- in some cases --- also
be applied to obtain critical measures for sets that do \emph{not} satisfy \ConditionZetaText.
This can be done by taking a critical measure
for a set satisfying {\ConditionZetaText} and ``pushing it forward'' via a ``sub-H{\"o}lder'' map,
see \cref{prop:transfer_principle}.
Using this technique, we will construct a critical measure for the class of certain
$C^2$-regular subsets of $\R^2$.
These sets have been intensely studied in harmonic analysis in the context of the
optimal approximation and encoding of so-called \emph{cartoon-like functions}
\cite{CandesCurvelets,VoigtlaenderPeinAnalysisVSSynthesisSparsityAlphaShearlets,
GrohsCartoonApproximation,KutyniokCompactlySupportedShearlets}.
The precise class of sets that we will consider is defined as follows:

\begin{definition}[{cf.\ \cite[Section~2.2]{GrohsCartoonApproximation}
  and \cite[Definition~6.1]{VoigtlaenderPeinAnalysisVSSynthesisSparsityAlphaShearlets}}]
  \label{def:StarShapedSets}
  Let $\nu \geq 1$ and fix ${0 < \varrho_0 \leq \frac{1}{4} < \frac{1}{2} \leq \varrho_1 < 1}$.
  Then the family of \emph{star-shaped subsets of $\R^2$ with $C^{2}$ smooth boundary}
  is defined as
  \[
    \mathrm{STAR}^{2} (\nu; \varrho_0, \varrho_1)
    := \biggl\{
         x_0 + S_{\varrho}
         \,\,:\,\,
         \begin{array}{l}
           x_0 \in \R^2 \text{ and } \varrho \in C^{2}_{2 \pi}(\R)
           \text{ with } \| \varrho'' \|_{L^\infty} \leq \nu , \\
           \varrho_0 \leq \varrho \leq \varrho_1 ,
           \text{ and } x_0 + S_{\varrho} \subset [0,1]^2
         \end{array}
       \biggr\}
    ,
  \]
  where $C_{2\pi}^{2}(\R)$ denotes the set of all $2 \pi$-periodic $C^2$ functions
  $\varrho : \R \to \R$, and where
  \begin{equation}
    S_{\varrho}
    := \biggl\{
         r \cdot \begin{pmatrix} \cos(\phi) \\ \sin(\phi) \end{pmatrix}
         \,\,:\,\,
         \phi \in [0,2\pi] \text{ and } 0 \leq r \leq \varrho(\phi)
       \biggr\}
    \subset \R^2
    .
    \label{eq:StarShapedSet}
  \end{equation}
\end{definition}

By combining our results with existing results on non-linear $n$-term approximation using shearlets,
we prove the following result.

\begin{proposition}\label{prop:ShearletsOptimalForC2Boundary}
  Let $\nu \geq 1$ and fix $0 < \varrho_0 \leq \frac{1}{4} < \frac{1}{2} \leq \varrho_1 < 1$.
  Then, the set
  \[
    \SC
    := \bigl\{ \indicator_{S} \,\,:\,\, S \in \mathrm{STAR}^2 (\nu; \varrho_0, \varrho_1) \bigr\}
    ,
  \]
  considered as a subset of $\XX = L^2 (\R^2)$, satisfies $\md(\SC) = 1$ and
  there exists a Borel probability measure on $\XX$ that is critical for $\SC$.

  Moreover, there exist a suitable \emph{cone-adapted shearlet system}
  (see \cite[Definition~5.6]{VoigtlaenderPeinAnalysisVSSynthesisSparsityAlphaShearlets}),
  a suitable enumeration $\CalG = (g_n)_{n \in \N}$ of this system,
  and a polynomial $P$ such that the family
  $(\CalA_n)_{n \in \N} = \bigl(\Sigma_n (\CalG, P)\bigr)_{n \in \N}$
  defined in \Cref{prop:PolynomiallyBoundedNTermApproximationControlledComplexity}
  is optimal for $\SC$ in the class of families of controlled complexity in $L^2(\RR^2)$.
  In particular, a phase transition as in
  \Cref{thm:PhaseTransitionForFamiliesWithControlledComplexity} occurs.
\end{proposition}

\begin{remark*}
    It is well-known that non-linear $n$-term approximation via shearlets
    achieves a rate of $\CalO(n^{-1})$ up to log factors for $C^2$ cartoon-like functions,
    and that this is the optimal rate that can be achieved by non-linear $n$-term approximation
    using \emph{any} dictionary, if one imposes a polynomial search depth restriction
    \cite{Grohs2015,GrohsCartoonApproximation,KutyniokCompactlySupportedShearlets,
    VoigtlaenderPeinAnalysisVSSynthesisSparsityAlphaShearlets}.
    The only genuinely new result that we prove is the phase transition. 
    Moreover, we show optimality in the class of all approximation schemes of controlled complexity.

    The considered function class $\SC$ is a subset of the set of \emph{$C^2$ cartoon-like functions}
    that is typically considered in the literature
    \cite{CandesCurvelets,VoigtlaenderPeinAnalysisVSSynthesisSparsityAlphaShearlets,
    GrohsCartoonApproximation,KutyniokCompactlySupportedShearlets}.
    Since shearlets achieve the same order of approximation $\CalO(n^{-1})$ up to log factors
    for the full set, and since every Borel probability measure $\mu$ on $L^2(\R^2)$
    that satisfies the small-ball condition with parameter $\sigma$ and
    $\mu^\ast(\SC) > 0$ also satisfies
    \[
      \mu^\ast \Bigl( \{ \text{all $C^2$ cartoon-like functions} \} \Bigr)
      > 0
      ,
    \]
    the result as stated for $\SC$ is in fact \emph{stronger}
    than it would be if it was stated for the class of $C^2$ cartoon-like functions.
\end{remark*}

\begin{proof}[Proof of \cref{prop:ShearletsOptimalForC2Boundary}]
  See \Cref{sec:ShearletProof}.
\end{proof}

\subsection{Implications for Besov- and Triebel-Lizorkin spaces}
\label{sub:ImplicationsForClassicalSpaces}

In this section, we provide a large class of sets for which critical measures do exist:
unit balls in Besov and Triebel-Lizorkin spaces that are compactly embedded
into other Besov or Triebel-Lizorkin spaces.
We consider the full range of parameters (except for $p=\infty$ in the $F$-case)
and give results in the isotropic and in the dominating mixed smoothness case.
In particular, these unit balls serve as valid signal/target classes
to which the phase transition results stated in \cref{thm:RateDistortionTheoryPhaseTransition}
and \cref{thm:PhaseTransitionForFamiliesWithControlledComplexity} apply.

Let $\Omega \subset \RR^d$ be open and bounded.
For $s\in \RR$, $0<p,q\le \infty$ denote by $B^{s}_{p,q}(\Omega)$
the Besov space on $\Omega$ and by $F^{s}_{p,q}(\Omega)$ (with $p<\infty$)
the Triebel-Lizorkin space on $\Omega$. 
For the precise definition of these spaces,
we refer to \cite[Section 1.11]{triebelTheoryFunctionSpaces2006}.

Our main result on function spaces of isotropic smoothness is the following
\cref{existence_crit_measure_isotropic}.
\cite{grohsPhaseTransitionsRate2023} already contains the case where $\Omega$
is a bounded Lipschitz domain, and the parameters are given by $s_2=0$, $p_2=q_2=2$
paired either with $A_1=B$ or with $(A_1)^{s_1}_{p_1,q_1} = W^{k,p}$.
Our result generalizes theirs to arbitrary bounded open subsets and the full range of parameters.
Its proof is deferred to \cref{sec_proof_of_existence_crit_measure_isotropic}.

\begin{theorem} \label{existence_crit_measure_isotropic}
  Let $\emptyset\ne \Omega \subset \RR^d$ be an open and bounded subset.
  Let $A_1,A_2\in \{B,F\}$.
  Let $s_1,s_2\in \RR$ and let $0<p_1,q_1,p_2,q_2 \le \infty$ (with $p_i<\infty$ if $A_i=F$).
  Assume that 
  \begin{equation*}
     s_1 > s_2 
     \quad \text{and} \quad
      s_1 - \frac{d}{p_1} > s_2 -\frac{d}{p_2}.
  \end{equation*}
  Let $\SC$ denote the unit ball of $(A_1)^{s_1}_{p_1,q_1}(\Omega)$
  and let $\XX = (A_2)^{s_2}_{p_2,q_2}(\Omega)$,
  endowed with the \mbox{(quasi-)}\allowbreak metric induced by
  $\norm{\cdot\mid (A_2)^{s_2}_{p_2,q_2}(\Omega)}$.
  Then $\md(\SC) = \frac{d}{s_1-s_2}$ and there exists a Borel probability measure on
  $\XX$ that is critical for $\SC$.
\end{theorem}

For an arbitrary bounded domain%
\footnote{In the sense of \cite[Section 3.1]{vybiralFunctionSpacesDominating2006}.}
$\Omega \subset \RR^d$, let $S^{s}_{p,q}B(\Omega)$ and $S^{s}_{p,q}F(\Omega)$ (with $p<\infty$)
denote the Besov and Triebel-Lizorkin spaces of dominating mixed smoothness.
For the precise definition of these spaces,
we refer to \cite[Section 3.1]{vybiralFunctionSpacesDominating2006}.

Our main result on function spaces of dominating mixed smoothness is the following
\cref{existence_crit_measure_dms}.
Its proof is deferred to \cref{sec_proof_of_existence_crit_measure_dms}.

\begin{theorem}\label{existence_crit_measure_dms}
Let $\emptyset \ne \Omega \subset \RR^d$ be an arbitrary bounded domain.
Let $A_1,A_2\in \{B,F\}$.
Let $s_1,s_2\in \RR$ and let $0<p_1,q_1,p_2,q_2 \le \infty$ (with $p_i<\infty$ if $A_i=F$).
Assume that 
\begin{equation*}
   s_1 > s_2 
   \quad \text{and} \quad
    s_1 - \frac{1}{p_1} > s_2 -\frac{1}{p_2}.
\end{equation*}
Let $\SC$ denote the unit ball of $S^{s_1}_{p_1,q_1}A_1(\Omega)$
and let $\XX= S^{s_2}_{p_2,q_2}A_2(\Omega)$, endowed with the \mbox{(quasi-)}\allowbreak metric
induced by $\norm{\cdot\mid S^{s_2}_{p_2,q_2}A_2(\Omega)}$.
Then $\md(\SC)=\frac{1}{s_1-s_2}$ and there exists a Borel probability measure
on $\XX$ that is critical for $\SC$.
\end{theorem}

\begin{remark}
Of course, in both \cref{existence_crit_measure_isotropic} and \cref{existence_crit_measure_dms},
the generalized Minkowski dimension of $\SC$ was already known.
Our contribution is the existence of the critical probability measure for $\SC$.
\end{remark}

\subsection{Related work}%
\label{sub:RelatedWork}

This section discusses the relation between the present paper
and earlier works.
The discussion is split into different subsections by topic.

\subsubsection{Previous results on critical measures}

The present paper was heavily inspired by \cite{grohsPhaseTransitionsRate2023},
which introduced the notion of a critical measure and proved the existence
of such measures for unit balls of Sobolev- and Besov spaces,
considered as subsets of $L^2 (\Omega)$, for the whole range of parameters
for which these smoothness spaces embed compactly into $L^2$.

The construction in \cite{grohsPhaseTransitionsRate2023} is meticulously tailored
to the specific setting of Sobolev- and Besov spaces as subsets of $L^2$,
which allows to construct a critical measure in the setting of certain sequence spaces,
and then transfer this measure to the actual function spaces using wavelets.
In contrast, the results in the present paper establish the existence of critical measures
in much larger generality, assuming only a mild topological condition (see \cref{def:condition_zeta})
and a quite general condition on the covering numbers of the ``signal set'' $\SC$,
namely\footnote{
The condition $\lmd(\SC) = \umd(\SC)$ is not strictly speaking necessary for our results to hold,
but it ensures that the notion of a critical measure from \Cref{def:CriticalMeasure}
agrees with the one in \cite{grohsPhaseTransitionsRate2023} and that one gets a single
sharp threshold; see also the discussion following \Cref{def:CriticalMeasure}.
}
that $\lmd(\SC) = \umd(\SC)$.
In addition to this much greater generality on the ``allowed signal classes'' $\SC$,
we emphasize that our approach allows for a general quasi-Banach space $\XX$
as the ``ambient space'', whereas \cite{grohsPhaseTransitionsRate2023} only
constructs critical measures for the case where the ambient space is $L^2(\Omega)$.

\subsubsection{Results regarding the ``size'' of the set of difficult-to-approximate functions}

One focus in classical approximation theory is to analyze the asymptotic decay
of the \emph{worst-case} (or \emph{minimax}) approximation error, given by
\[
  \delta_n (\SC, \CalA)_{\XX}
  := \sup_{f \in \SC} \,\,
       \inf_{g \in \CalA_n} \,\,
         \| f - g \|_{\XX}
  ,
\]
where $\SC$ is the function class of interest, $\CalA = (\CalA_n)_{n \in \N}$
is a family of approximating sets of increasing ``complexity'',
and the space $\XX$ is used to measure the error.
In many cases, one obtains a decay behavior of $\delta_n (\SC, \CalA)_{\XX} \asymp n^{-t^\ast}$
for a certain $t^\ast = t^\ast (\SC; \XX)$; see for instance
\cite{LorentzApproximationOfFunctions,DahmenCompressedSensing,
CandesCurvelets,GrohsCartoonApproximation,VoigtlaenderPeinAnalysisVSSynthesisSparsityAlphaShearlets,
KutyniokCompactlySupportedShearlets,MhaskarNNsForOptimalApproximation,YarotskyErrorBounds,
ElbraechterDNNApproximationTheory, YarotskyPhaseDiagram,PetersenVoigtlaenderOptimalApproximationUsingNNs,
ShenYangZhangOptimalApproximationRateWidthDepth,schneiderNonlocalTechniquesAnalysis2026,NNApproximationSobolev}.
For a more detailed discussion of most of these results,
see \Cref{sub:IntroNonlinearApproximation}.

However, knowing that $\delta_n (\SC, \CalA)_{\XX} \asymp n^{-t^\ast}$ only implies
that there is \emph{at least} one function that is ``difficult to approximate'',
which still leaves open the possibility that ``most'' functions $f \in \SC$
can be approximated at a rate $n^{-t}$ with $t > t^\ast$.
In order to obtain a more precise understanding of the difficulty of approximating
functions $f \in \SC$, it is thus of interest to quantify the ``size''
of the set of functions in $\SC$ that are ``difficult to approximate''.

This was in fact the main motivation for the paper \cite{grohsPhaseTransitionsRate2023},
which showed that if
\begin{itemize}
  \item $\SC \subset L^2 (\Omega)$ is the unit ball of a Sobolev-
        or Besov space that is compactly embedded in $L^2$,
  \item $\mu$ is a critical measure for $\SC$,
  \item $t^\ast = 1 / \md(\SC)$, and
  \item $\CalA_n$ is the set of neural networks with at most $n$
        non-zero weights, and \emph{all weights suitably quantized},
\end{itemize}
then the following phase transition occurs:
For each $0 < t < t^\ast$, one has $\delta_n (\SC, \CalA)_{L^2} \lesssim n^{-t}$,
but for each $t > t^\ast$, the set of functions $f \in \SC$ for which
$\dist(f, \CalA_n) \lesssim n^{-t}$ forms a $\mu$-null set;
see \cite[Theorem~3 and Remark~3]{grohsPhaseTransitionsRate2023}.

The proof in \cite{grohsPhaseTransitionsRate2023} proceeds by first establishing
a similar phase transition for the setting of lossy compression
of the elements $x \in \SC$ using $n$ bits
(see \Cref{sub:IntroRateDistortionTheory} for more details),
which is then used to derive the result for approximation by quantized neural networks.
The present paper generalizes the results in \cite{grohsPhaseTransitionsRate2023}
regarding lossy compression (see \Cref{sub:IntroRateDistortionTheory})
and moreover introduces the concept of an
\emph{approximating family $\CalA = (\CalA_n)_{n \in \N}$ of controlled complexity},
showing for each such family that a phase transition similar to the one
in \cite{grohsPhaseTransitionsRate2023} occurs; see \Cref{sub:IntroNonlinearApproximation}.
This new notion not only covers the setting of quantized neural networks
considered in \cite{grohsPhaseTransitionsRate2023}, but also
non-quantized neural networks with bounds on the network depth and the magnitude of the weights,
as well as non-linear approximation (with polynomially bounded search depth)
using arbitrary dictionaries.

We mention that the earlier works \cite{zbMATH01333083,zbMATH07161892}
also analyzed the ``size'' of difficult to approximate functions 
(in less generality).
These papers provided important inspiration for \cite{grohsPhaseTransitionsRate2023}
and the present work.

The idea of probabilistic lower bounds for approximation ---
in the sense that one equips the set of ``target functions'' with a probability measure
and studies the measure of the set of functions that are approximated well ---
has also been employed in \cite{KURKOVA201734,KURKOVA2023654}
to derive lower bounds for approximation by neural networks,
albeit in the setting where the input set $\Omega \subset \R^d$ on which the error
is measured is a finite set, so that also the ``target class'' of functions to
be approximated is finite.

\subsubsection{Approximation rates and power-exponential Minkowski dimension}

Our results in \Cref{sub:IntroNonlinearApproximation}
regarding approximating families of controlled complexity
in particular imply that one cannot achieve approximation error
$\CalO(n^{-t})$ for $t > 1 / \lmd(\SC)$ for the cases of
\emph{(i)} non-linear approximation (with polynomial search depth) using arbitrary dictionaries, and
\emph{(ii)} approximation by neural networks, under certain assumptions on the
network depth and the magnitude of the network weights.

While the notion of families of controlled complexity is new,
the specific results \emph{(i)} and \emph{(ii)} are well-known;
they already appear in \cite{Grohs2015} and \cite{bolcskeiOptimalApproximationSparsely2019},
respectively.
See also \cite{ElbraechterDNNApproximationTheory} and the references therein.

The fact that suitable sets of neural networks form families of controlled complexity
relies on known bounds for the entropy numbers of such neural networks,
which we cite from \cite{GrohsVoigtlaenderTheoryToPracticeGapDeepLearning}.
Related bounds can be found in \cite{OuBoelcskeiCoveringNumbersForDeepReLUNNs,SiegelXuEntropyNumbers}.

\subsubsection{Generalized Hausdorff and Minkowski dimensions and the Frostman lemma}

The covering numbers of finite-dimensional sets $A \subset \R^d$ typically satisfy
the polynomial scaling behavior $\cn (A, \eps) \asymp \eps^{-s}$.
Moreover, the polynomial scale $r \mapsto r^s$ is foundational for the construction
of the Hausdorff measure $\CalH^s$ on $\R^d$.
For subsets of infinite-dimensional spaces, however, the ``power-exponential''
scaling behavior $\cn(A, \eps) \asymp \exp(C \cdot \eps^{-s})$ is more common,
which is why we focus especially on this type of scaling through the power-exponential
Minkowski dimension and the power-exponential Hausdorff measure and Hausdorff dimension.

These are special cases of the general notion of \emph{gauge functions} considered
in the field of geometric measure theory, where each such gauge function (or \emph{scale})
defines a corresponding notion of Minkowski dimension,
Hausdorff measure, and Hausdorff dimension; see
\cite{mattilaGeometrySetsMeasures1995,kloecknerGeneralizationHausdorffDimension2012}
and the references therein.

Key results known to hold in the setting of general gauge functions are
that the Minkowski dimension for a given scale is an upper bound for the Hausdorff dimension
of the same scale (see \cite[Proposition~3.2]{kloecknerGeneralizationHausdorffDimension2012}),
and the general \emph{Frostman lemma},
see \cite[Proposition~3.1 and Theorem~F]{fan2026infinitedimensionalmultifractalanalysis}.
The latter in particular implies that for a compact metric space $\SC$ and a gauge function $f$,
the positivity of the Hausdorff measure $\CalH^{f}(\SC)>0$ implies that there exists
a Borel probability measure $\mu$ on $\SC$ and some constant $C>0$
such that $\mu$ satisfies the ``small-ball condition''
\[
  \mu(\ball(x,r)) \leq C \cdot f(6r)
  \qquad \forall \, x \in \SC \text{ and } r > 0
  .
\]
While the Frostman lemma implies --- under the above assumptions ---
the existence of a measure satisfying a small-ball condition,
the emphasis of our results is different:
The Frostman lemma imposes the measure theoretic condition $\CalH^{f}(\SC) > 0$,
whereas the goal of the present paper is to identify topological conditions on the
set $\SC$ (namely, \ConditionZetaText) under which one can show the existence
of a measure that satisfies a small-ball condition matching the
growth behavior of the covering numbers of $\SC$.
This then implies a posteriori (see \Cref{thm:MainResultCriticalMeasureExistence})
that the lower power-exponential Minkowski dimension agrees
with the power-exponential Hausdorff dimension of the set,
but makes no claim about whether the corresponding Hausdorff measure
of the set is positive at the critical exponent.

We remark that \cite{kloecknerGeneralizationHausdorffDimension2012}
deduces bounds for the Hausdorff dimension of certain specific sets
(so-called ``Hilbert cubes'') by explicitly constructing
a ``Frostman measure'', i.e., a measure satisfying a suitable type of
small-ball condition; see
\cite[proofs of Propositions 4.2 and 4.3]{kloecknerGeneralizationHausdorffDimension2012}.

Another paper concerned with different notions of dimension
on infinite-dimensional sets, with respect to general scalings
(i.e., families of gauge functions) is \cite{HelfterScales}.
The considered dimensions include a version of the Minkowski dimension
(see \cite[Definition~1.5]{HelfterScales}),
the Hausdorff dimension (see \cite[Definition~2.13]{HelfterScales}),
and certain ``local scales'' and ``quantization scales''
(see \cite[Definitions 1.7 and 1.8]{HelfterScales}) associated to a given Borel measure
on the considered metric space.
The main results of \cite{HelfterScales} (Theorems~A and C) state several inequalities
between these different notions of dimension.
We remark that the power-exponential scale that we focus on
is studied closely in \cite{HelfterScales}; see \cite[Example~1.1]{HelfterScales}.

In particular, \cite[Theorem~E]{HelfterScales} shows for the unit ball $\SC^{k,\alpha}_d$
of the space $C^{k,\alpha}([0,1]^d)$, considered as a subset of $C([0,1]^d)$ that the
power-exponential Minkowski dimension and Hausdorff dimension of $\SC^{k,\alpha}_d$
are both given by $\frac{d}{k + \alpha}$, which is part of our \Cref{prop:NNsOptimalForCr}
(for the special case $p = \infty$).
What is not explicitly stated in \cite{HelfterScales} is the existence of a critical measure
that is shown in \Cref{prop:NNsOptimalForCr}.
However, an inspection of the proof in \cite{HelfterScales} shows that
it proceeds by \emph{(i)} considering a ``Cantor-type set'' $Z = \prod_{n=1}^\infty Z_n$
which is a product of finite, discrete sets $Z_n$, equipped with a suitable metric,
for which it is then shown (in our terminology) that the product measure of the uniform measures
is critical (see \cite[Proposition~4.1 and Equation~(4.3)]{HelfterScales}),
and \emph{(ii)} showing that the space $Z$ can be embedded via an expanding map into $\SC^{k,\alpha}_d$;
see \cite[Lemma~4.7]{HelfterScales}.
The construction of this expanding map, however, is tailored to the special set
$\SC^{k,\alpha}_d$; in fact, it is based on dilating and shifting a fixed function,
which is similar in spirit to the wavelet-based construction in \cite{grohsPhaseTransitionsRate2023}.
In particular, the results in \cite{HelfterScales} do not imply the general
existence result for critical measures that we prove.

Finally, one might wonder whether one could simply take a (general) Hausdorff measure
itself as a measure satisfying the ``small-ball condition''.
In general, this is not possible, as for instance implied by the results
in \cite{boardmanHausdorffMeasureProperties1973}.
That paper shows for a certain compact space $\SC$ that
${\CalH^f (\SC) \in \{ 0, \infty \}}$ for every gauge function $f$,
whereas we are interested in \emph{finite, positive} measures satisfying a small-ball condition.
We note that the pathological behavior of the set $\SC$ mentioned above cannot be explained
by it being ``too large''; in fact, 
\cite[Theorem 1.8]{kloecknerGeneralizationHausdorffDimension2012}
shows that the set considered in \cite{boardmanHausdorffMeasureProperties1973}
has power-exponential Hausdorff dimension equal to $1$.
Overall, this shows that Hausdorff measures may not always play the role of
a maximally spread out measure, even for sets with finite
power-exponential Hausdorff dimension.

\subsubsection{Existence of critical measures for the polynomial scale}

In the finite-dimensional setting and for the polynomial scale $r \mapsto r^s$, $s > 0$,
critical measures are known to exist for a wide range of ``regular'' sets.
In fact, two-sided versions of \cref{eq:SmallBallConditionPrecise},
with a polynomial instead of a power-exponential scaling,
hold for Hausdorff measures restricted to compact and convex sets
\cite[Example 12.7]{grafFoundationsQuantizationProbability2000}, to surfaces
of compact and convex sets \cite[Example 12.8]{grafFoundationsQuantizationProbability2000},
to compact differentiable manifolds \cite[Example 12.9]{grafFoundationsQuantizationProbability2000},
and to self-similar sets satisfying the open set condition
\cite[Example 12.10]{grafFoundationsQuantizationProbability2000}, and, more generally,
self-similar sets satisfying the weak separation property in their affine hull
\cite[Theorem~2.1]{fraserAssouadDimensionSelfsimilar2015}.
These sets are then also called \emph{Ahlfors regular}.
Here, we note that if $\mu(\SC) = 1$ and
\[
  \mu(\ball(x,r)) \asymp r^s
  ,
\]
then this implies that the Minkowski dimension of $\SC$ (with respect to the polynomial scale)
is $s$, so that $\mu$ is a critical measure (with respect to the polynomial scale).

\subsubsection{Small-ball probabilities for Gaussian measures and metric entropy}

Measures satisfying a small-ball condition are also studied in probability theory,
where particularly the small-ball behavior of Gaussian measures has received attention
\cite{KL93,LL99,LS01}.
To describe a major result in this direction, let $\nu$ be a centered Gaussian measure
on a separable Banach space $\XX$,
and let $B_\nu$ denote the unit ball of its reproducing-kernel Hilbert space $H_{\nu}$,
viewed as a subset of $\XX$
(see e.g.\ \cite[Lemma~2.1]{KuelbsRKHSForGaussianMeasure} for the definition).
The results of \cite{KL93,LL99} relate the asymptotic behavior of
$-\log \nu\bigl(\ball(0,r)\bigr)$ as $r\downarrow 0$ to the metric entropy
of $B_\nu$; see also the survey \cite{LS01}.

Despite the similarity in the studied notions (small-ball conditions and metric entropy),
this setting considerably differs from ours:
For the Gaussian measures, the measure $\nu$ is prescribed and its small-ball behavior
is inferred from the metric entropy properties of the associated unit ball $B_\nu$.
In contrast, in our results the set $\SC \subset \XX$ is prescribed, and we construct a Borel measure
$\mu$ on $\X$ that is concentrated on $\SC$ and satisfies a suitable small-ball property.
We also note that our sets $\SC$ are always bounded, whereas the support of a non-trivial
Gaussian measure is always unbounded;
moreover, the reproducing-kernel Hilbert space $H_\nu$ always satisfies
$\nu(H_\nu) = 0$ except if $H_\nu$ is finite-dimensional;
see \cite[Theorem~2.4.7]{BogachevGaussianMeasures}
or \cite[Proposition~4.45]{hairer2023introductionstochasticpdes}.
Finally, \cite[Lemma~3.3]{Goodman1988} shows that the covering numbers of the ball $B_\nu$ always
satisfy $\log_2 (\cn(B_\nu, \eps)) = o(\eps^{-2})$, which is not always satisfied
for the sets $\SC$ that we consider.

\subsection{Notation}%
\label{sub:Notation}

We introduce some additional notation to complement that already introduced in this introduction.
Let $\N = \{ 1,2,\dots \}$ denote the set of natural numbers and let $\N_0 := \N \cup\{0\}$.
We also use the notation $\N_{\geq k} := \{ n \in \N \,\,:\,\, n \geq k \}$.
The cardinality of a set $A$ will be denoted by $\num{A} \in \N_0 \cup \{ \infty \}$.
For a function $f : \XX \to \YY$ and a subset $A \subset \XX$, we denote by $f|_{A} : A \to \YY$
the restriction of $f$ to $A$.
For a set $A \subset \XX$, we sometimes consider the indicator function associated to $A$,
given by
\[
  \indicator_A (x)
  := \begin{cases}
       1, & \text{if } x \in A, \\
       0, & \text{otherwise}.
     \end{cases}
\]
Given a measure space $(\XX, \CalA, \mu)$, a measurable space $(\YY, \CalB)$,
and a measurable map $\Phi : \XX \to \YY$, we define the push-forward of $\mu$ by $\Phi$ as
\[
  \Phi_{\#} \mu : \quad
  \CalB \to [0,\infty], \quad
  (\Phi_{\#} \mu) (B) := \mu(\Phi^{-1}(B))
  .
\]

For a multi-index $\beta = (\beta_1,\dots, \beta_d) \in \NN_0^d$,
let $\abs{\beta} = \sum_{i=1}^d \beta_i$, and let
$\partial^{\beta} := \partial_{x_1}^{\beta_1} \dots \partial_{x_d}^{\beta_d}$.
We will denote by $\log_2$ the logarithm to the base $2$ and by $\ln$ the natural logarithm.

For non-negative quantities $a,b$ that depend on several other objects,
we write $a \lesssim b$ to mean that $a \leq C \cdot b$,
where the implied constant $C > 0$ only depends on objects that are considered fixed
and is independent of objects that are considered variable;
the context will make clear (or explicitly state) the details.
Similarly, $a \gtrsim b$ means that $b \lesssim a$, and $a \asymp b$ means that $a \lesssim b$
and $b \lesssim a$.

Let $(\XX,\varrho)$ be a quasi-metric space. 
We denote by $\ball\bigl(x,r \mid (\XX,\varrho)\bigr) := \{ y \in \XX : \varrho(y,x)\le r\}$
the ``closed ball'' centered at $x\in \XX$ with radius $r>0$.
Depending on context, we will sometimes simply write $\ball(x,r \mid \XX)$ or $\ball(x,r)$.
If $(\XX,\norm{\cdot})$ is a quasi-normed space, we will also write $\ball(x,r \mid \norm{\cdot})$.
For a non-empty subset $\CalA \subset \XX$ of a quasi-metric space $(\XX, \varrho)$,
we write
\[
  \dist (x, \CalA)
  := \inf_{a \in \CalA} \varrho(x,a)
  , \qquad \text{for } x \in \XX
  .
\]

Given two quasi-normed spaces $(\XX, \| \cdot \|_{\XX})$ and $(\YY, \| \cdot \|_{\YY})$,
we write $\XX \hookrightarrow \YY$ to mean that $\XX$ continuously embeds into $\YY$;
this means that $\XX \subset \YY$ and that there exists a constant $C > 0$ with
$\| x \|_{\YY} \leq C \cdot \| x \|_{\XX}$ for all $x \in \XX$.
We then also write $\iota : \XX \hookrightarrow \YY$ for the embedding map $\iota (x) = x$.
If the embedding is compact, this will be explicitly stated/assumed.
Finally, for a subset $\SC \subset \XX$ (not necessarily a subspace),
we sometimes write $\SC \hookrightarrow \XX$ for the inclusion map.

Let $\emptyset \ne \XX$ be a set and let $d_1,d_2$ be two quasi-metrics on $\XX$. 
We say that $d_1$ and $d_2$ are \emph{bi-Lipschitz equivalent},
if there exists a constant $C\ge 1$, called a \emph{bi-Lipschitz constant}, such that
\begin{equation*}
    \frac{1}{C} \cdot d_1(x,y) \le d_2(x,y) \le C \cdot d_1(x,y) 
    \quad \text{ for all $x,y\in \XX$.}
\end{equation*}

Let $(\XX,\varrho)$ be a quasi-metric space, let $\emptyset \ne \SC\subset \XX$, and let $\eps > 0$.
We say that a subset $P\subset \SC$ is $\eps$-\emph{separated in $\SC$},
or an $\eps$-\emph{packing in $\SC$}, if
\begin{equation*}
    \metr(x,x') >\eps
    \qquad 
    \text{for all $x,x' \in P$ with $x\ne x'$.}
\end{equation*}
We call a subset $N\subset \SC$ an $\eps$-\emph{net for $\SC$},
if for every $x\in \SC$ there exists $y \in N$ with $\varrho(x,y)\le \eps$.
We call an $\eps$-separated subset $P\subset \SC$ maximal,
if for every $x\in \SC\setminus P$, the set $P\cup \{x\}$ is not $\eps$-separated.
Maximal $\eps$-separated subsets always exist, see \cref{maximal_separated_subsets_exist},
and are $\eps$-nets, see \cref{maximal_separated_subsets_are_nets}.

We define the $\eps$-\emph{covering number} of a non-empty set $\SC \neq \emptyset$ as
\begin{equation*}
  \cn(\SC,\eps)
  := \inf \bigg\{
            n \in \NN
            \,\,:\,\,
            \begin{array}{c}
              \text{there exists an $\eps$-net $\SC_{\eps}\subset \SC$}\\
              \text{with $\num{\SC_{\eps}}=n$}
            \end{array}
          \bigg\}
  ,
\end{equation*}
where, by convention, we let $\cn(\SC,\eps):=\infty$, if there is no finite $\eps$-net for $\SC$.
Moreover, we set $\cn(\emptyset, \eps) := 0$ for every $\eps > 0$.

In words, $\cn(\SC,\eps)$ is the smallest natural number $n\in \NN$
such that $\SC \neq \emptyset$ can be covered by $n$ closed balls of radius $\eps$
\emph{with centers in $\SC$}.
For $\emptyset \neq \SC \subset \XX$, we also define the
\emph{external $\eps$-covering number of $\SC$ in $\XX$},
denoted by $\extcn(\SC,\XX,\eps)$ to be the smallest number $n\in \NN$
such that $\SC$ can be covered with $n$ balls of radius $\eps$ \emph{with centers in $\XX$};
if no such $n \in \N$ exists, define $\extcn(\SC, \XX, \eps) := \infty$.
Moreover, set $\extcn(\emptyset, \XX, \eps) := 0$.
It is clear that $\extcn(\SC,\XX,\eps) \le \cn(\SC,\eps)$.
For a bound in the other direction, see \cref{lem:InternalVSExternalCoveringNumbers}.

\subsection{Structure of the remainder of the paper}%
\label{sub:PaperStructure}

\cref{sec:ExistenceOfCriticalMeasures} is dedicated to proving the existence
of critical measures for sets satisfying {\ConditionZetaText}.
We begin with \cref{sec:MeasureConstruction}, which details the construction of a spread-out measure
using packing sets and establishes our main technical tool, \cref{measure_SC_has_growth_order}.
Building on this, \cref{sub:CriticalMeasureExistence} proves our main existence result,
\cref{equivalence_lower_minkowski_existence_of_measure},
which immediately yields \cref{thm:MainResultCriticalMeasureExistence}.
\cref{sub:TransferringCriticalMeasures} concludes the section by establishing
two additional technical tools:
A transfer result for transferring measures satisfying a small-ball condition via ``sub-H{\"o}lder'' maps,
and a clarification of the role of the ``ambient space'' $\XX$
in our definition (\Cref{def:CriticalMeasure}) of a critical measure.

\cref{sec:sets_satisfying_cond_zeta} investigates the generality of {\ConditionZetaText}.
\cref{sec:convex_bounded_and_complete_are_zeta} shows that every convex, bounded, and complete subset
of a quasi-normed space satisfies {\ConditionZetaText},
while \cref{sec:unit_balls_are_zeta} extends this to (continuously embedded) unit balls
of quasi-Banach spaces.
\cref{sec_unit_Balls_B_F_are_zeta} applies these observations to Besov- and Triebel-Lizorkin spaces,
providing proofs of \cref{existence_crit_measure_isotropic,existence_crit_measure_dms}.

The appendix sections contain postponed technical proofs and auxiliary constructions.
The proofs for the lossy compression- and non-linear approximation results in \cref{sub:PhaseTransitionConsequences}
are provided in \cref{sec:TechnicalProofsIntroduction}.
\cref{sec:QuasiMetricTechnicalProofs}
contains proofs for several technical results regarding quasi-metric spaces
that are used in \Cref{sec:intro,sec:ExistenceOfCriticalMeasures}.
In \cref{sec:proofs_of_technical_results_on_function_spaces}, we collect
and prove several facts regarding the relation between
entropy numbers of linear maps and the power-exponential Minkowski dimension,
which are then used to prove technical results regarding
function spaces that are crucial for \Cref{sec_unit_Balls_B_F_are_zeta}.
Finally, \cref{sec:example_lmd<umd} presents a construction of a compact convex set
whose upper and lower power-exponential Minkowski dimensions strictly differ,
and \cref{sec:standard_technical_lemmas} collects standard technical lemmas.

\subsection{Usage of ``AI'' tools}

The core results in this paper were fully obtained before April 2026,
i.e., before ``AI'' tools based on large language models (LLMs)
showed the ability to prove substantial, genuinely novel results.

Consequently, the use of these tools was restricted to the following:
\begin{itemize}
  \item The construction of the compact convex set with non-existing
        power-exponential Minkowski dimension in \Cref{sec:example_lmd<umd}
        was suggested by Google Gemini and heavily simplified by us.

        Gemini did not provide any source for this construction,
        but when we later asked ChatGPT to suggest additional related work
        that we might have missed, it pointed us to the paper \cite{HelfterScales},
        whose Example~4.4 is roughly similar in spirit,
        with one decisive difference being that the set constructed in \cite{HelfterScales} is
        a Cartesian product of discrete sets, whereas our set is convex.

  \item Google Gemini was used as an OCR tool, i.e., for transcribing
        handwritten scanned proofs to LaTeX.

  \item LLM-based tools were used for language improvements.

  \item ChatGPT 5.6 Sol was used for a final proofreading of the paper.
        This included suggesting additional references
        (including \cite{HelfterScales,fan2026infinitedimensionalmultifractalanalysis}),
        all of which we carefully checked.
\end{itemize}

The authors have independently verified, validated, and rewritten all parts of the paper
influenced by LLM-generated material and take full responsibility for the mathematical content
of the paper.

%% file: parts/content.tex
\section{Existence of critical measures}
\label{sec:ExistenceOfCriticalMeasures}

This section proves the existence of critical measures for sets satisfying
\ConditionZetaText, see \Cref{equivalence_lower_minkowski_existence_of_measure}.
The main tools for the proof are \cref{measure_SC_n_is_well_def,measure_SC_has_growth_order}
which construct a Borel probability measure on $\XX$ that is concentrated on $\SC$
(in the sense that $\mu^\ast (\XX \setminus \SC) = 0$ and hence $\mu^\ast(\SC) = 1$)
and satisfies a certain small-ball condition.

The idea for constructing $\mu$ is as follows:
We start with a sequence of finite sets $\emptyset \neq \SC_n \subset \SC$,
where in most cases we will take $\SC_n$ to be a maximal $2^{-n}$-separated subset of $\SC$.
The naive idea for creating a ``spread-out measure'' would be to take
a random vector $Y_n$ to be uniformly chosen from $\SC_n$
(with the $Y_n$, $n \in \N$, being independent),
and then let $\mu$ be the law (i.e., the distribution) of
the ``random walk'' $X := \sum_{n=1}^{\infty} Y_n$.
The main issue with this approach is that the sum defining $X$ does not necessarily converge.
To circumvent this issue, we require that $\SC$ satisfies \ConditionZetaTextAlphaC{(\alpha,C_0)}
and then consider the ``damped'' random walk
\[
  X := \sum_{n=1}^{\infty} \frac{Y_n}{C_0 \cdot n^\alpha}
  ,
\]
which is then guaranteed to converge to an element of $\SC$.

\Cref{sec:MeasureConstruction} below shows that this construction indeed yields
a well-defined Borel measure on $\XX$ that satisfies
$\mu^\ast (\XX \setminus \SC) = 0$ and $\mu^\ast (\SC) = 1$.
Since $\SC$ is not assumed to be Borel measurable, these properties are non-trivial.
Moreover, \Cref{sec:MeasureConstruction} proves \cref{measure_SC_has_growth_order}, which
states that if the sets $\SC_n$ are $2^{-n}$-separated with a certain growth
of the cardinalities $\# \SC_n$, then the measure $\mu$ satisfies a certain
small-ball condition.
Roughly speaking, this is shown by choosing $n = n(r)$ suitably and then
bounding $\PP (X \in \ball(x,r))$ by conditioning on all $Y_m$ with $m \neq n$
and using that $Y_n$ is suitably spread out.
We also prove a kind of converse of this statement (which is probably folklore),
namely that if a probability measure $\mu$ on $\SC$ satisfies a certain small-ball condition,
then this implies a lower bound on the cardinality of any $\eps$-net for $\SC$.

In \Cref{sub:CriticalMeasureExistence}, we use \Cref{measure_SC_has_growth_order}
to prove the existence of a critical measure for any set $\SC \subset \XX$
satisfying \ConditionZetaText; see \Cref{equivalence_lower_minkowski_existence_of_measure}.

Finally, in \Cref{sub:TransferringCriticalMeasures},
we establish two technical properties of critical measures
that are frequently useful:
The first of these results is a ``transfer result'' which shows that if $\Phi : \SC \to \SC '$
satisfies the ``lower Hölder condition'' $\varrho' (\Phi(x), \Phi(x')) \gtrsim (\varrho(x, x'))^\beta$,
and if $\mu$ is a probability measure on $\SC$ satisfying a small-ball condition with parameter $s$,
then the push-forward $\Phi_{\#} \mu$ satisfies a small-ball condition with parameter $s / \beta$.
The second result clarifies the exact definition of a critical measure:
Our definition of a critical measure (\Cref{def:CriticalMeasure})
is formulated for subsets $\SC$ of an "ambient" (quasi)-metric space $\XX$.
\Cref{restriction_inherits_criticality} in \Cref{sub:TransferringCriticalMeasures} shows that if
one ``forgets'' about the ambient space $\XX$, the resulting ``intrinsic definition''
is equivalent to the original ``extrinsic notion'' from \Cref{def:CriticalMeasure}.

\subsection{Constructing a spread-out measure using packing sets}
\label{sec:MeasureConstruction}

The following \cref{definition_measure_SC_n} and \cref{measure_SC_n_is_well_def}
show that on sets satisfying {\ConditionZetaText},
the construction using the ``damped'' random walk outlined at the start of this section
indeed yields a well-defined probability measure $\mu$.

\begin{definition}
    \label{definition_measure_SC_n}
    Let $(\XX,\norm{\cdot})$ be a quasi-normed space.
    Let $\SC\subset \XX$ satisfy \ConditionZetaTextAlphaC{(\alpha,C_0)}.
    For every $n\in \NN$ let $\SC_n\subset \SC$ be a non-empty, finite subset.
    Then we define the measure
    \begin{equation*}
        \mu(A)
        := \mu\bigl[ (\SC_n)_{n\in \NN},\alpha,C_0\bigr](A)
        := \PP
           \biggl(\,\,
             \sum_{n=1}^{\infty}
               \frac{Y_n}{C_0 \cdot n^{\alpha}} \in A
           \biggr), 
        \quad A \in \Borel(\XX),
    \end{equation*}
    where the random vectors $(Y_n)_{n\in \NN}$ are independent
    and for every $n$ the random vector $Y_n$ is uniformly distributed on $\SC_n$.
\end{definition}

The following \cref{measure_SC_n_is_well_def} shows that the measure $\mu$
in \cref{definition_measure_SC_n} is well-defined.
Its proof is rather technical and is deferred to \cref{sec:proof_of_measure_SC_n_is_well_def}.

\begin{lemma}\label{measure_SC_n_is_well_def}
  Let $(\XX, \| \cdot \|)$ be a quasi-normed space,
  let $\emptyset \neq \SC \subset \XX$ satisfy {\ConditionZetaTextAlphaC{(\alpha,C_0)}} for certain $\alpha, C_0 > 0$.
  Finally, for each $n \in \N$, let $\emptyset \neq \SC_n \subset \SC$ be finite.\\
  Then the measure $\mu = \mu[(\SC_n)_{n \in \N}, \alpha, C_0]$
  introduced in \Cref{definition_measure_SC_n} is a well-defined Borel probability measure on $\XX$.
  Furthermore, we have $\mu^\ast (\XX \setminus \SC) = 0$ and $\om{\mu}(\SC)=1$.
\end{lemma}

The following \cref{measure_SC_has_growth_order} is our main technical result.
It shows that if the sets $\SC_n$ in \Cref{definition_measure_SC_n}
are suitably separated subsets of $\SC$, then the resulting distribution $\mu$
of the damped random walk will have its mass spread accordingly.

\begin{theorem}\label{measure_SC_has_growth_order}
    Let $(\XX,\norm{\cdot})$ be a quasi-normed space with modulus of concavity $C_2 \ge 1$.
    Assume that $\SC$ satisfies {\ConditionZetaTextAlphaC{(\alpha,C_0)}} for some constants $\alpha, C_0>0$.
    Let $s\ge0$, $t\in \RR$ (with $t>0$ if $s=0$), and let $c_1>0$. 
    Let $(\SC_n)_{n\in\NN}$ be a sequence of finite $2^{-n}$-separated subsets $\SC_n \subset \SC$,
    and let $n_0\in \NN$ such that
    \begin{equation}
        \num\SC_n \ge \exp\bigl( c_1 \cdot 2^{ns}n^{t}\bigr)
        \quad \text{for all $n\ge n_0$.}
        \label{eqn_assumption_lower_bound_packing_num}
    \end{equation}
    Then the measure $\mu$ given in \cref{definition_measure_SC_n} satisfies the following:
    There exist $r_0>0$ and $c>0$ such that
    \begin{equation}
        \om{\mu}\bigl(\ball(x,r)\bigr) 
        \le \exp\biggl( - c  
        \cdot \Bigl(\frac{1}{r}\Bigr)^{s} \cdot \Bigl( \log_2 \Bigl(\frac{1}{r}\Bigr)\Bigr)^{t-\alpha s}\biggr)
        \label{eqn_growth_of_mu}
    \end{equation}
    for all $r\in (0,r_0)$ and all $x\in \XX$.
\end{theorem}

\begin{proof}
    \emph{Step 1.}
    We first derive some preliminary results.
    \cref{definition_measure_SC_n} is applicable and \cref{measure_SC_n_is_well_def}
    implies that $\mu$ is well-defined.
    By \cref{thm:Aoki_Rolewicz}, there exist $p \in (0,1]$, a $p$-norm $\norm{\cdot}_{\ast}$
    and a constant $C_3\ge 1$ such that
    \begin{equation*}
        \ball\Bigl( x, \frac{r}{C_3} \mid \norm{\cdot}_{\ast}\Bigr)
        \subset \ball(x,r) \subset \ball\bigl(x,C_3 r \mid \norm{\cdot}_{\ast}\bigr) =: \ball^{\ast}(x,C_3r)
        \quad \text{for all $x\in \XX$, $r>0$.}
    \end{equation*}
    Since $d(x,y):= \norm{x-y}^p_{\ast}$ is a metric,
    $\ball^{\ast}(x,C_3 r)$ is closed in the topology induced by $\norm{\cdot}_{\ast}$.
    Since $\norm{\cdot}$ and $\norm{\cdot}_{\ast}$ are bi-Lipschitz equivalent,
    they induce the same topology and thus the same Borel sigma-algebras.
    It follows that $\ball^{\ast}(x,C_3 r) \in \Borel(\XX)$.

    \medskip{}
    \emph{Step 2:}
    We now investigate the growth behavior of $\mu$.
    Let $x\in \XX$ and let $r>0$ be arbitrary. 
    Furthermore, let $n\in \NN$ be arbitrary.
    We will make some suitable choice for $n$ and restrictions on the maximal admissible size of $r$ further below.
    Using that $\om{\mu}$ is monotone at (a) and that $\om{\mu}=\mu$ on Borel-measurable subsets at (b), we obtain
    \begin{equation}
        \label{eqn_1_11}
        \om{\mu}(\ball(x,r)) 
        \overset{(a)}{\le} \om{\mu}\bigl( \ball^{\ast}(x,C_3 r)\bigr)
        \overset{(b)}{=} \mu\bigl( \ball^{\ast}(x,C_3 r)\bigr).
    \end{equation}
    We compute
    \begin{align}
        \mu\bigl(\ball^{\ast}(x,C_3 r)\bigr)
       &= \PP\bigl( Y \in \ball^{\ast}(x,C_3 r)\bigr)\nonumber\\
       &= \PP
          \Biggl(
            \frac{Y_n}{C_0\cdot n^{\alpha}}
            + \sum_{\substack{m=1\\m\ne n}}^{\infty} \frac{Y_m}{C_0\cdot m^{\alpha}}
            \in \ball^{\ast}(x,C_3 r)
          \Biggr)\nonumber\\
       &= \int \one_{\ball^{\ast}(x,C_3 r)}\Bigl(\frac{Y_n}{C_0\cdot n^{\alpha}} + Z_n\Bigr) \dif \PP,
       \label{eqn_1_1}
    \end{align}
    where we define $Z_n:= \sum_{\substack{m=1\\m\ne n}}^{\infty} \frac{Y_m}{C_0\cdot m^{\alpha}}$.
    Notice that by an argument as in the proof of \cref{measure_SC_n_is_well_def},
    $Z_n$ is well-defined and Borel measurable.
    More precisely, $Z_n$ is measurable with respect to the Borel $\sigma$-algebra
    on the co-domain and the sigma-algebra generated by $(Y_m)_{m\in \NN\setminus\{n\}}$
    on the domain.
    Since the family $(Y_m)_{m\in\NN}$ is independent,
    we have that $Y_n$ is independent of $(Y_m)_{m\in \NN\setminus\{n\}}$,
    see \cite[Corollary 4.7]{kallenbergFoundationsModernProbability2021}.
    It follows that $Y_n$ is independent of $Z_n$.
    Hence, by \cite[Lemma 4.10]{kallenbergFoundationsModernProbability2021},
    we have $\PP^{Y_n,Z_n} = \PP^{Y_n} \otimes \PP^{Z_n}$.
    By Tonelli's theorem, it follows that
    \begin{align}
        \int \one_{\ball^{\ast}(x,C_3 r)}\Bigl(\frac{Y_n}{C_0n^{\alpha}} + Z_n\Bigr) \dif \PP
        &= \int \one_{\ball^{\ast}(x,C_3 r)}\Bigl(\frac{y_n}{C_0n^{\alpha}} + z_n\Bigr) \, \dif\PP^{(Y_n,Z_n)}(y_n,z_n)\nonumber\\
        &= \iint \one_{\ball^{\ast}(x,C_3 r)}\Bigl(\frac{y_n}{C_0n^{\alpha}} + z_n\Bigr) \, \dif\PP^{Y_n}(y_n)\, \dif\PP^{Z_n}(z_n).
        \label{eqn_1_2}
    \end{align}
    Now fix $z_n \in \XX$ and consider the inner integral.
    We have
    \begin{align}
       \int\one_{\ball^{\ast}(x,C_3 r)}\Bigl(\frac{y_n}{C_0n^{\alpha}} + z_n\Bigr) \, \dif\PP^{Y_n}(y_n)
       &= \PP\left( \frac{Y_n}{C_0n^{\alpha}} + z_n \in \ball^{\ast}(x,C_3 r)\right)\nonumber\\
       &= \PP\Bigl( Y_n \in \ball^{\ast}\bigl(C_0n^{\alpha}(x-z_n),C_3C_0n^{\alpha}r\bigr)\Bigr)\nonumber\\
       &\le \sup_{z\in \XX} \PP\Bigl( Y_n \in \ball^{\ast}\bigl(z,C_3C_0n^{\alpha}r\bigr)\Bigr).
        \label{eqn_1_3}
    \end{align}
    Notice that this upper bound does not depend on $z_n$.
    Combining \eqref{eqn_1_1}, \eqref{eqn_1_2}, and \eqref{eqn_1_3}, we deduce that
    \begin{equation}
        \mu\bigl(\ball^{\ast}(x,C_3r)\bigr) 
        \le\sup_{z\in \XX} \PP\Bigl( Y_n \in \ball^{\ast}\bigl(z,C_3C_0n^{\alpha}r\bigr)\Bigr).
      \label{eqn_1_8}
  \end{equation}

  Now let $z \in \XX$ be fixed and let $n\ge n_0$.
  By the bi-Lipschitz equivalence of $\norm{\cdot}$ and $\norm{\cdot}_{\ast}$, 
  we have
  \begin{equation}
      \label{eqn_1_10}
      \ball^{\ast}\bigl(z,C_3C_0n^{\alpha}r\bigr)   
      \subset \ball\bigl(z,C_3^2C_0n^{\alpha}r\bigr).
  \end{equation}
  Using that $Y_n$ is uniformly distributed on $\SC_n$ at (a) and \eqref{eqn_1_10} at (b), we infer that
  \begin{align}
      \PP\Bigl( Y_n \in \ball^{\ast}\bigl(z,C_3C_0n^{\alpha}r\bigr)\Bigr)
     &\overset{(a)}{=} \frac{ \num\bigl(\ball^{\ast}(z,C_3C_0n^{\alpha}r) \cap \SC_n\bigr)}{\num\SC_n}\nonumber\\
     &\overset{(b)}{\le} \frac{ \num\bigl(\ball(z,C_3^2C_0n^{\alpha}r) \cap \SC_n\bigr)}{\num\SC_n}
     \label{eqn_1_5}
  \end{align}
  Combining \eqref{eqn_1_11}, \eqref{eqn_1_8}, and \eqref{eqn_1_5}, we deduce that
  \begin{equation}
        \om{\mu}(\ball(x,r)) 
        \le \frac{ \num\bigl(\ball(z,C_3^2C_0n^{\alpha}r) \cap \SC_n\bigr)}{\num\SC_n}
        \quad \text{for all $n\in \NN$, $x\in \XX$, $r>0$.}
        \label{eqn_1_20}
  \end{equation}

  \medskip{}
  \emph{Step 3:}
  We now choose a suitable value for $n\in \NN$ in \eqref{eqn_1_20}.
  Assume for the moment that
  \begin{equation}
      2C_3^2C_2 C_0 n^{\alpha} r \le 2^{-n}.
      \label{eqn_1_4}
  \end{equation}
  Then, for $z',z'' \in \ball(z,C_3^2C_0n^{\alpha}r) \cap \SC_n$ we would have 
  \begin{equation*}
      \norm{z'-z''} 
      \le C_2 \cdot \bigl( \norm{z'-z} +\norm{z-z''}\bigr)
      \le C_2 \cdot 2C_3^2C_0 n^{\alpha}r
      \le 2^{-n}.
  \end{equation*}
  Since $\SC_n$ is $2^{-n}$-separated, we would infer that $z'=z''$.
  Therefore, we would obtain that $\num\bigl(\ball(z,C_3^2C_0n^{\alpha}r) \cap \SC_n\bigr)\le 1$.

  With this consideration in mind, we define 
  \begin{equation}
      n:= n(r):= \Bigl\lfloor \log_2(u) + \log_2\bigl(\log_2(u)^{-\alpha}\bigr) \Bigr\rfloor,
      \qquad
      \text{where } u:=u(r):=\frac{1}{2C_3^2C_2C_0r}.
      \label{eqn_1_6}
  \end{equation}
  Notice that $\lim_{r\to 0}u(r)=\infty$.
  Hence, for all sufficiently small $r > 0$ (depending on $C_0,C_2,C_3$)
  we have $\log_2(u)>0$ and so $n(r)$ is well-defined in $\ZZ$.
  Furthermore, we have
  \begin{equation*}
      n(r)
      \ge \log_2(u) + \log_2\bigl(\log_2(u)^{-\alpha}\bigr) -1
      = \log_2(u)\Bigl(1 -\alpha \frac{\log_2\log_2(u) }{\log_2(u)} - \frac{1}{\log_2(u)}\Bigr)
      .
  \end{equation*}
  Therefore, for all sufficiently small $r > 0$ (depending on $C_0,C_2,C_3,\alpha,n_0$)
  we have $n(r) \in \NN$ and $n(r) \ge n_0$.
  Regarding \eqref{eqn_1_4}, we notice that the function $\nu \mapsto 2^{\nu} {\nu}^{\alpha}$ is increasing on $(0,\infty)$.
  Hence, for all sufficiently small $r > 0$ (depending on $C_0,C_2,C_3,\alpha,n_0$) we have
  \begin{align*}
      2^n n^{\alpha} 
      &\le 2^{\log_2(u) + \log_2\bigl(\log_2(u)^{-\alpha}\bigr)} \Bigl(\log_2(u) + \log_2\bigl(\log_2(u)^{-\alpha}\bigr)\Bigr)^{\alpha}\\
      &= u \cdot \log_2(u)^{-\alpha} \cdot \Bigl(\log_2(u) + \log_2\bigl(\log_2(u)^{-\alpha}\bigr)\Bigr)^{\alpha}\\
      &= u \cdot \Bigl( 1 -\alpha \frac{\log_2\log_2(u)}{\log_2(u)}\Bigr)^{\alpha}.
  \end{align*}
  Thus, for all sufficiently small $r > 0$ (depending on $C_0,C_2,C_3,\alpha,n_0$),
  we have $2^nn^{\alpha}\le u(r)$, which is precisely \eqref{eqn_1_4}.
  Therefore, \eqref{eqn_1_20} implies that for all sufficiently small $r > 0$
  (depending on $C_0,C_2,C_3,\alpha,n_0$) and $n(r)$ as in \eqref{eqn_1_6}, we have
  \begin{equation}
      \om{\mu}(\ball(x,r)) \le \frac{1}{\num \SC_n}
      \quad \text{for all $x\in \XX$}.
      \label{eqn_1_21}
  \end{equation}

  \medskip{}
  \emph{Step 4a.}
  We consider the case $s>0$ and $t\in \RR$.
  Looking back at \eqref{eqn_1_21} and recalling \eqref{eqn_assumption_lower_bound_packing_num},
  it remains to make the dependency of $2^{ns}n^t$ on $r$ explicit with $n=n(r)$ is as in \eqref{eqn_1_6}.
  Since the map $\nu \mapsto 2^{\nu s }\nu^t$ is increasing on $[\nu_0, \infty)$
  for sufficiently large $\nu_0 \geq 1$ (as can easily be seen by taking the derivative),
  we have for all sufficiently small $r > 0$ (depending on $C_0,C_2,C_3,\alpha,n_0,s,t$) that
    \begin{align*}
        2^{ns}n^t 
        &\ge 2^{s\bigl(\log_2(u) + \log_2\bigl(\log_2(u)^{-\alpha}\bigr)-1\bigr)} \cdot \Bigl( \log_2(u) + \log_2\bigl(\log_2(u)^{-\alpha}\bigr)-1\Bigr)^{t}\\
        &= u^s \cdot \log_2(u)^{-s\alpha} \cdot 2^{-s} \cdot \Bigl( \log_2(u) + \log_2\bigl(\log_2(u)^{-\alpha}\bigr)-1\Bigr)^{t}\\
        &= u^s \cdot \log_2(u)^{t-s\alpha} \cdot 2^{-s} \cdot \Bigl( 1 -\alpha \frac{\log_2\log_2(u)}{\log_2(u)} -\frac{1}{\log_2(u)}\Bigr)^{t}.
    \end{align*}
    Hence, for all sufficiently small $r > 0$ (depending on $C_0,C_2,C_3,\alpha,n_0,s,t$), we have
    \begin{align}
        2^{ns}n^t 
        &\ge u^s \log_2(u)^{t-s\alpha} \cdot 2^{-s-1}\nonumber\\
        &= \Bigl(\frac{1}{2 C_0C_2C_3^2}\Bigr)^{s} \cdot \Bigl(\frac{1}{r}\Bigr)^{s} \cdot \Bigl(\log_2\Bigl( \frac{1}{r}\Bigr) -\log_2(2C_0C_2C_3^2)\Bigr)^{t-s\alpha}\cdot 2^{-s-1}\nonumber\\
        &=\frac{1}{2}\Bigl(\frac{1}{4 C_0C_2C_3^2}\Bigr)^{s} \cdot \Bigl(\frac{1}{r}\Bigr)^{s} \cdot \log_2\Bigl(\frac{1}{r}\Bigr)^{t-s\alpha}
        \cdot \Bigl( 1 - \frac{\log_2(2C_0C_2C_3^2)}{\log_2(1/r)}\Bigr)^{t-s\alpha}\nonumber\\
        &\ge \frac{1}{4}\Bigl(\frac{1}{4 C_0C_2C_3^2}\Bigr)^{s} \cdot \Bigl(\frac{1}{r}\Bigr)^{s} \cdot \log_2\Bigl(\frac{1}{r}\Bigr)^{t-s\alpha}.
        \label{eqn_1_7}
    \end{align}
    By \cref{eqn_assumption_lower_bound_packing_num}, it follows that
    \begin{equation*}
        \frac{1}{\num \SC_n }
       \overset{\eqref{eqn_assumption_lower_bound_packing_num}}{\le} \exp\bigl( -c_1 \cdot 2^{ns}n^{t}\bigr)
       \overset{\eqref{eqn_1_7}}{\le}\exp \Bigl( -\frac{c_1}{4}\Bigl(\frac{1}{4 C_0C_2C_3^2}\Bigr)^{s} \cdot \Bigl(\frac{1}{r}\Bigr)^{s} \cdot \log_2\Bigl(\frac{1}{r}\Bigr)^{t-s\alpha} \Bigr).
    \end{equation*}
    In combination with \eqref{eqn_1_21}, this proves \eqref{eqn_growth_of_mu} in the case $s>0$ and $t\in \RR$.

    \medskip{}
    \emph{Step 4b.} We consider the case $s=0$ and $t>0$.
    In that case, we compute
    \begin{equation*}
       n^t 
       \ge \bigl(\log_2(u) + \log_2\bigl(\log_2(u)^{-\alpha}\bigr)-1\bigr)^t
       = \log_2(u)^t \cdot \Bigl( 1 -\alpha \frac{\log_2\log_2(u)}{\log_2(u)} - \frac{1}{\log_2(u)}\Bigr)^t.
    \end{equation*}
    Hence, for all sufficiently small $r > 0$, depending on $C_0,C_2,C_3,n_0,t$, we have
    \begin{align}
        n^t 
        &\ge \frac{1}{2} \log_2(u)^t
        = \frac{1}{2}\Bigl(\log_2\Bigl( \frac{1}{r}\Bigr) -\log_2(2C_3^2C_2C_0)\Bigr)^t\nonumber\\
        &= \frac{1}{2}\Bigl(\log_2\Bigl(\frac{1}{r}\Bigr)\Bigr)^t \cdot \Bigl( 1 - \frac{\log_2(2C_3^2C_2C_0)}{\log_2(1/r)}\Bigr)^t
        \ge \frac{1}{4}\Bigl(\log_2\Bigl(\frac{1}{r}\Bigr)\Bigr)^t.
        \label{eq:1_8}
    \end{align}
    By \cref{eqn_assumption_lower_bound_packing_num}, it follows that
    \begin{equation*}
        \frac{1}{\num \SC_n }
       \overset{\eqref{eqn_assumption_lower_bound_packing_num}}{\le} \exp\bigl( -c_1 \cdot n^{t}\bigr)
       \overset{\eqref{eq:1_8}}{\le}\exp \Bigl( -\frac{c_1}{4} \bigl(\log_2(1/r)\bigr)^t \Bigr).
    \end{equation*}
    In combination with \eqref{eqn_1_21}, this proves \eqref{eqn_growth_of_mu} in the case $s=0$ and $t> 0$.
\end{proof}

\begin{remark}
 The case $s=0$ and $t=1$ corresponds to the polynomial scale.
 In that case, however, the result is not optimal.
 For example, if the packing numbers scale like $(1/\eps)^d$, or more precisely,
 if $\num\SC_n \ge 2^{nd}$ for $\eps=2^{-n}$, we get $\num\SC_n \ge \exp(\ln(2) nd)$.
 \cref{measure_SC_has_growth_order} then only provides a measure $\mu$ with
 $\mu(\ball_r) \le \exp(-c \log_2(1/r)) = r^{\tilde{c}}$,
 where $\tilde{c} > 0$ might be much smaller than $d$.
\end{remark}

In the following, we show that a kind of converse of \cref{measure_SC_has_growth_order} also holds;
see \cref{measure_gives_packings} below.
The proof relies on the following lemma.

\begin{lemma}\label{growth_order_implies_lower_bnd_covering_num}
 Let $(\SC, \varrho)$ be a non-empty quasi-metric space.
 Let $\mu$ be a Borel probability measure on $\SC$, and let $s\in (0,\infty)$ and $t\in \RR$.
 Assume that there exist $r_0>0$ and $c>0$ such that
 \begin{equation*}
     \om{\mu} \bigl( \ball(x,r)\bigr) 
    \leq \exp
         \biggl(
           - c
             \cdot \Bigl( \frac{1}{r}\Bigr)^{s}
             \cdot \Bigl( \log_2 \Bigl(\frac{1}{r}\Bigr)\Bigr)^{t}
         \biggr)
    \quad \text{for all $x\in \SC$ and all $r\in (0,r_0)$.}
 \end{equation*}
 Then for every $\eps\in (0,r_0)$ and every $\eps$-net $\SC_{\eps}\subset \SC$ we have
 \begin{equation*}
     \num\SC_{\eps}
     \geq \exp
          \biggl(
            c \cdot \left(\frac{1}{\eps}\right)^s
              \cdot \left(\log_2 \left(\frac{1}{\eps}\right)\right)^{t}
          \biggr).
 \end{equation*}
\end{lemma}

\begin{proof}
  Let $\eps\in (0,r_0)$ and let $\SC_{\eps} \subset \SC$ be an $\eps$-net for $\SC$.
  If $\num\SC_{\eps} = \infty$, we are done.
  Assume in the following that $\num\SC_{\eps}<\infty$.
  It follows from the definition of an $\eps$-net that $\SC \subset \bigcup_{x \in \SC_{\eps}} \ball(x,\eps)$.
  Therefore, we have
  \begin{equation*}
    1
    =\mu(\SC) 
    =\om{\mu}(\SC)
    \leq \sum_{x\in \SC_{\eps}} \om{\mu}\bigl( \ball(x,\eps)\bigr)
    \leq \num(\SC_{\eps})
         \cdot \exp
               \biggl(
                 - c
                   \cdot \Bigl( \frac{1}{\eps}\Bigr)^{s}
                   \cdot \Bigl( \log_2 \Bigl(\frac{1}{\eps}\Bigr)\Bigr)^{t}
             \biggr).
  \end{equation*}
  From this, we deduce the claim.
\end{proof}

\begin{proposition}\label{measure_gives_packings}
 Let $(\SC, \varrho)$ be a non-empty quasi-metric space.
 Let $\mu$ be a Borel probability measure on $\SC$, and let $s\in (0,\infty)$ and $t\in \RR$.
 Assume that there exist $r_0>0$ and $c_0>0$ such that
 \begin{equation*}
     \om{\mu} \bigl( \ball(x,r)\bigr) 
    \le \exp\biggl( - c_0 \cdot \Bigl( \frac{1}{r}\Bigr)^{s} \cdot \Bigl( \log_2 \Bigl(\frac{1}{r}\Bigr)\Bigr)^{t}\biggr)
    \quad \text{for all $x\in \SC$ and all $r\in (0,r_0)$.}
 \end{equation*}
 Then for every $\eps\in (0,r_0)$ there exists an $\eps$-separated subset $P_{\eps}$ with
 \begin{equation*}
   \num P_{\eps}
   \geq \exp
        \biggl(
          c_0
          \cdot \left(\frac{1}{\eps}\right)^s
          \cdot \left(\log_2 \left(\frac{1}{\eps}\right)\right)^{t}
        \biggr).
 \end{equation*}
\end{proposition}

\begin{proof}
 Let $\eps\in (0,r_0)$ and let $P_{\eps}\subset \SC$ be a maximal $\eps$-separated subset.
 Such a subset exists by \cref{maximal_separated_subsets_exist}.
 Then, by \cref{maximal_separated_subsets_are_nets}, $P_{\eps}$ is an $\eps$-net.
 The claim then follows from \cref{growth_order_implies_lower_bnd_covering_num}.
\end{proof}

\subsection{Critical measures}
\label{sub:CriticalMeasureExistence}

In \cref{measure_SC_has_growth_order}, we have shown that given a suitable sequence
of packings in $\SC$, we can construct a measure on $\SC$ that spreads its mass
according to the size of the packings;
see \eqref{eqn_assumption_lower_bound_packing_num} and \eqref{eqn_growth_of_mu}.
In this section, we show that if the packings used in \cref{measure_SC_has_growth_order} are maximal,
then the resulting measure is ``maximally spread out'',
meaning that it is critical in the sense of \cref{def:CriticalMeasure}.

\begin{theorem}\label{equivalence_lower_minkowski_existence_of_measure}
    Let $(\XX,\norm{\cdot})$ be a quasi-normed space.
    Let $\SC\subset \XX$ be a non-empty subset satisfying \ConditionZetaText.
    Then there exists a Borel probability measure on $\XX$ that is critical for $\SC$
    and that satisfies $\om{\mu}(\XX\setminus\SC) = 0$ and $\om{\mu}(\SC) = 1$.
\end{theorem}

\begin{proof}
    If $\lmd(\SC) = 0$, then we can choose $\delta_{x_0}$ for some $x_0 \in \SC$,
    since $\delta_{x_0}$ vacuously satisfies the condition in \cref{def:CriticalMeasure}.

    Assume in the following that $\lmd(\SC) > 0$.
    We will deduce the existence of $\mu$ with the desired properties
    from \Cref{measure_SC_has_growth_order}.

    \medskip{}
    \emph{Step 1:}
    We claim that there exists a sequence $(\SC_n)_{n \in \NN}$ of finite subsets
    $\SC_n \subset \SC$ such that for every $\tilde{s} \in (0,\lmd(\SC))$
    there exists $n_0 = n_0(\tilde{s})\in \NN$ such that
    \begin{equation}
        \label{eqn_3_5}
        \num \SC_n
        \geq \exp
             \bigl(
               \ln(2) \cdot 2^{n\tilde{s}}
             \,\bigr)
      \quad \text{for every $n\ge n_0$}.
    \end{equation}

    We first consider the case when $\SC$ is totally bounded.
    Let
    \begin{equation}\label{eqn_3_4}
      \scrit
      := \lmd(\SC)
       = \liminf_{\eps \to 0}
         \frac{ \log_2 \Bigl(\log_2 \bigl(\cn(\SC,\eps)\bigr)\Bigr)}
              {\log_2 \bigl(\frac{1}{\eps}\bigr)}
      \in (0,\infty].
    \end{equation}
    For $n\in \NN$ let $\SC_n \subset \SC$ be a maximal $2^{-n}$-separated subset.
    By \cref{maximal_separated_subsets_exist}, such subsets exist.
    Since $\SC$ is totally bounded, it follows that $\num\SC_n<\infty$ for all $n\in \NN$.

    Let $\tilde{s} <\scrit$.
    Then our assumption \eqref{eqn_3_4} together with \cref{lemma_technical_log_log_asymp}
    imply that there exists $\eps_0 > 0$ such that
    \begin{equation}
        \label{eqn_3_1}
        \cn(\SC,\eps) \ge 2^{\eps^{-\tilde{s}}}
        \quad
        \text{for all $\eps \in (0,\eps_0)$.}
    \end{equation}
    Since $\SC_n \subset \SC$ is a maximal $2^{-n}$-separated subset,
    it is a $2^{-n}$-net for $\SC$ by \cref{maximal_separated_subsets_are_nets}.
    Hence $\num(\SC_n)\ge \cn(\SC,2^{-n})$ for all $n\in \NN$.
    Using this at (a) and \eqref{eqn_3_1} at (b), we deduce that
    \begin{equation*}
       \num(\SC_n)
       \overset{(a)}{\ge} \cn(\SC,2^{-n}) 
       \overset{(b)}{\ge} \exp\bigl(\ln(2) \cdot 2^{n\tilde{s}}\bigr)
    \end{equation*}
    for all $n\ge n_0$, where $n_0\in \NN$ is so large that $2^{-n_0}< \eps_0$.
    This shows \eqref{eqn_3_5} in the case when $\SC$ is totally bounded.
    
    Now we consider the case when $\SC$ is not totally bounded.
    Then, there exists $n_1\in \NN$ such that for every $n \geq n_1$
    there exists an infinite $2^{-n}$-separated subset of $\SC$.
    In particular, for every $n \geq n_1$ there exists a finite $2^{-n}$-separated subset
    $\SC_n\subset \SC$ such that 
    \begin{equation*}
        \num\SC_n 
        \ge \exp\Bigl( \ln(2) \cdot 2^{n^2}\Bigr) \quad \text{for every $n\ge n_1$}.
    \end{equation*}
    For $n < n_1$, define $\SC_n := \{ x_0 \}$ for a fixed $x_0 \in \SC \neq \emptyset$.

    Let $\tilde{s} \in (0,\lmd(\SC))$ and let
    $n_0 := \max\bigl\{ \lceil \,\tilde{s}\, \rceil, n_1 \bigr\}$.
    Then for every $n \geq n_0$, we have
    \begin{equation*}
        \num\SC_n 
        \ge \exp\Bigl( \ln(2) \cdot 2^{n^2}\Bigr)
        \ge \exp\Bigl( \ln(2) \cdot 2^{n\tilde{s}}\Bigr).
    \end{equation*}
    This shows \eqref{eqn_3_5}.

    \medskip{}
    \emph{Step 2:}
    With the subsets $(\SC_n)_{n\in \NN}$ in place,
    we can now define our measure and deduce some of the claimed properties.
    Recall that $\SC$ satisfies \ConditionZetaText.
    Hence there exist $\alpha>0$ and $C_0>0$ such that $\SC$ satisfies \ConditionZetaTextAlphaC{(\alpha,C_0)}.
    Let $\mu:= \mu\bigl[ (\SC)_{n\in \NN},\alpha,C_0\bigr]$ be the measure from \cref{definition_measure_SC_n}.
    Then \cref{measure_SC_n_is_well_def} implies that $\mu$ is a Borel probability measure on $\XX$
    and that $\om{\mu}(\XX\setminus\SC)=0$ and $\om{\mu}(\SC) = 1$.

    \medskip{}
    \emph{Step 3:} We will now investigate the growth behavior of $\mu$.
    Let $s \in [0,\lmd(\SC))$ and let $\tilde{s} \in (s,\lmd(\SC))$.
    It follows from \eqref{eqn_3_5} and \cref{measure_SC_has_growth_order} with $t=0$
    that there exist $c_1= c_1(\tilde{s}\,) > 0$ and $r_1=r_1(\tilde{s}\,) > 0$ such that
    \begin{equation}
        \label{eqn_3_2}
        \om{\mu}\bigl(\ball(x,r)\bigr)
       \leq \exp
            \biggl(
              -c_1
              \cdot \Bigl(\frac{1}{r}\Bigr)^{\tilde{s}}
              \cdot \Bigl( \log_2 \Bigl(\frac{1}{r}\Bigr)\Bigr)^{-\alpha\tilde{s}}
            \biggr)
    \end{equation}
    for all $x\in \XX$ and all $r\in (0,r_1)$.
    Since $s<\tilde{s}$, we have
    \begin{equation*}
        \lim_{r\to0}\left( \Bigl(\frac{1}{r}\Bigr)^{\tilde{s}-s} \log_2 \Bigl(\frac{1}{r}\Bigr)^{-\alpha \tilde{s}} \right)
        = \infty.
    \end{equation*}
    Therefore, there exists $r_2 \in (0,r_1)$ such that
    \begin{equation}
        \label{eqn_3_3}
       \Bigl(\frac{1}{r}\Bigr)^{\tilde{s}-s} \log_2 \Bigl(\frac{1}{r}\Bigr)^{-\alpha \tilde{s}}
       \ge 1 
       \quad \text{for all $r\in (0,r_2)$.}
    \end{equation}
    Combining \eqref{eqn_3_2} with \eqref{eqn_3_3}, we infer that
    \begin{equation}
        \om{\mu}\bigl(\ball(x,r)\bigr)
       \le \exp\biggl( -c_1 \cdot \Bigl(\frac{1}{r}\Bigr)^{s} \biggr)
    \end{equation}
    for all $x\in \XX$ and all $r\in(0,r_2)$.
    This shows that $\mu$ satisfies the small-ball condition of order $s$. 
    Since $s \in [0,\lmd(\SC))$ was arbitrary, we conclude that
    the measure $\mu$ is critical for $\SC$.
\end{proof}

\subsection{Transferring critical measures and intrinsic characterization}
\label{sub:TransferringCriticalMeasures}

In applications, e.g.~in the proof of \cref{prop:ShearletsOptimalForC2Boundary},
it is useful \emph{(i)} to be able to transfer a given measure from one quasi-metric space
to another in such a way that the small-ball condition is affected in a predictable way,
and \emph{(ii)} to have an ``intrinsic'' perspective, where the ambient space $\XX$ is forgotten
and only the quasi-metric space $\SC$ (with the trace quasi-metric) is considered.
This section provides several technical results that address these points.

The following \cref{prop:transfer_principle} shows that the small-ball condition is preserved,
potentially with a different order, if one pushes a measure forward using an ``expansive'' map.
This is a generalization of \cite[Lemma~7]{grohsPhaseTransitionsRate2023},
where subsets of Banach spaces in the case $\beta=1$ were considered.

\begin{proposition}\label{prop:transfer_principle}
 Let $(\SC,\metr)$ and $(\SC',\metr')$ be two non-empty quasi-metric spaces. 
 Let $\mu$ be a Borel probability measure on $\SC$ that satisfies the small-ball condition
 with parameter $s>0$ on $\SC$.
 Let $\Phi \colon \SC \to \SC'$ be a Borel-measurable map that satisfies
 the following sub-H{\"o}lder condition:
 there exist $c>0$ and $\beta>0$ such that
 \begin{equation}
     \label{eq:sub_Hölder_small_condition}
     \metr'\bigl( \Phi(x_1), \Phi(x_2) \bigr) \ge c \cdot \metr( x_1,x_2)^{\beta} \quad \text{for all $x_1,x_2\in \SC$.}
 \end{equation}
 Then the measure $\Phi_{\#}\mu$ is a Borel probability measure on $\SC'$
 that satisfies the small-ball condition with parameter $\frac{s}{\beta}$ on $\SC'$.
\end{proposition}

\begin{proof}
    Since $\Phi$ is Borel-measurable, we immediately get that $\Phi_{\#} \mu$ is a
    well-defined probability measure on $\SC '$.
    We thus only have to verify the small-ball condition.
    
    For this, we start with some preparations:
    Let $c_1>0$ and $r_1>0$ denote the constants from the small-ball condition for $\mu$,
    see \eqref{eq:SmallBallConditionPrecise}.
    Next, note as a consequence of \cref{quasi_metric_Heinonen_result}
    that there exist quasi-metrics $\varrho_{\ast}$ and $\varrho'_{\ast}$,
    and constants $C_2, C_3 \geq 1$ such that
    \begin{equation*}
        \frac{1}{C_2} \varrho \leq \varrho_{\ast} \le C_2 \varrho 
        \quad \text{and} \quad
        \frac{1}{C_3} \varrho' \leq \varrho'_{\ast} \le C_3 \varrho'
        ,
    \end{equation*}
    and such that in addition, the quasi-metrics $\varrho_{\ast}$ and $\varrho'_{\ast}$,
    when raised to some suitable power, are metrics. 
    In particular, the balls with respect to $\varrho_{\ast}$ and $\varrho'_{\ast}$
    are closed sets and thus measurable.
    Furthermore, let $C_4 \geq 1$ denote a triangle constant for $\varrho'_{\ast}$.

    We now verify the definition of the small-ball condition for $\Phi_{\#}\mu$.
    Let $y \in \SC'$ and let $r > 0$.
    Then we have $\ball(y,r \mid \varrho') \subset \ball(y, C_3r \mid \varrho'_{\ast})$.
    By the monotonicity of outer measures, and since the latter ball is
    Borel measurable (in fact, closed), we get
    \begin{equation}
        \label{eqn_10_2}
        (\Phi_{\#}\mu)^{\ast} \bigl( \ball(y,r\mid \varrho')\bigr)
        \le (\Phi_{\#}\mu)^{\ast}\bigl( \ball(y,C_3 r\mid \varrho_{\ast}')\bigr)
        = \Phi_{\#}\mu \bigl( \ball(y,C_3 r\mid \varrho_{\ast}')\bigr).
    \end{equation}

    We now distinguish two cases.
    In the first case, assume that $\Phi(\SC) \cap \ball(y,C_3 r\mid \varrho_{\ast}') = \emptyset$.
    Then $\Phi_{\#}\mu \bigl( \ball(y,C_3 r\mid \varrho_{\ast}')\bigr)=0$,
    which implies the small-ball condition for the ball $\ball(y,r\mid \varrho')$. 
    In the second case, assume that $\Phi(\SC) \cap \ball(y,C_3 r\mid \varrho_{\ast}') \neq \emptyset$.
    Then there exists $z \in \Phi(\SC) \cap \ball(y,C_3 r\mid \varrho_{\ast}')$.
    It follows from \cref{lemma_ball_inclusions} that 
    \begin{equation}
        \ball(y, C_3 r \mid \varrho_{\ast}')
        \subset \ball(z, 2C_3C_4 r \mid \varrho_{\ast}')
        .
        \label{eq:SmallBallTransferTechnicalInclusion}
    \end{equation}
    Let $x\in \SC$ such that $z = \Phi(x)$. 
    Then for every $v \in \Phi^{-1}\bigl( \ball(z,2 C_3 C_4 r \mid \varrho_{\ast}')\bigr)$,
    our assumption \eqref{eq:sub_Hölder_small_condition} implies that
    \begin{equation*}
        c \cdot \varrho(x,v)^{\beta}
        \le \varrho'(z, \Phi(v))
        \le C_3 \cdot \varrho'_{\ast}(z,\Phi(v))
        \le 2 C_3^2 C_4 r.
    \end{equation*}
    Hence,
    \begin{equation}
        \label{eqn_10_1}
       \Phi^{-1}\bigl( \ball(z,2 C_3 C_4 r \mid \varrho_{\ast}')\bigr)
       \subset \ball\Bigl(x, \Bigl(\frac{2 C_3^2 C_4 r}{c}\Bigr)^{\frac{1}{\beta}} \mid \varrho\Bigr).
    \end{equation}

    Inserting \eqref{eqn_10_1} into \eqref{eqn_10_2}, and using the small-ball condition for $\mu$, we obtain
    \begin{align*}
        (\Phi_{\#}\mu)^{\ast} \bigl( \ball(y,r\mid \varrho')\bigr)
        \hspace*{0.3cm}
        &\hspace*{-0.25cm}\overset{\eqref{eqn_10_2}}{\leq} \Phi_{\#}\mu \bigl( \ball(y,C_3 r\mid \varrho_{\ast}')\bigr)
        = \mu \Bigl( \Phi^{-1}\bigl(\ball(y,C_3 r\mid \varrho_{\ast}')\bigr)\Bigr)\\
        &\hspace*{-0.6cm}\overset{\eqref{eq:SmallBallTransferTechnicalInclusion},\eqref{eqn_10_1}}{\leq}
        \mu^{\ast} \biggl( \ball\Bigl(x, \Bigl(\frac{2 C_3^2 C_4 r}{c}\Bigr)^{\frac{1}{\beta}} \,\bigg|\, \varrho\Bigr) \biggr)\\
        &\leq \exp\biggl( -c_1 \cdot \Bigl(\frac{c}{2 C_3^2 C_4}\Bigr)^{\frac{s}{\beta}} r^{-\frac{s}{\beta}}\biggr)
        ,
    \end{align*}
    provided that $(\frac{2 C_3^2 C_4 r}{c})^{1/\beta} < r_1$, which is satisfied for all
    sufficiently small $r > 0$.
    This shows the small-ball condition for $\Phi_{\#}\mu$.
\end{proof}

As pointed out at the start of \Cref{sec:ExistenceOfCriticalMeasures},
\Cref{def:CriticalMeasure} is ``extrinsic'' in the sense that the
set $\SC$ is considered as a subset of a larger ambient space $\XX$.
In contrast, the ``intrinsic'' point of view would consider the space $\SC$
equipped with the trace quasi-metric and the resulting Borel $\sigma$-algebra $\Borel(\SC)$
as the primary object.

It is well-known that this $\sigma$-algebra, i.e.~$\Borel(\SC)$,
agrees with the trace of $\Borel(\XX)$ on $\SC$.
As we show below, this allows us to pass from the extrinsic point of view to the intrinsic one
by simply restricting the quasi-metric, the Borel $\sigma$-algebra, and the measure.
For this, we will need the following notation:
For a family $\mathcal{F}$ of subsets of $\XX$ and a fixed subset $\SC \subset \XX$,
we let
\[
  \mathcal{F} \sigmacap \SC
  := \SC \sigmacap \mathcal{F} := \{ F \cap \SC : F \in \mathcal{F}\}
  .
\]
In particular, if $\mathcal{F}$ is a $\sigma$-algebra on $\XX$,
$\mathcal{F} \sigmacap \SC$ is the \emph{trace of $\mathcal{F}$ on $\SC$}.

The following \cref{restriction_inherits_growht_order} implies that restricting a measure to a
subset of full outer measure preserves the growth order.
The rather technical proof is deferred to \cref{sec_proof_of_restriction_inherits_growht_order}.

\begin{lemma}\label{restriction_inherits_growht_order}
    Let $(\XX,\metr)$ be a non-empty quasi-metric space,
    let $\mathcal{F}$ be a sigma-algebra on $\XX$, let $\SC \subset \XX$ be a non-empty subset,
    and let $\mu \colon \mathcal{F}\to [0,1]$ be a probability measure on $\XX$ with $\om{\mu}(\SC)=1$.
    Define the measure $\nu$ on $\mathcal{F}\sigmacap \SC$ via $\nu(A\cap \SC) := \mu(A)$ for $A\in \mathcal{F}$.
    Then $\nu$ is a well-defined probability measure on $\SC$.
    Moreover, the following hold:
    \begin{enumerate}[label=(\roman*)]
      \item If $\mu$ is a Borel measure on $\Borel(\XX)$, then $\nu$ is a Borel measure on $\Borel(\SC)$.

      \item If
            \begin{equation}
                \label{eqn_restrictions_growth_condition_mu}
                \om{\mu}\bigl( \ball\bigl(x,r \mid (\XX,\metr)\bigr)\bigr) \le \exp\bigl(-c_0 \cdot r^{-s}\bigr)
            \end{equation}
            for all $r\in (0,r_0)$ and all $x\in \SC$, then
            \begin{equation}
                \label{eqn_restrictions_growth_condition_nu}
                \om{\nu}\bigl( \ball\bigl(x,r \mid (\SC,\metr\vert_{\SC\times \SC})\bigr)\bigr) 
                \le \exp\bigl( -c_0 \cdot r^{-s}\bigr)
            \end{equation}
            for all $r\in (0,r_0)$ and all $x\in \SC$.
\end{enumerate}
\end{lemma}

The previous \cref{restriction_inherits_growht_order} implies that it is always possible
to move from an ``extrinsic'' point of view to an ``intrinsic'' one.
The following corollary explicitly records this fact.

\begin{corollary}\label{restriction_inherits_criticality}
    Let $(\XX,\metr)$ be a non-empty quasi-metric space,
    let $\SC \subset \XX$ be a non-empty subset, and let $\mu$ be a probability measure on $\XX$
    with $\om{\mu}(\SC)=1$.
    If $\mu$ is critical for $\SC$,
    then the measure $\nu$ as defined in \cref{restriction_inherits_growht_order}
    is critical for $\SC$.
\end{corollary}

\begin{proof}
  By \cref{restriction_inherits_growht_order},
  the measure $\nu$ is a well-defined probability measure on $\SC$. 
  Furthermore, the definition of $\lmd(\SC)$ implies that the value
  of $\lmd(\SC)$ only depends on the restriction of the metric on $\XX$ to $\SC$.
  In combination with \cref{restriction_inherits_growht_order},
  this implies that $\nu$ is critical for $\SC$.
\end{proof}

\cref{prop:transfer_principle} applies in particular to the case
when $\SC$ is a subset of $\XX=:\SC'$ equipped with the trace metric
and the trace Borel sigma-algebra.
It follows that in the case of Borel measures, it is also possible to switch
from the ``intrinsic'' point of view to an extrinsic one by embedding the space $\SC$
into the larger space $\XX$,
as recorded in the following corollary.

\begin{corollary}\label{cor:IntrinsicToExtrinsic}
  Let $(\XX, \varrho)$ be a quasi-metric space, and let $\emptyset\ne \SC \subset \XX$.
  Let $\nu$ be a Borel probability measure on $\SC$ that is critical for $\SC$.
  Let $\mu$ denote the push-forward of $\nu$ under the inclusion map $\SC \embeds \XX$.
  Then $\mu$ is a Borel probability measure on $\XX$ that is critical for $\SC$.
\end{corollary}

\begin{proof}
  This follows from \Cref{prop:transfer_principle} (with $\beta = 1$),
  applied to the inclusion map, and the definition of a critical measure.
\end{proof}

Finally, we note that in order to check whether a measure has a given growth order,
it suffices to check balls with centers in some subset of full outer measure,
as stated in the following result.

\begin{corollary}
  Let $(\XX,\varrho)$ be a quasi-metric space, and let $\emptyset \ne \SC \subset \XX$.
  Let $\mu$ be a Borel probability measure on $\XX$ that satisfies $\om{\mu}(\SC)=1$
  and the small-ball condition with parameter $s>0$ on $\SC$, i.e.,
  there exist $r_0, c > 0$ such that
  \[
    \mu^\ast (\ball (x,r)) \leq \exp(- c \cdot (1/r)^s)
    \qquad \forall \, x \in \SC \text{ and } 0 < r < r_0
    ,
  \]
  where $\ball(x,r) = \{ y \in \XX \,\,:\,\, \varrho(y,x) \leq r \}$.

  Then $\mu$ satisfies the small-ball condition with parameter $s$ on $\XX$,
  i.e., there exist constants $r_1, c^\ast > 0$ such that
  \[
    \mu^\ast (\ball (x,r)) \leq \exp(- c^\ast \cdot (1/r)^s)
    \qquad \forall \, x \in \XX \text{ and } 0 < r < r_1
    .
  \]
\end{corollary}

\begin{proof}
  Let $\nu$ denote the restriction of $\mu$ to $\SC$
  (defined as in \Cref{restriction_inherits_growht_order}, meaning $\nu(A \cap \SC) := \mu(A)$
  for $A \in \Borel(\XX)$; this is well-defined by \Cref{restriction_inherits_growht_order})
  and let $\iota\colon \SC\to \XX$ denote the embedding.
  Then $\mu =\iota_{\#}\nu$ is the push-forward of $\nu$ by $\iota$.
  Indeed, for every $A \in \Borel(\XX)$, we have by the definitions of the push-forward measure,
  of the map $\iota$, and of the measure $\nu$, that
  \[
    \iota_{\#}\nu(A)
    = \nu( \iota^{-1}(A))
    = \nu( A \cap \SC)
    = \mu(A)
    .
  \]
  By assumption and \cref{restriction_inherits_growht_order},
  the measure $\nu$ is a well-defined Borel probability measure on $\SC$
  that satisfies the small-ball condition of order $s$ on $\SC$.
  By \cref{prop:transfer_principle}, the measure $\mu =\iota_{\#}\nu$
  then satisfies the small-ball condition with parameter $s$ on $\XX$.
\end{proof}

\section{Sets that satisfy \texorpdfstring{{\ConditionZetaText}}{Condition (Zeta)} and application to unit balls of quasi-Banach spaces}
\label{sec:sets_satisfying_cond_zeta}

In this section, we indicate the generality of \ConditionZetaText.
Firstly, in \cref{sec:convex_bounded_and_complete_are_zeta}, we show that convex,
complete, and bounded subsets of quasi-normed spaces satisfy \ConditionZetaText.
Secondly, in \cref{sec:unit_balls_are_zeta}, we show that whenever
$(\YY, \norm{\cdot}_{\YY}) \embeds (\XX, \norm{\cdot}_{\XX})$ is a continuous embedding
of quasi-Banach spaces, then the unit ball $\SC:= \ball\bigl(0,1\mid \norm{\cdot}_{\YY}\bigr)$
satisfies {\ConditionZetaText} as a subset of $\XX$.
Then, in \cref{sec_unit_Balls_B_F_are_zeta}, we apply these results to compactly embedded unit balls
in Besov- and Triebel-Lizorkin spaces and prove \cref{existence_crit_measure_isotropic}
and \cref{existence_crit_measure_dms}.

\subsection{Convex, bounded, and complete sets}
\label{sec:convex_bounded_and_complete_are_zeta}

\begin{remark}\label{remark_condition_zeta_implies}
    In \cref{condition_zeta}, the series is assumed to converge to an element in $\SC$.
    The {\ConditionZetaText} therefore includes a certain closedness or completeness condition.
    Furthermore, it implies that $\SC$ is bounded, as the following argument shows.
    For unbounded $\SC$, one could choose a sequence $(x_n)_{n\in \NN} \subset \SC$
    with $\norm{x_n}\ge C \, n^{\alpha}$.
    Then
    \[
      \norm{ \sum_{i=1}^n \frac{x_i}{C \, i^{\alpha}} - \sum_{i=1}^{n-1} \frac{x_i}{C \, i^{\alpha}} }
      = \norm{x_n} C^{-1} n^{-\alpha}
      \geq 1
      ,
    \]
    and so the partial sums do not converge.
\end{remark}

The following result shows that {\ConditionZetaText} generalizes the conditions of being a complete, bounded, and convex set.

\begin{lemma}\label{lem:ConvexSetsAreQuasiConvex}
  Let $(\XX, \norm{\cdot})$ be a quasi-normed space
  and let $\SC \subset \XX$ be complete, bounded, and convex.
  Then $\SC \subset \XX$ satisfies \ConditionZetaText.
\end{lemma}

\begin{proof}
    We prove the claim by checking the definition of \ConditionZetaText.
    By \zcref{thm:Aoki_Rolewicz} there exists $p\in (0,1]$ and a $p$-norm $\norm{\cdot}_{\ast}$ that is equivalent to $\norm{\cdot}$.
    Let $C\ge1$ denote the corresponding bi-Lipschitz constant.
    Let $\alpha>\frac{1}{p}$ and let $C_0:= \sum_{k=1}^{\infty} k^{-\alpha}$.
    Since $p\le 1$, we have $\alpha>1$ and so $C_0<\infty$.
    We will show that $\SC$ satisfies Condition~$\bigl(\text{\ref{condition_zeta}}(\alpha,C_0)\bigr)$.

    Let $(x_k)_{k\in \NN} \subset \SC$.
    For $n\in \NN$ let 
    \begin{equation*}
      s_n
      := \sum_{k=1}^n
           \frac{x_k}{C_0 \, k^{\alpha}}
      \in \XX
    \end{equation*}
    denote the corresponding partial sums.
    Let $x_0 \in \SC$ be some fixed point.
    For $n\in \NN$ let ${w_n:= 1- \frac{1}{C_0}\sum_{k=1}^{n} k^{-\alpha}}$.
    Then $w_n \in (0,1)$ and $w_n + \sum_{k=1}^{n} \frac{1}{C_0k^{\alpha}} =1$.
    Since $\SC$ is convex, it follows that
    \begin{equation*}
        w_n x_0 + s_n = w_n x_0 + \sum_{k=1}^n \frac{x_k}{C_0 k^{\alpha}} \in \SC 
        \quad \text{for all $n\in \NN$}.
    \end{equation*}
    This shows that the partial sums are,
    up to a shift by the null sequence $(w_n x_0)_{n\in \NN}$, in $\SC$.

    We now show that the limit is in $\SC$.
    Let $m,n\in \NN$ with $n>m$.
    Using the equivalence of the quasi-norms $\norm{\cdot}$ and $\norm{\cdot}_{\ast}$
    at (a) and at (c), and that $\norm{\cdot}_{\ast}$ is a $p$-norm at (b), we infer that
    \begin{align*}
      \norm{w_n x_0 + sn - (w_mx_0+s_m)}^p 
      &= \norm{
           w_n x_0
           + \sum_{k=1}^n \frac{x_k}{C_0 k^{\alpha}}
           - \biggl(
               w_m x_0 + \sum_{k=1}^{m} \frac{x_k}{C_0 k^{\alpha}}
             \biggr)
         }^p\\
      &\overset{(a)}{\leq} C^p \cdot
                           \norm{
                             w_n x_0
                             + \sum_{k=1}^n \frac{x_k}{C_0 k^{\alpha}}
                             - w_m x_0
                             - \sum_{k=1}^{m} \frac{x_k}{C_0 k^{\alpha}}
                           }_{\ast}^p\\
      &\overset{(b)}{\leq} C^p
                           \norm{x_0}_{\ast}^p \abs{w_n-w_m}^p 
                           + C^p \sum_{k=m+1}^n \frac{\norm{x_k}_{\ast}^p}{C_0^pk^{\alpha p}} \\
      &\overset{(c)}{\leq} C^{2p} \sup_{x\in \SC} \norm{x}^p \cdot
                           \biggl(
                             \abs{w_n-w_m}^p
                             + \frac{1}{C_0^p}\sum_{k=m+1}^n \frac{1}{k^{\alpha p}}
                           \biggr).
    \end{align*}
    Since $\SC$ is bounded, since $\lim_{n\to \infty} w_n =0$,
    and since $\sum_{k=1}^{\infty} k^{-\alpha p}<\infty$,
    we thus see that the sequence \allowbreak\mbox{$(w_n x_0 + s_n)_{n\in \NN}$}
    is a Cauchy sequence in $\SC$. 
    Since $\SC$ is complete, there exists a limit $x \in \SC$
    such that $\lim_{n\to \infty} w_n x_0+s_n = x$.
    Since $\lim_{n\to \infty} (w_n x_0) =0$, it follows that $\lim_{n\to\infty} s_n = x$.
\end{proof}

\subsection{Unit balls satisfy \texorpdfstring{{\ConditionZetaText}}{Condition (Zeta)}}
\label{sec:unit_balls_are_zeta}

In this section, we verify {\ConditionZetaText} for unit balls
in quasi-Banach spaces $(\YY,\norm{\cdot}_{\YY})$.
Later on, these unit balls will then play the role of the signal class $\SC$
and they will be considered as compactly embedded into some other quasi-Banach space
$(\XX,\norm{\cdot}_{\XX})$.
Notice that now two quasi-norms are involved.
A first quasi-norm $\norm{\cdot}_{\YY}$ that is used to define the unit ball,
and a second quasi-norm $\norm{\cdot}_{\XX}$ with respect to which the various notions
of a size of $\SC$ are measured.
However, once {\ConditionZetaText} is verified for $\SC$ with respect to $\norm{\cdot}_{\YY}$,
the continuity of the embedding of $\SC$ into $\XX$ will imply that $\SC$ also satisfies
{\ConditionZetaText} with respect to the quasi-norm $\norm{\cdot}_{\XX}$.
From then on, the quasi-norm $\norm{\cdot}_{\YY}$ can be forgotten.

In \cref{quasi_Banach_unit_balls_are_quasi_convex} we show that unit balls
in quasi-Banach spaces satisfy \ConditionZetaText.
In order to prove it, we will first show the claim in the setting of $p$-normed spaces,
see \cref{p_norm_unit_balls_are_quasi_convex}, and then lift the result
to general quasi-Banach spaces using the Aoki-Rolewicz theorem.
In \cref{quasi-convexity_preserved}, we show that {\ConditionZetaText} is preserved
under bounded linear maps.

\begin{lemma}\label{p_norm_unit_balls_are_quasi_convex}
    Let $(\YY,\norm{\cdot})$ be a complete $p$-normed space.
    Then, for every $\alpha > \frac{1}{p}$, the unit ball $\ball\bigl(0,1 \mid \norm{\cdot}\bigr)$
    satisfies \ConditionZetaTextAlphaC{(\alpha,C)} with
    $C:= \bigl( \sum_{n=1}^{\infty} n^{-\alpha p}\bigr)^{\frac{1}{p}} \in (0,\infty)$.
\end{lemma}

\begin{proof}
    Let $\SC:= \ball\bigl(0,1\mid \norm{\cdot}\bigr)$.
    Let $\alpha>\frac{1}{p}$ be arbitrary and let $C:= \bigl( \sum_{n=1}^{\infty} n^{-\alpha p}\bigr)^{\frac{1}{p}}$.
        Notice that $C\in (0,\infty)$ since $\alpha p >1$.
    Let $(y_n)_{n\in \NN} \subset \SC$.
    Then for all $M,N\in \NN$ with $M\le N$, we have
    \begin{align}
        \norm{\sum_{n=M}^N \frac{y_n}{C n^{\alpha}}}^p
        \le\sum_{n=M}^N \frac{\norm{y_n}^p}{C^pn^{\alpha p}}
        \le \frac{1}{C^p} \sum_{n=M}^{N} \frac{1}{n^{\alpha p}}.
        \label{eqn_2_1}
    \end{align}
    Since $\alpha>\frac{1}{p}$, this implies that
    $\bigl(\sum_{n=1}^N \frac{y_n}{C n^{\alpha}}\bigr)_{N\in \NN}$ is a Cauchy sequence
    with respect to $\norm{\cdot}$.
    Since $(\YY,\norm{\cdot})$ is complete,
    the limit $y:=\sum_{n=1}^{\infty} \frac{y_n}{C n^{\alpha}} \in \YY$ exists.

    Let us check whether $y \in \SC$.
    Recall that $\norm{\cdot}^p$ is a metric.
    Hence, $\norm{\cdot}$ is continuous with respect to its induced topology.
    Using \eqref{eqn_2_1} with $M=1$, we thus obtain
    \begin{equation*}
        \norm{y} 
        = \lim_{N\to \infty} \norm{ \sum_{n=1}^{N} \frac{y_n}{C n^{\alpha}} }
        \le \liminf_{N\to \infty} \frac{1}{C} \cdot\biggl(\sum_{n=1}^{N} \frac{1}{n^{\alpha p}}\biggr)^{\frac{1}{p}}
        = \frac{1}{C} \cdot\biggl(\sum_{n=1}^{\infty} \frac{1}{n^{\alpha p}}\biggr)^{\frac{1}{p}}=1.
    \end{equation*}
    Hence, $y\in \SC$.
    This shows that $\SC$ satisfies \ConditionZetaTextAlphaC{(\alpha,C)}.
\end{proof}

\begin{lemma}\label{quasi_Banach_unit_balls_are_quasi_convex}
    Let $(\YY,\norm{\cdot})$ be a quasi-Banach space with modulus of concavity $K$.
    Let \linebreak\mbox{$\alpha> 1 + \log_2(K)$}. 
    Then there exists $C>0$ such that the closed unit ball $\ball\bigl(0,1\mid \norm{\cdot}\bigr)$
    satisfies \ConditionZetaTextAlphaC{(\alpha,C)}.
\end{lemma}

\begin{proof}
    By \cref{thm:Aoki_Rolewicz}, there exists a $p$-norm $\norm{\cdot}_{\ast}$
    with $\frac{1}{p} = 1 + \log_2(K)$ and constants $c_0,C_0>0$ such that
    \begin{equation*}
        c_0\cdot\norm{y}_{\ast} 
        \le \norm{y} 
        \le C_0\cdot \norm{y}_{\ast}
    \qquad
    \text{for all } y\in \YY.
    \end{equation*}
    Let $\SC := \ball\bigl(0,1 \mid \norm{\cdot}\bigr)$
    and $\SC_{\ast} := \ball\bigl(0,1\mid \norm{\cdot}_{\ast}\bigr)$.
    Let $\alpha > 1 + \log_2(K) = \frac{1}{p}$.
    By \cref{p_norm_unit_balls_are_quasi_convex}, there exists $C_1>0$ such that
    $\ball\bigl(0,1 \mid \norm{\cdot}_{\ast}\bigr)$ satisfies \ConditionZetaTextAlphaC{(\alpha,C_1)}.

    Let $(y_n)_{n\in \NN} \subset \SC$.
    Then 
    \begin{equation*}
        \norm{c_0 y_n}_{\ast} \le \norm{y_n} \le 1
        \quad \text{for all $n\in \NN$}.
    \end{equation*}
    Hence $\bigl( c_0 y_n\bigr)_{n\in \NN}\subset \SC_{\ast}$.
    Since $\SC_{\ast}$ satisfies \ConditionZetaTextAlphaC{(\alpha, C_1)}, it follows that
    $\sum_{n=1}^{\infty} \frac{c_0 y_n}{C_1 n^{\alpha}}$ converges in $\norm{\cdot}_{\ast}$ to some element in $\SC_{\ast}$.
    We deduce that the series also converges in $\norm{\cdot}$ and that
    \begin{equation*}
        \norm{\sum_{n=1}^{\infty} \frac{c_0 y_n}{C_0 C_1 n^{\alpha}}} 
        \le \norm{\sum_{n=1}^{\infty} \frac{c_0 y_n}{C_1 n^{\alpha}}}_{\ast} 
        \le 1.
    \end{equation*}
    This shows that $\SC$ satisfies \ConditionZetaTextAlphaC{(\alpha,\frac{C_0C_1}{c_0})}.
\end{proof}

\begin{lemma}
    \label{quasi-convexity_preserved}
    Let $(\YY, \norm{\cdot}_{\YY})$ be a quasi-normed space and let $\SC \subset \YY$
    satisfy \ConditionZetaTextAlphaC{(\alpha,C_0)} for some $\alpha,C_0>0$.
    \begin{enumerate}[label=(\roman*)]
      \item \label{quasi-convexity_preserved_item_lin_ops}
            Let $(\XX, \norm{\cdot}_{\XX})$ be a quasi-normed space
            and let $\Phi \colon \YY \to \XX$ be a bounded linear map.
            Then $\Phi(\SC)$ satisfies \ConditionZetaTextAlphaC{(\alpha,C_0)}.

        \item \label{quasi-convexity_preserved_item_translations}
              Assume that $\frac{1}{C_0} \sum_{n=1}^{\infty} \frac{1}{n^{\alpha}} =1$.
              Then $y+ c\SC$ satisfies \ConditionZetaTextAlphaC{(\alpha,C_0)}
              for all $c>0$ and all $y\in \YY$.
    \end{enumerate}
\end{lemma}

\begin{proof}
    We show $\ref{quasi-convexity_preserved_item_lin_ops}$.
    Let $(x_i)_{i\in \NN} \subset \Phi(\SC)$.
    Then there exists $(y_i)_{i\in \NN} \subset \SC$ such that $x_i = \Phi(y_i)$ for all $i\in \NN$.
    Since $\SC$ satisfies \ConditionZetaTextAlphaC{(\alpha,C_0)} in $\YY$, we have
    \begin{equation*}
        \sum_{i=1}^n \frac{y_i}{C_0 i^{\alpha}} 
        \xrightarrow[(n\to\infty)]{\norm{\cdot}_{\YY}} \sum_{i=1}^{\infty} \frac{y_i}{C_0 i^{\alpha}} =: y \in \SC.
     \end{equation*}
     Using the linearity of $\Phi$ and its continuity, we infer that
     \begin{equation*}
        \sum_{i=1}^n \frac{x_i}{C_0 i^{\alpha}} 
        = \sum_{i=1}^n \frac{\Phi(y_i)}{C_0 i^{\alpha}}
        =\Phi\biggl( \sum_{i=1}^n \frac{y_i}{C_0 i^{\alpha}} \biggr)
        \xrightarrow[(n\to\infty)]{\norm{\cdot}_{\XX}} \Phi(y) =:x \in \Phi(\SC).
     \end{equation*}
     This shows that $\Phi(\SC)$ satisfies \ConditionZetaTextAlphaC{(\alpha,C_0)} in $\XX$.

     We show $\ref{quasi-convexity_preserved_item_translations}$.
     Let $y\in \YY$ and let $c>0$.
     Let $(y_i)_{i\in \NN} \subset y+c\SC$. 
     For $i\in \NN$ let $z_i := \frac{y_i-y}{c}$.
     Then $(z_i)_{i\in \NN} \subset \SC$ and we have
     \begin{equation*}
         \sum_{i=1}^n \frac{y_i}{C_0i^{\alpha}}
         = \sum_{i=1}^n \frac{y}{C_0i^{\alpha}}
         +c\sum_{i=1}^n \frac{z_i}{C_0i^{\alpha}}
     \end{equation*}
     for all $n\in \NN$.
     Since $\sum_{i=1}^{\infty} \frac{1}{C_0i^{\alpha}}=1$, the first sum on the right-hand side converges to $y$.
     Since $\SC$ satisfies \ConditionZetaTextAlphaC{(\alpha,C_0)}, the second sum on the right-hand side converges to some element in $\SC$.
     Hence the sum on the left-hand side converges to some element in $y+ c\SC$.
     This shows that $y + c \SC$ satisfies \ConditionZetaTextAlphaC{(\alpha,C_0)}.
\end{proof}

\begin{corollary}\label{continuously_embedded_balls_are_zeta}
    Let $(\YY, \norm{\cdot}_{\YY})$ and $(\XX,\norm{\cdot}_{\XX})$ be quasi-Banach spaces with $\YY \embeds \XX$.
    Then any closed ball $\ball\bigl(0,r\mid \YY\bigr)$ of radius $r > 0$ centered at $0$
    considered as a subset of $\XX$ satisfies \ConditionZetaText.
\end{corollary}

\begin{proof}
  Let $r>0$.
  By \cref{quasi_Banach_unit_balls_are_quasi_convex},
  the unit ball $\ball(0,1\mid \YY)$ satisfies {\ConditionZetaText} in $\YY$.
  Since dilations are bounded linear maps, part \ref{quasi-convexity_preserved_item_lin_ops}
  of \cref{quasi-convexity_preserved} implies that the ball $\ball(0,r\mid \YY)$
  satisfies {\ConditionZetaText} in $\YY$.
  Since the embedding $\iota \colon \YY \embeds \XX$ is a bounded linear map,
  part \ref{quasi-convexity_preserved_item_lin_ops} of \cref{quasi-convexity_preserved}
  implies that $\iota\bigl( \ball(0,r\mid \YY)\bigr) \subset \XX$ satisfies {\ConditionZetaText}.
\end{proof}

\subsection{Application to Besov- and Triebel-Lizorkin spaces}
\label{sec_unit_Balls_B_F_are_zeta}

\subsubsection{
  Proof of \texorpdfstring{
             \Cref{existence_crit_measure_isotropic}
           }{
              Proposition~\ref{existence_crit_measure_isotropic}
           }
}
\label{sec_proof_of_existence_crit_measure_isotropic}

We recall the setting from \cref{existence_crit_measure_isotropic} and introduce some notation.
Let $\emptyset \ne \Omega \subset \RR^d$ be an open and bounded subset.
Let $A_1,A_2\in \{B,F\}$.
Let $s_1,s_2\in \RR$ and let $0<p_1,q_1,p_2,q_2 \le \infty$ (with $p_i<\infty$ if $A_i=F$).
Assume that 
\begin{equation*}
   s_1 > s_2 
   \quad \text{and} \quad
    s_1 - \frac{d}{p_1} > s_2 -\frac{d}{p_2}.
\end{equation*}
Let 
\begin{equation*}
   \YY := (A_1)^{s_1}_{p_1,q_1}(\Omega),
   \quad 
   \text{and}
   \quad 
   \XX := (A_2)^{s_2}_{p_2,q_2}(\Omega).
\end{equation*}
Furthermore, let
\begin{equation*}
\SC := \bigl\{ y \in \YY : \norm{y}_{\YY}\le 1\bigr\}.
\end{equation*}

The following lemma is well-known.
For a proof and relevant references, we refer to \cref{sec_proof_of_minkowski_dim_isotropic}.
\begin{lemma}
    \label{minkowski_dim_isotropic}
The spaces $\bigl(\XX,\norm{\cdot}_{\XX}\bigr)$ and $\bigl(\YY,\norm{\cdot}_{\YY}\bigr)$ from above are quasi-Banach spaces.
    Furthermore, we have a compact embedding $\YY \embeds \XX$ and the subset $\SC\subset\XX$ from above is a non-empty totally bounded subset with 
    \begin{equation*}
        \md(\SC) = \frac{d}{s_1-s_2}.
    \end{equation*}
\end{lemma}

\begin{proof}[Proof of \cref{existence_crit_measure_isotropic}]
That $\md(\SC) = \frac{d}{s_1-s_2}$ is already shown in \cref{minkowski_dim_isotropic}.
In order to show that a critical measure exists,
we will invoke \cref{equivalence_lower_minkowski_existence_of_measure}. 
In the following, we check its prerequisites.
By \cref{minkowski_dim_isotropic}, the embedding $\YY \embeds \XX$ is compact.
In particular, it is continuous.
By \cref{continuously_embedded_balls_are_zeta}, $\SC$ satisfies {\ConditionZetaText} as a subset of $\XX$.
Hence, \cref{equivalence_lower_minkowski_existence_of_measure} is applicable.
\end{proof}

\subsubsection{
  Proof of \texorpdfstring{
             \Cref{existence_crit_measure_dms}
           }{
              Proposition~\ref{existence_crit_measure_dms}
           }
}
\label{sec_proof_of_existence_crit_measure_dms}

We recall the setting from \cref{existence_crit_measure_dms} and introduce some notation.
 Let $d\in \NN$ and let $\emptyset \ne \Omega \subset \RR^d$ be an arbitrary bounded domain%
 \footnote{In the sense of \cite[Section 3.1]{vybiralFunctionSpacesDominating2006}.}.
 Let $A_1,A_2 \in \{B,F\}$, let $s_1,s_2\in \RR$ and let $0<p_1,p_2,q_1,q_2\le \infty$ (with $p_i<\infty$ if $A_i=F$).
 Assume that
    \begin{equation*}
        s_1 > s_2 
        \quad \text{and} \quad 
        s_1 - \frac{1}{p_1} > s_{2} - \frac{1}{p_{2}}.
    \end{equation*}
Let 
\begin{equation*}
   \YY := S^{s_1}_{p_1,q_1}A_1(\Omega),
   \quad 
\SC := \bigl\{ y \in \YY : \norm{y}_{\YY}\le 1\bigr\},
   \quad 
   \text{and}
   \quad 
   \XX := S^{s_2}_{p_2,q_2}A_2(\Omega).
\end{equation*}

The following lemma is well-known.
For a proof and relevant references, we refer to \cref{sec_proof_of_minkowski_dim_dms}.

\begin{lemma}\label{minkowski_dim_dms}
  The spaces $\bigl(\XX,\norm{\cdot}_{\XX}\bigr)$ and $\bigl(\YY,\norm{\cdot}_{\YY}\bigr)$
  from above are quasi-Banach spaces.
  Furthermore, we have a compact embedding $\YY \embeds \XX$ and the subset $\SC\subset \XX$
  from above is a non-empty totally bounded subset with 
  \begin{equation*}
    \md(\SC) = \frac{1}{s_1-s_2}.
  \end{equation*}
\end{lemma}

\begin{proof}[Proof of \cref{existence_crit_measure_dms}]
The identity $\md(\SC)= \frac{1}{s_1-s_2}$ follows from \cref{minkowski_dim_dms}.
In order to show the existence of a critical measure,
we will invoke \cref{equivalence_lower_minkowski_existence_of_measure}. 
In the following, we check its prerequisites.
By \cref{minkowski_dim_dms}, the embedding $\YY \embeds \XX$ is compact.
In particular, it is continuous.
By \cref{continuously_embedded_balls_are_zeta},
$\SC$ satisfies {\ConditionZetaText} as a subset of $\XX$.
Hence, \cref{equivalence_lower_minkowski_existence_of_measure} is applicable.
\end{proof}

%% file: parts/technical_proofs_introduction.tex
\section{Postponed proofs from the introduction related to non-linear approximation}
\label{sec:TechnicalProofsIntroduction}

\subsection{
  Proof of \texorpdfstring{
             \Cref{thm:RateDistortionTheoryPhaseTransition}
           }{
              Theorem~\ref{thm:RateDistortionTheoryPhaseTransition}
           }
}
\label{sec:proof_of_thm:RateDistortionTheoryPhaseTransition}

\begin{proof}[Proof of \cref{thm:RateDistortionTheoryPhaseTransition}]
  Let $k \geq 1$ be a "triangle constant" for $d$.

  \medskip{}

  \emph{(i)}
  The claim is vacuously satisfied if $\CR^{\ast}=0$.
  Hence, we can assume that $\CR^{\ast}>0$.
  For $n \in \N$, define
  \[
    \delta_n
    := \inf \bigl\{ \delta(E_n, D_n) \,\,:\,\, (E_n, D_n) \text{ is an $n$-bit codec for } \SC \bigr\}
    .
  \]
  By assumption of the theorem, $\SC$ is bounded, which easily implies that $\delta_n < \infty$.
  By definition of $\delta_n$, we can choose for each $n \in \N$ an $n$-bit codec
  $(E_n^\ast, D_n^\ast)$ for $\SC$ satisfying $\delta(E_n^\ast, D_n^\ast) \leq 2^{-n} + \delta_n$.

  Now, let $0 \leq \CR < \CR^\ast$ be arbitrary and choose $C_{\CR} > 0$ with $2^{-n} \leq C_{\CR} \cdot n^{-\CR}$
  for all $n \in \N$.
  Note by \cref{lem:CompressionRateVSMinkowskiDimension} that indeed
  $\CR^\ast 
  = \bigl(\umd(\SC)\bigr)^{-1} 
  = \CR^\ast (\SC, \XX)$.
  By definition of $\CR^\ast(\SC, \XX)$, there
  thus exists $\tau \in (\CR, \CR^\ast]$ and a codec sequence $\mathscr{S} = \bigl((E_n, D_n)\bigr)_{n\in \NN}$
  for $\SC$ that achieves rate $\tau$, meaning there is a constant $C > 0$ such that
  $\delta(E_n, D_n) \leq C \cdot n^{- \tau} \leq C \cdot n^{-\CR}$ for all $n \in \N$.
  By choice of $(E_n^\ast, D_n^\ast)$ and $\delta_n$, this implies
  \[
    \delta(E_n^\ast, D_n^\ast)
    \leq 2^{-n} + \delta_n
    \leq 2^{-n} + \delta(E_n, D_n)
    \leq 2^{-n} + C \cdot n^{-\CR}
    \leq (C + C_{\CR}) \cdot n^{-\CR}
  \]
  for all $n \in \N$.
  This proves Part~(i).

  \medskip{}

  \emph{(ii)}
  The claim is vacuously true if $\CR^{\sharp}=\infty$. Hence, assume that $\CR^{\sharp}<\infty$.
  Let $\CR \in (\CR^\sharp, \infty) \subset (0,\infty)$.
  Since $\CR^\sharp = (\lmd(\SC))^{-1}$, this implies $\frac{1}{\CR} < \md(\SC)$.
  Hence, by definition of a critical measure, $\mu$ satisfies the small-ball condition of order
  $\sigma := \frac{1}{\CR}$.
  Hence, there exist $c_0 = c_0(\CR) > 0$ and $r_0 = r_0 (\CR) > 0$ such that
  \begin{equation}
    \mu^\ast \bigl(\ball(x, r)\bigr)
    \leq \exp\bigl(-c_0 \cdot r^{-1/\CR} \bigr)
    \qquad \forall \, x \in \XX \text{ and } 0 < r < r_0
    .
    \label{eq:RateDistortionPhaseTransitionProofSmallBallProperty}
  \end{equation}

  Set $\eps_0 := r_0$ and $c := c_0 \log_2(e)> 0$.
  Let $n \in \N$, let $(E_n, D_n)$ be an $n$-bit codec for $\SC$,
  and let $\eps \in (0, \eps_0)$.
  Let
  \[
    M
    := \bigl\{
         x \in \SC
         \,\,:\,\,
         d\bigl(x, D_n(E_n(x))\bigr) \leq \eps
       \bigr\}
  \]
  and write
  \[
    \mathrm{Image}(D_n) = \{ x_1, \dots, x_{2^n} \}
  \]
  with (not necessarily pairwise distinct) $x_i \in \XX$.
  Then 
  \[
    M \subset \bigcup_{i=1}^{2^n} \ball(x_i,\eps).
  \]
  By the subadditivity of $\mu^\ast$ and \cref{eq:RateDistortionPhaseTransitionProofSmallBallProperty},
  this implies
  \[
    \mu^\ast (M)
    \leq \sum_{i=1}^{2^n} \mu^\ast \bigl(\ball(x_i, \eps)\bigr)
    \leq 2^n \cdot \exp\bigl(-c_0 \cdot \eps^{-1/\CR}\bigr)
    =    2^{n - c_0\log_2(e) \cdot \eps^{-1/\CR}}
    .
  \]
  Since $c = c_0\log_2(e)> 0$, this proves the claim of Part~(ii).

  \medskip{}

  \emph{(iii)}
  The claim is vacuously true if $\CR^{\sharp}=\infty$. Hence, assume that $\CR^{\sharp}<\infty$.
  Let $\CR \in (\CR^\sharp, \infty)$, pick $\sigma \in (\CR^\sharp, \CR) \subset (\CR^\sharp, \infty)$,
  and choose $c = c(\sigma) > 0$ and $\eps_0 = \eps_0 (\sigma) > 0$ as provided by Part~(ii),
  applied to $\sigma$ instead of $\CR$.
  Let $\mathscr{S} = ( (E_n, D_n) )_{n \in \N}$ be a codec sequence for $\SC$.
  For each $m \in \N$, let
  \[
    \CalR_m^\CR
    := \bigl\{
         x \in \SC
         \,\,:\,\,
         \forall \, n \in \N : d\bigl(x, D_n(E_n(x))\bigr) \leq m \cdot n^{-\CR}
       \bigr\}
    .
  \]
  Since $\CalR^\CR (\mathscr{S}, \SC, \XX) = \bigcup_{m \in \N} \CalR_m^\CR$, it is enough to show that
  $\mu^\ast (\CalR_m^\CR) = 0$ for every $m \in \N$.
  Thus, fix $m \in \N$.

  There exists $n_0 = n_0 (\CR,m)$ with $m \cdot n^{-\CR} < \eps_0$ for all $n \geq n_0$.
  Hence, since
  \[
    \CalR_m^\CR
    \subset 
              \bigl\{ x \in \SC \,\,:\,\, d\bigl(x, D_n(E_n(x))\bigr) \leq m \cdot n^{-\CR} \bigr\}
    ,
  \]
  we get by Part~(ii) (for $\sigma$ instead of $\CR$) for every $n \geq n_0$ that
  \begin{align*}
    \mu^\ast (\CalR_m^\CR)
    & \leq \mu^\ast \Bigl(\bigl\{ x \in \SC \,\,:\,\, d\bigl(x, D_n(E_n(x))\bigr) \leq m \cdot n^{-\CR} \bigr\}\Bigr) \\
    & \leq 2^{n - c \cdot (m \cdot n^{-\CR})^{-1/\sigma}}
      =    2^{n - c \cdot m^{-1/\sigma} n^{\CR/\sigma}}
      \xrightarrow[n\to\infty]{} 0
    .
  \end{align*}
  Here, the last step used that $\CR/\sigma > 1$.

  For the proof of the final part, note that $\CalR^\CR (\mathscr{S}, \SC, \XX) \subset \CalR^s (\mathscr{S}, \SC, \XX)$ if $\CR > s$.
  This implies that
  \[
    \bigcup_{\CR \in (\CR^\sharp, \infty)} \CalR^\CR (\mathscr{S}, \SC, \XX)
    = \bigcup_{n \in \N} \CalR^{\CR^\sharp + 1/n} (\mathscr{S}, \SC, \XX)
  \]
  is a $\mu^\ast$-null set, as a countable union of $\mu^\ast$-null sets.
\end{proof}

\subsection{
  Proof of \texorpdfstring{
             \Cref{thm:PhaseTransitionForFamiliesWithControlledComplexity}
           }{
              Theorem~\ref{thm:PhaseTransitionForFamiliesWithControlledComplexity}
           }
}
\label{sec:PhaseTransitionForControlledComplexityFamiliesProof}

\begin{proof}[Proof of \Cref{thm:PhaseTransitionForFamiliesWithControlledComplexity}]
  \textbf{Proof of Part~(i):}
  Consider the codec sequence ${\mathscr{S}^\ast \!=\! ( (E_n^\ast, D_n^\ast) )_{n \in \N}}$
  from Part~(i) of \Cref{thm:RateDistortionTheoryPhaseTransition}.
  Set $\CalA_n^\ast := \mathrm{Image}(D_n^\ast) \subset \XX$.
  Then the bound in Part~(i) of \Cref{thm:RateDistortionTheoryPhaseTransition} directly implies
  for arbitrary $0 \leq \AR < \AR^\ast$ that
  \[
    \sup_{x \in \SC} \,\, \dist(x, \CalA_n^\ast)
    = \sup_{x \in \SC} \,\, \inf_{y \in \mathrm{Image}(D_n^\ast)} d(x, y)
    \leq \sup_{x \in \SC} d\bigl(x, D_n^\ast(E_n^\ast(x))\bigr)
    = \delta(E_n^\ast, D_n^\ast)
    \leq C(\AR) \cdot n^{-\AR}
    .
  \]
  Finally, we note for any $x\in \XX$ and $R>0$ that
  \[
    \cn(\ball(x,R) \cap \CalA_n^\ast, n^{-\alpha})
    \leq \# (\ball(x,R) \cap \CalA_n^\ast)
    \leq \num (\CalA_n^\ast)
    \leq 2^n
    =    \exp(\fcc(n)),
  \]
  where $\fcc(n) := \ln(2) \cdot n$.
  Hence Condition~\eqref{eq:ControlledComplexity}
  is satisfied.
  It is trivial to see that Condition~\eqref{eq:ZetaGrowthCondition} holds as well.
  Thus, $(\CalA_n^\ast)_{n \in \N}$ is a family of controlled complexity.

  \bigskip{}

  \textbf{Proof of Part~(ii):}
  Suppose that \eqref{eq:FamilyApproximationRate} holds for some $\AR \geq 0$.
  If $\AR = 0$, we trivially have $\AR \leq \AR^\ast$; we can thus assume $\AR > 0$.
  Let $k \geq 1$ be a triangle constant for $d$.
  Fix $\delta > 0$.
  Since $\SC \subset \XX$ is bounded, there exist $R \geq 4 k \, C(\AR) > 0$
  and $x_0 \in \XX$ with $\SC \subset \ball(x_0, \frac{R}{2 k})$.
  Let $\fcc : \N \to [0,\infty)$ and $n_0 \in \N$
  satisfy \eqref{eq:ZetaGrowthCondition} and \eqref{eq:ControlledComplexity}
  for $\alpha = \AR$ and $x = x_0$.
  By \eqref{eq:ZetaGrowthCondition}, we have $\fcc(n) \leq C_\delta \cdot n^{1 + \delta}$
  for all $n \in \N$ and some constant $C_\delta > 0$.

  Let $\eps \in (0, n_0^{-t})$ and choose $n \in \N_{\geq 2}$ with
  \[
    k \cdot (1 + 2 \, C(\AR)) \cdot n^{-\AR}
    \leq \frac{\eps}{2k}
    \leq k \cdot (1 + 2 \, C(\AR)) \cdot (n-1)^{-\AR}
    .
  \]
  Note that this implies
  \(
    n^{-t}
    \leq \eps
    < n_0^{-t}
  \)
  and thus $n \geq n_0$.

  Let $\CalN_n$ be a minimal $n^{-\AR}$-net for $\ball(x_0, R) \cap \CalA_n$.
  By \eqref{eq:ControlledComplexity}, we then have
  \[
    \# \CalN_n
    \leq \exp(\fcc(n))
    \leq \exp(C_\delta \cdot n^{1+\delta})
    .
  \]

  Now, given any $x \in \SC$, by \eqref{eq:FamilyApproximationRate} there exists
  $\tilde{x} \in \CalA_n$ with
  $d(x, \tilde{x}) \leq 2 \, C(\AR) \cdot n^{-\AR} \leq 2 \, C(\AR) \leq \frac{R}{2k}$.
  We then have $d(\tilde{x}, x_0) \leq k \cdot (d(\tilde{x}, x) + d(x, x_0)) \leq k \cdot (\frac{R}{2k} + \frac{R}{2k}) = R$
  and hence $\tilde{x} \in \ball(x_0, R) \cap \CalA_n$.
  There thus exists $\hat{x} \in \CalN_n$ with $d(\tilde{x}, \hat{x}) \leq n^{-\AR}$.
  This finally implies
  \[
    d(x, \hat{x})
    \leq k \cdot \bigl(d(x, \tilde{x}) + d(\tilde{x}, \hat{x})\bigr)
    \leq k \cdot \bigl(2 \, C(\AR) \cdot n^{-\AR} + n^{-\AR}\bigr)
    =    (1 + 2 \, C(\AR)) \cdot k \cdot n^{-\AR}
    \leq \frac{\eps}{2 k}
    .
  \]
  This shows that $\CalN_n$ is an (external) $\frac{\eps}{2k}$-net for $\SC$.
  In combination with \Cref{lem:InternalVSExternalCoveringNumbers}, this implies
  \[
    \cn(\SC, \eps)
    \leq \extcn\left(\SC, \XX, \frac{\eps}{2k}\right)
    \leq \# \CalN_n
    \leq \exp(C_\delta \cdot n^{1+\delta})
    .
  \]
  Finally, note that
  \[
    n^\AR
    \leq (1 + (n-1))^\AR
    \leq (2 \cdot (n-1))^\AR
    \leq 2^\AR \cdot (n-1)^\AR
    \leq 2^{\AR+1} k^2 (1 + 2 \, C(\AR)) \cdot \eps^{-1}
  \]
  and hence $n \leq C^\ast \cdot \eps^{-1/\AR}$ for a suitable constant $C^\ast = C^\ast (\AR,k) > 0$.

  We have thus shown for all $\eps \in (0, n_0^{-t})$ that
  \[
    \cn(\SC, \eps)
    \leq \exp\bigl( (C^\ast)^{1+\delta} C_\delta \cdot \eps^{-(1+\delta)/\AR}\bigr)
  \]
  and thus
  \[
    \frac{\log_2 \log_2 \cn(\SC, \eps)}{\log_2(1/\eps)}
    \leq \frac{\log_2 \bigl( \log_2(e) (C^\ast)^{1+\delta} C_\delta\bigr) + \frac{1+\delta}{\AR} \log_2(1/\eps)}{\log_2(1/\eps)}
    ,
  \]
  which easily implies $\umd(\SC) \leq \frac{1+\delta}{\AR}$.

  Since $\delta > 0$ was arbitrary, this shows $\umd(\SC) \leq \AR^{-1}$
  and hence as claimed that
  \[
    \AR^\ast = \bigl(\umd(\SC)\bigr)^{-1} \geq \AR
    .
  \]

  \bigskip{}

  \textbf{Proof of Part~(iii):}
  Let $k \geq 1$ be a triangle constant for the quasi-metric $d$ on $\XX$.
  The claim is vacuously satisfied if $\AR^\sharp = \infty$; hence, we can assume in the
  following that $\AR^\sharp < \infty$.

  Fix $x_0 \in \SC$.
  Since $\SC$ is bounded, there exists $R > 0$ such that $\SC \subset \ball(x_0, \frac{R}{2 k})$.
  Let $\AR>\AR^{\sharp}$ and choose $\sigma \in (\AR^\sharp, \AR)$.
  Since $0 \leq \AR^{\sharp} < \sigma < \AR$, we can write $\frac{\AR}{\sigma} = 1 + \delta$ with
  $\delta > 0$.

  The idea of the proof is to construct a codec sequence
  $\mathscr{S} = \bigl( (E_m, D_m) \bigr)_{m \in \N}$ for $\SC$ such that
  \begin{equation}
    \CalR^\AR \bigl( (\CalA_n)_{n \in \N}, \SC, \XX\bigr)
    \subset \CalR^\sigma (\mathscr{S}, \SC, \XX)
    .
    \label{eq:PhaseTransitionControlledComplexityMainClaim}
  \end{equation}
  Once this is shown, the claim follows from Part~(iii) of \cref{thm:RateDistortionTheoryPhaseTransition}.

  For the construction of $\mathscr{S}$, let $n_0 \in \N$
  and $\fcc = \fcc_{\alpha,x_0,R} : \N \to [0,\infty)$ be as in
  \cref{def:ControlledComplexity}, for $\alpha = \AR$ and $x = x_0$.
  By assumption \eqref{eq:ZetaGrowthCondition} on $\fcc$,
  we have $\fcc (n) \leq C_\delta \cdot n^{1+\delta}$ for all $n \in \N$,
  for a suitable constant $C_\delta \in (0,\infty)$.
  Let $m_0 \in \N$ such that $\frac{\ln(2)}{C_\delta} m \geq n_0^{1 + \delta}$ for $m \geq m_0$.

  For $m < m_0$, let $E_m : \SC \to \{ 0,1 \}^m$ be arbitrary, and define $D_m : \{ 0,1 \}^m \to \SC$
  via $D_m (b) = x_0$ for all $b \in \{ 0,1 \}^m$.
  For $m \in \N_{\geq m_0}$,
  \begin{equation}
    \text{let $n = n(m) \in \N$ be maximal with } n^{1+\delta} \leq \frac{\ln(2)}{C_\delta} \cdot m
    .
    \label{eq:ControlledComplexityPhaseTransitionChoiceOfN}
  \end{equation}
  This is possible, since the set of allowed $n$ is clearly bounded from above,
  and is non-empty because of $\frac{\ln(2)}{C_\delta} \cdot m \geq n_0^{1 + \delta}$,
  which also implies that $n(m) \geq n_0$ for $m \geq m_0$.
  Then, by Condition~\eqref{eq:ControlledComplexity} of \Cref{def:ControlledComplexity},
  given any $m \geq m_0$ we have
  \[
    \cn\bigl( \ball(x_0, R) \cap \CalA_n, n^{-\AR}\bigr)
    \leq \exp(\fcc(n))
    \leq \exp(C_\delta \cdot n^{1+\delta})
    \leq \exp(\ln(2) \cdot m)
    =    2^m
    .
  \]
  Hence, there exists an $n^{-\AR}$-net $\CalN_m \subset \XX$ for $\ball(x_0, R) \cap \CalA_n$
  with $\# \CalN_m \leq 2^m$.
  We now set
  \[
    \CalN_m^\ast
    := \begin{cases}
         \CalN_m \subset \ball(x_0, R),   & \text{if } \CalN_m \neq \emptyset, \\
         \{ x_0 \} \subset \ball(x_0, R), & \text{otherwise} ,
       \end{cases}
  \]
  noting that $\CalN_m^\ast \neq \emptyset$ and $\# \CalN_m^\ast \leq 2^m$.

  We can thus choose a surjection $D_m : \{ 0,1 \}^m \to \CalN_m^\ast$
  and for each $x \in \SC$ we can choose $E_m (x) \in \{ 0,1 \}^m$ with
  \begin{equation}
    d\bigl(x, D_m(E_m(x))\bigr)
    = \min_{b \in \{ 0,1 \}^m} d(x, D_m(b))
    = \min_{\hat{x} \in \CalN_m^\ast} d(x, \hat{x})
    .
    \label{eq:PhaseTransitionControlledComplexityCodecRelatesToNet}
  \end{equation}

  Now, given $x \in \CalR^\AR ( (\CalA_n)_{n \in \N}, \SC, \XX)$, the quantity
  \[
    \kappa (x)
    := \sup_{n \in \N} \bigl(n^\AR \cdot \dist(x, \CalA_n)\bigr)
    \in [0,\infty)
  \]
  is a finite constant.
  Note that by maximality of $n = n(m)$ (see \Cref{eq:ControlledComplexityPhaseTransitionChoiceOfN}),
  we have
  \[
    \frac{\ln(2)}{C_\delta} \cdot m
    < (n+1)^{1+\delta}
    \leq 2^{1+\delta} \cdot n^{1 + \delta}
  \]
  and thus, for $c_\delta := \ln(2) / (2^{1+\delta} \, C_\delta) > 0$,
  \begin{equation}
    n = n(m) \geq (c_\delta \cdot m)^{1 / (1+\delta)}
    \qquad \forall \, m \geq m_0
    .
    \label{eq:ControlledComplexityPhaseTransitionNLowerBound}
  \end{equation}
  In particular, $n(m) \to \infty$ as $m \to \infty$.
  We can thus choose $m_x \geq m_0$ with
  \begin{equation}
    n = n(m) \geq \Bigl(\frac{2 k}{R} \cdot (1 + \kappa (x))\Bigr)^{1/\AR}
    \qquad \forall \, m \geq m_x
    .
    \label{eq:ControlledComplexityPhaseTransitionMxChoice}
  \end{equation}

  \medskip{}

  Now, choose $\tilde{x} \in \CalA_n$ with
  \[
    d(x, \tilde{x})
    \leq \dist (x, \CalA_n) + n^{-\AR}
    \leq n^{-\AR} \kappa (x) + n^{-\AR}
    = \frac{1 + \kappa (x)}{n^\AR}
    \overset{\eqref{eq:ControlledComplexityPhaseTransitionMxChoice}}{\leq} \frac{R}{2 k}
    .
  \]
  This implies because of $x \in \SC \subset \ball(x_0, \frac{R}{2 k})$ that
  \[
    d(x_0, \tilde{x})
    \leq k \cdot \bigl(d(x_0, x) + d(x, \tilde{x})\bigr)
    \leq k \Bigl(\frac{R}{2 k} + \frac{R}{2 k}\Bigr)
    =    R
    ,
  \]
  and thus $\tilde{x} \in \CalA_n \cap \ball(x_0, R) \neq \emptyset$.
  Since $\CalN_m$ is an $n^{-t}$-net for $\ball(x_0, R) \cap \CalA_n$,
  this implies $\CalN_m \neq \emptyset$ and thus $\CalN_m^\ast = \CalN_m$.
  Since $\CalN_m$ is an $n^{-\AR}$-net for $\CalA_n \cap \ball(x_0, R)$,
  there thus exists $\hat{x} \in \CalN_m = \CalN_m^\ast$ with $d(\tilde{x}, \hat{x}) \leq n^{-\AR}$.
  This implies
  \begin{align*}
    d\bigl(x, D_m(E_m(x))\bigr)
    & \overset{\eqref{eq:PhaseTransitionControlledComplexityCodecRelatesToNet}}{=}
      \min_{y \in \CalN_m^\ast} d(x,y)
      \leq d(x, \hat{x})
      \leq k \cdot \bigl(d(x, \tilde{x}) + d(\tilde{x}, \hat{x})\bigr) \\
    & \leq k \cdot \left(\frac{1 + \kappa(x)}{n^\AR} + n^{-\AR}\right)
      =    \underbrace{k \cdot (2 + \kappa (x))}_{=: \tilde{\kappa}(x)} \cdot n^{-\AR} \\
    & \overset{\eqref{eq:ControlledComplexityPhaseTransitionNLowerBound}}{\leq}
        \tilde{\kappa} (x) \cdot (c_\delta \cdot m)^{-\AR / (1+\delta)}
      = \tilde{\kappa} (x) \cdot c_\delta^\sigma \cdot m^{-\sigma}
    .
  \end{align*}
  This holds for all $m \geq m_x$.
  For $m_0 \leq m < m_x$, we have $\mathrm{Image}(D_m) \subseteq \CalN_m^\ast \subset \ball(x_0, R)$
  and for $m < m_0$ we have $\mathrm{Image}(D_m) = \{ x_0 \} \subset \ball (x_0, R)$.
  Because of $x \in \SC \subset \ball(x_0, \frac{R}{2k}) \subset \ball (x_0, R)$,
  this implies for $m < m_x$ that
  \begin{align*}
    d\bigl(x, D_m(E_m(x))\bigr)
    & \leq k \cdot \Bigl(d(x,x_0) + d\bigl(x_0, D_m(E_m(x))\bigr)\Bigr) \\
    & \leq 2 k R
      =    2 k R \cdot m^\sigma \cdot m^{-\sigma}
      \leq 2 k R \cdot m_x^\sigma \cdot m^{-\sigma}
    .
  \end{align*}
  Overall, this shows that $x \in \CalR^\sigma (\mathscr{S}, \SC, \XX)$.

  We have thus shown \cref{eq:PhaseTransitionControlledComplexityMainClaim}.
  As discussed earlier, this completes the proof.
\end{proof}

\subsection{
  Proof of \texorpdfstring{
             \Cref{prop:NNApproximationOfControlledComplexity}
           }{
              Proposition~\ref{prop:NNApproximationOfControlledComplexity}
           }
}
\label{sec:NNControlledComplexityProofs}

For the proof of \Cref{prop:NNApproximationOfControlledComplexity}, we will need the
following technical lemma regarding the covering numbers of sets of neural networks.
It in turn relies on \cite[Lemma~6.1]{GrohsVoigtlaenderTheoryToPracticeGapDeepLearning},
but generalized to a different "ambient space" than $C([0,1]^d)$.

\begin{lemma}\label{lem:NNCoveringNumberBounds}
  For $d, n, L_0 \in \N$ and $C_0 \geq 1$, define
  \begin{align*}
    &\CalNN_n (C_0, L_0;\RR^d) \\
    &:= \biggl\{
           R_\varrho (\CalW)
           \,\,:\,\,
           \begin{array}{c}
                \text{$\CalW$ neural network weights with}\\ 
           \| \CalW \|_{\ell^0} \leq n, \,\,
           \| \CalW \|_{\ell^\infty} \leq C_0, \,\,
           L(\CalW) \leq L_0, \,\,
           d_{\mathrm{in}}(\CalW) = d, \,\,
           d_{\mathrm{out}}(\CalW) = 1
           \end{array}
         \biggr\}
    .
  \end{align*}
  Then for any measurable, bounded set $\Omega \subset \R^d$ and any $p \!\in\! (0,\infty]$,
  there exist constants $R \!=\! R(\Omega) \!\geq\! 1$ and $C = C(\Omega, p, d) > 0$
  such that for every $\eps \in (0,1)$, it holds that
  \[
      \cn \bigl(\CalNN_n (C_0, L_0;\RR^d), \norm{\cdot}_{L^p(\Omega)}, \eps\bigr)
    \leq \left( C \cdot L_0^4 \cdot \left( 8 d^2 R \cdot C_0 \cdot n \right)^{1+L_0} \cdot \eps^{-1} \right)^{2n}
    .
  \]
\end{lemma}

\begin{proof}
\textbf{Step~1 (Rescaling):}
Choose $R \geq 1$ with $\Omega \subseteq [-R,R]^d =: Q$ and define
\[
  T_Q \, x := 2R \cdot x + (-R,\dots,-R)^T
  \qquad \text{for } x \in \R^d
  ,
\]
noting that $T_Q : \R^d \to \R^d$ is affine-linear, and $T_Q([0,1]^d) = Q$.
In this step, we show that
\begin{equation}
  \bigl\{ f \circ T_Q \,\,:\,\, f \in \CalNN_n(C_0, L_0;\RR^d) \bigr\}
  \subseteq \CalNN_{2n}(4 d R \cdot C_0, L_0;\RR^d)
  .
  \label{eq:NeuralNetworkRescaling}
\end{equation}
To see this, let $\CalW = ( (A_1,b_1),\dots,(A_L,b_L))$ be a set of neural network weights
with $\|\CalW\|_{\ell^0} \leq n$ and $\|\mathcal{W}\|_{\ell^\infty} \leq C_0$,
as well as $L \leq L_0$ and $d_{\mathrm{in}}(\mathcal{W}) = d$,
as well as $d_{\text{out}}(\mathcal{W}) = 1$. 
Then, define
\[
  \CalW_Q
  := \Bigl(
       \left( 2R \, A_1, \,\, A_1 (-R,\dots,-R)^T + b_1 \right),
       (A_2, b_2),
       \dots,
       (A_L, b_L)
     \Bigr)
  .
\]
Note for $T_1 \, x := A_1 \, x + b_1$ that
$(T_1 \circ T_Q) x = 2R \, A_1 \, x + A_1 (-R,\dots,-R)^T + b_1$,
and thus
\[
  R_\varrho (\CalW_Q) = (R_\varrho \CalW) \circ T_Q
  .
\]
Moreover, it is easy to see that $\|2R A_1\|_{\ell^0} = \|A_1\|_{\ell^0}$ and
\[
  \left\| A_1 (-R,\dots,-R)^T \right\|_{\ell^0}
  = \left\| (-R) \cdot \sum_{j=1}^d (A_1)_{-, j} \right\|_{\ell^0}
  \leq \sum_{j=1}^d \|(A_1)_{-, j}\|_{\ell^0}
  \leq \|A_1\|_{\ell^0}
  ,
\]
which shows $\|\mathcal{W}_Q\|_{\ell^0} \leq 2 \|\mathcal{W}\|_{\ell^0} \leq 2n$.
Finally, note
\[
  \|2R A_1\|_{\ell^\infty}
  = 2R \|A_1\|_{\ell^\infty}
  \leq 2R \cdot C_0
  \leq 2dR \cdot C_0
\]
and
\[
  \left\| A_1 (-R,\dots,-R)^T + b_1 \right\|_{\ell^\infty}
  \leq \|b_1\|_{\ell^\infty} + R \sum_{j=1}^d \|(A_1)_{-, j}\|_{\ell^\infty}
   \leq C_0 + dR \, C_0
   \leq 2dR \cdot C_0,
\]
so that $\|\mathcal{W}_Q\|_{\ell^\infty} \leq 4 d R \cdot C_0$.
Overall, this proves \Cref{eq:NeuralNetworkRescaling}.

\bigskip{}

\textbf{Step~2 (Reducing to covering number bounds in $\XX = C([0,1]^d)$):}
In this step, we show that
\[
  \extcn\bigl( \CalNN_n(C_0, L_0;\RR^d), \,\, L^p(\Omega), \,\, (2R)^{d/p} \cdot \eps \bigr)
  \leq \cn\bigl( \CalNN_{2n}(4 d R \cdot C_0, L_0;\RR^d), \,\, C([0,1]^d), \,\, \eps \bigr)
  .
\]

To this end, first note for any continuous $f : \R^d \to \R$ that
\begin{equation}
  \|f\|_{L^p(\Omega)}
  \leq \|f\|_{L^p([-R,R]^d)}
  \leq (2R)^{d/p} \cdot \|f\|_{L^\infty(Q)}.
  \label{eq:NeuralNetworkCoveringNumbersNormChange}
\end{equation}
Now, to show the claim regarding the covering numbers,
let $\{f_1, \dots, f_N\} \subseteq \CalNN_{2n}(4 d R \cdot C_0, L_0;\RR^d)$
be an $\eps$-net for $\CalNN_{2n}(4 d R \cdot C_0, L_0;\RR^d)$,
with respect to $\| \cdot \|_{L^\infty([0,1]^d)}$.
Set $g_i := f_i \circ T_Q^{-1} \in C(\R^d; \R)$ for $i \in \{ 1,\dots,N \}$.
Then given $f = R_\varrho (\CalW) \in \CalNN_n(C_0, L_0;\RR^d)$,
step 1 shows that
\[
  f \circ T_Q
  = R_\varrho(\CalW_Q)
  \in \CalNN_{2n}(4 d R \cdot C_0, L_0;\RR^d)
  ,
\]
so that there exists $i \in \{1, \dots, N\}$ with $\|f \circ T_Q - f_i\|_{L^\infty([0,1]^d)} \leq \eps$.
Then
\begin{align*}
  \|f - g_i\|_{L^p(\Omega)}
  \overset{\eqref{eq:NeuralNetworkCoveringNumbersNormChange}}&{\leq} (2R)^{d/p} \cdot \|f - g_i\|_{L^\infty(Q)} \\
  \overset{Q = T_Q([0,1]^d)}&{=} (2R)^{d/p} \cdot \|(f - g_i) \circ T_Q\|_{L^\infty([0,1]^d)} \\
                            &=   (2R)^{d/p} \cdot \|f \circ T_Q - f_i\|_{L^\infty([0,1]^d)}
  \leq (2R)^{d/p} \cdot \eps.
\end{align*}
This easily implies the claim of Step~2.

\bigskip{}

\textbf{Step~3 (Completing the proof):}
Using \Cref{lem:InternalVSExternalCoveringNumbers} and Step~2,
we get (with $k \geq 1$ denoting a triangle constant for $L^p(\Omega)$):
\begin{align*}
  \cn\bigl(\CalNN_n(C_0, L_0; \RR^d), \,\, L^p(\Omega), \,\, \eps\bigr)
  &\leq \extcn \left( \CalNN_n(C_0, L_0; \RR^d), \,\, L^p(\Omega), \,\, (2R)^{d/p} \frac{\eps}{(2R)^{d/p} \cdot 2k} \right) \\
  &\leq \cn \left( \CalNN_{2n}(4 d R \cdot C_0, L_0; \RR^d), \,\, C([0,1]^d), \,\, \frac{\eps}{(2R)^{d/p} \cdot 2k} \right)
  .
\end{align*}
Using \cite[Lemma~6.1]{GrohsVoigtlaenderTheoryToPracticeGapDeepLearning},
we thus obtain for $\eps \in (0,1)$ that
\begin{align*}
  \cn\bigl(\CalNN_n(C_0, L_0; \RR^d), \,\, L^p(\Omega), \,\, \eps\bigr)
  &\leq \left(
          \frac{88 \cdot k \cdot (2R)^{d/p}}{\eps}
          \cdot L_0^4
          \cdot \left( \max\{d, 2 n\} \cdot 4 d R \cdot C_0 \right)^{1+L_0}
        \right)^{2n} \\
  &\leq \left( C \cdot L_0^4 \cdot \left( 8 d^2 R \cdot C_0 \cdot n \right)^{1+L_0} \cdot \eps^{-1} \right)^{2n},
\end{align*}
with $C = C(p, d, \Omega)$.
Here, we used that $\max\{d,2n\} \le 2dn$.
\end{proof}

With this preparation, we can prove \Cref{prop:NNApproximationOfControlledComplexity}.

\begin{proof}[Proof of \Cref{prop:NNApproximationOfControlledComplexity}]
  In the notation of \Cref{lem:NNCoveringNumberBounds}, we have
  \[
    \CalA_n
    := \CalNN_n (g, \R^d)
    = \CalNN_n \bigl(g \, n^g, \lfloor (\ln(en))^g \rfloor; \RR^d\bigr)
    .
  \]
  Let $k\ge 1$ be a triangle constant for $L^p(\Omega)$. 
  Then, using \Cref{lem:InternalVSExternalCoveringNumbers,lem:NNCoveringNumberBounds}, we see
  \begin{align*}
    \cn \bigl(\ball(x,R) \cap \CalA_n, L^p(\Omega), n^{-\alpha}\bigr)
    &\leq \extcn \bigl(\ball(x,R) \cap \CalA_n, L^p(\Omega), n^{-\alpha} / (2 k)\bigr) \\
    &\leq \extcn \bigl(\CalA_n, L^p(\Omega), n^{-\alpha} / (2 k)\bigr)
     \leq \cn \bigl(\CalA_n, L^p(\Omega), n^{-\alpha} / (2 k)\bigr) \\
    &\leq \bigl(C \cdot (\ln(en))^{4 g} \cdot (8 d^2 R g \, n^{g+1})^{1 + (\ln(en))^g} \cdot 2k \, n^\alpha\bigr)^{2n}
  \end{align*}
  for certain constants $C = C(\Omega,d,p) > 0$ and $R=R(\Omega) \geq 1$.
  Hence,
  \begin{align*}
    &\log_2 \Bigl(\cn \bigl(\ball(x,R) \cap \CalA_n, L^p(\Omega), n^{-\alpha}\bigr)\Bigr) \\
    &\leq 2 n \cdot \Bigl(
                      \log_2 (C)
                      + 4 g \log_2 (\ln(en))\\
    &\hspace{6em}+ \bigl(1 + (\ln(en))^g\bigr) \cdot \log_2 (8 d^2 g R \, n^{g+1})
                      + \log_2 (2k)
                      + \alpha \, \log_2 (n)
                    \Bigr) \\
    &\leq 2 n \cdot \Bigl(
                     \widetilde{C}
                     + \widetilde{C} \ln(en)
                     + \bigl(\widetilde{C} \, \ln(en)\bigr)^g \cdot \bigl(\ln(8 d^2 g R) + (1+g) \, \ln(n)\bigr)
                     + \widetilde{C} \ln(en)
                    \Bigr) \\
    &\leq \widetilde{\widetilde{C}} \cdot n \cdot (\ln(en))^{g+1}
      =:  \fcc (n) ,
  \end{align*}
  for certain constants $\widetilde{C}, \widetilde{\widetilde{C}}$ that only depend on $p,d,g,R,C$.
  Hence, Condition~\eqref{eq:ControlledComplexity} is satisfied, and it is easy to see
  that Condition~\eqref{eq:ZetaGrowthCondition} holds as well.
\end{proof}

\subsection{Proof of \texorpdfstring{\Cref{prop:NNsOptimalForCr}}{Proposition~\ref{prop:NNsOptimalForCr}}}
\label{sec:NNsOptimalForCrProof}

\begin{proof}[Proof of \Cref{prop:NNsOptimalForCr}]
  We have $\| \cdot  \|_{L^1([0,1]^d)} \leq \| \cdot \|_{L^p([0,1]^d)}$ thanks
  to (e.g.) Hölder's inequality.
  Therefore, \cite[Theorem~3]{ClementsEntropiesOfSeveralSetsOfFunctions} shows
  that there exists a constant $c=c(d,r)>0$ such that
  \[
    \log_2 \bigl(\cn\bigl(\SC, L^p([0,1]^d),\eps\bigr)\bigr)
    \geq \log_2 \bigl(\cn\bigl(\SC, L^1([0,1]^d), \eps\bigr)\bigr)
    \geq c \cdot (1/\eps)^{d / r}
    \qquad \text{for } \eps \in (0,1)
    ,
  \]
  and this easily implies $\lmd(\SC) \geq \frac{d}{r}$ and thus
  $\AR^{\sharp} := (\lmd(\SC))^{-1} \leq \frac{r}{d}$.

  Conversely, \cite[Theorem~5]{SchimdtHieberNonparametricRegressionUsingReLUNNs}
  combined with \cite[Lemma~F.1]{PetersenVoigtlaenderOptimalApproximationUsingNNs}
  shows that there exist constants $C_0,\dots,C_3 \geq 1$ only depending on $d,m,\alpha$
  such that for every $\tau \in \NN$, every $N \in \NN_{\ge C_0}$ and every $f \in \SC$,
  there exists a neural network $\Phi_{f,\tau,N}$ of depth at most $C_1 \cdot \tau$,
  with at most $C_2 \cdot \tau N$ non-zero weights and all weights bounded in absolute value by $1$
  such that
  \[
    \| f - \Phi_{f,\tau,N} \|_{L^p ([0,1]^d)}
    \leq \| f - \Phi_{f,\tau,N} \|_{L^\infty ([0,1]^d)}
    \leq C_3 \cdot (2^{-\tau} N + N^{-r/d})
    .
  \]

  Now, let $n_0 = n_0 (C_0,\dots,C_3, d, m, \alpha) \in \N_{\ge2}$ sufficiently large
  (chosen below) and let $n \geq n_0$.
  Pick
  \[
    N := \Bigl\lfloor \frac{n}{C_2 \cdot (\log_2 (n))^2} \Bigr\rfloor
    \qquad \text{and} \qquad
    \tau := \bigl\lfloor (\log_2 (n))^2 \bigr\rfloor
    ,
  \]
  where $n_0$ is chosen so large that $N \geq C_0 \geq 1$ and $\tau \geq 1$.
  Then the network $\Phi_{f,n} := \Phi_{f,\tau,N}$ from above satisfies the following:
  \begin{itemize}
    \item All its weights are bounded in absolute value by $1 \leq 3 \cdot n^3 = g \cdot n^g$ for $g = 3$;
    \item The network has at most $C_2 \cdot \tau N \leq n$ non-zero weights;
    \item The network has at most $C_1 \, \tau \le C_1 \cdot (\log_2 (n))^2 \leq (\ln(en))^3 = (\ln(en))^g$ layers.
          Here, the second step holds for $n_0$ sufficiently large.
  \end{itemize}
  Overall, this shows that $\Phi_{f,n} \in \CalA_n = \CalNN_n (3, \R^d)$.
  Finally, by choice of $\tau, N$ we have on the one hand
  $2N \geq N + 1 \geq \frac{n}{C_2 \cdot (\log_2(n))^2}$
  and on the other hand $\tau \geq (\log_2(n))^2 - 1$.
  Since furthermore $N \leq n$, we thus see
  \begin{align*}
    \| f - \Phi_{f,n} \|_{L^p ([0,1]^d)}
    &\leq C_3 \cdot \bigl(2^{-\tau} N + N^{-r/d}\bigr) \\
    &\leq C_3 \cdot \Bigl(
                      n \cdot 2^{1 - (\log_2 (n))^2} 
                      + 2^{r/d}C_2^{r/d} \cdot \bigl(n / (\log_2(n))^2\bigr)^{-r/d}
                    \Bigr) \\
    &\leq C_4 \cdot \bigl(n / (\log_2(n))^2\bigr)^{-r/d}
    .
  \end{align*}
  Here, the last step holds for a sufficiently large choice of $n_0$ and $C_4$,
  since $2^{-(\log_2(n))^2}$ decays faster than any power $n^{-\nu}$ with $\nu > 0$.

  We have thus shown that the family $(\CalA_n)_{n \in \N} = \bigl(\CalNN_n (3,\R^d)\bigr)_{n \in \N}$
  achieves every approximation rate $n^{-\AR}$ with $\AR < r/d$ on $\SC$.
  Since this family is of controlled complexity by \Cref{prop:NNApproximationOfControlledComplexity},
  it follows from Part~(ii) of \Cref{thm:PhaseTransitionForFamiliesWithControlledComplexity} that
  $\AR^\ast := \bigl(\umd (\SC)\bigr)^{-1}$ satisfies $\AR^\ast \geq \frac{r}{d}$.

  Overall, we thus see:
  \begin{itemize}
    \item We have $\frac{d}{r} \leq \lmd(\SC) \leq \umd(\SC) \leq \frac{d}{r}$ and thus
          $\md(\SC) = \frac{d}{r} \in (0,\infty)$.
          In particular, $\SC \subset L^p ([0,1]^d)$ is totally bounded.

    \item The set $\SC$ is the unit ball of a Banach space,
          and thus satisfies {\ConditionZetaText} by \cref{p_norm_unit_balls_are_quasi_convex},
          so that \Cref{thm:MainResultCriticalMeasureExistence} shows the existence
          of a Borel probability measure $\mu$ on $\XX = L^p ([0,1]^d)$
          that is critical for $\SC$.

    \item The family $(\CalA_n)_{n \in \N} = (\CalNN_n (3, \R^d))_{n \in \N}$ is of controlled
          complexity and achieves every approximation rate $0 \leq \AR < \AR^\ast = \frac{r}{d}$.
  \end{itemize}
  Thus, all remaining claims follow from \Cref{thm:PhaseTransitionForFamiliesWithControlledComplexity}.
\end{proof}

\subsection{
  Proof of \texorpdfstring{
             \Cref{prop:PolynomiallyBoundedNTermApproximationControlledComplexity}
           }{
              Proposition~\ref{prop:PolynomiallyBoundedNTermApproximationControlledComplexity}
           }
}
\label{sec:NonlinearNTermControlledComplexityProofs}

For the proof, we will need the following lemma from
\cite{VybiralFiniteDimensionalEmbeddings}.

\begin{lemma}[{Lemma~7 in \cite{VybiralFiniteDimensionalEmbeddings}}]\label{lem:VybiralEntropyNumberEstimate}
  Let $p \in (0, 1]$ and let $\XX$ be a real, $n$-dimensional $p$-normed space.
  Then
  \[
    e_k(\mathrm{id}: \XX \to \XX) \leq 4^{\frac{1}{p}} \cdot 2^{-\frac{k-1}{n}}
    \qquad \forall \, k \in \N
    ,
  \]
  where $\en_k$ denotes the $k$-th (dyadic) entropy number, see \eqref{eq:def_entropy_numbers}.
\end{lemma}

Using the previous lemma, we can now prove the following auxiliary result,
which is the main ingredient for proving \Cref{prop:PolynomiallyBoundedNTermApproximationControlledComplexity}.

\begin{lemma}\label{lem:CoveringNumberEstimatesForBallsInFiniteDimensionalSpaces}
  Let $(V, \|\cdot\|)$ be a quasi-normed space.
  Then there exists a constant $C > 0$ such that for every $n \in \N$,
  $R > 0$ and $\eps \in (0, R]$, every $v \in V$ and every subspace $W \subseteq V$
  with $\dim_{\R}(W) = n$, we have
  \[
    \cn \bigl(\ball(v, R) \cap W, \eps\bigr)
    \leq \left( C \cdot \frac{R}{\eps} \right)^n.
  \]
  Here, $\dim_{\R}(W)$ denotes the dimension of $W$ interpreted as a real vector space.
\end{lemma}

\begin{proof}
Let $k \geq 1$ be a triangle constant for $\| \cdot \|$.
By the Aoki-Rolewicz theorem (see \Cref{thm:Aoki_Rolewicz}), there exist $p \in (0, 1]$,
$C_1 \geq 1$, and a $p$-norm $\| \cdot \|^*$ on $V$ with
\[
  C_1^{-1} \cdot \| \cdot \|^*
  \leq \| \cdot \|
  \leq C_1 \cdot \| \cdot \|^*
  .
\]
We will write $\ball^*(x, r) := \{y \in V \,\,:\,\, \| x - y \|^* \leq r\}$ for $x \in V$, $r > 0$.

By rescaling, it is easy to see that
\begin{align*}
  \cn\bigl(\ball(v, R) \cap W, \eps\bigr)
  &= \cn\left(R \cdot \left(\ball\left(\frac{v}{R}, 1\right) \cap \frac{1}{R} W\right), R \frac{\eps}{R}\right) \\
  &= \cn\left(\ball\left(\frac{v}{R}, 1\right) \cap W, \frac{\eps}{R}\right).
\end{align*}
This shows that we can without loss of generality assume that $R=1$ and thus $\eps \in (0, 1]$.

Now, set
\[
  \ell
  := \left\lceil
       1 + n \cdot \log_2 \left( \frac{(2k)^2 C_1^2 \cdot 8^{1/p}}{\eps} \right)
     \right\rceil
  \in \N_{\geq 2}.
\]
By \Cref{lem:VybiralEntropyNumberEstimate}, we then obtain existence of
$\widetilde{w}_1, \dots, \widetilde{w}_{2^{\ell - 1}} \in W$ satisfying
\begin{align*}
  \ball\left(0, \frac{1}{C_1}\right) \cap W
  \subseteq \ball^*(0, 1) \cap W
  &\subseteq \bigcup_{j=1}^{2^{\ell - 1}} \ball^*\left(\widetilde{w}_j , 8^{1/p} \cdot 2^{-\frac{\ell - 1}{n}}\right) \cap W \\
   &\subseteq \bigcup_{j=1}^{2^{\ell - 1}} \ball^*\left(\widetilde{w}_j, \frac{\eps}{(2k)^2 C_1^2}\right) \\
  &\subseteq \bigcup_{j=1}^{2^{\ell - 1}} \ball\left(\widetilde{w}_j, \frac{\eps}{(2k)^2 C_1}\right)
  .
\end{align*}
By rescaling, this implies
\begin{equation}
  \ball(0, 2k) \cap W
  \subseteq \bigcup_{j=1}^{2^{\ell - 1}}
              \ball\left(w_j, \frac{\eps}{2k}\right)
  \qquad \text{for certain } w_1,\dots,w_{2^{\ell-1}} \in W
  .
  \label{eq:FiniteDimensionalSpaceCoveringEstimateProofStep1}
\end{equation}

Now, if $\ball(v, 1) \cap W = \emptyset$, the claim of the lemma is trivially satisfied.
We can thus assume that there exists $w_0 \in \ball(v, 1) \cap W$.
Given $w \in \ball(v, 1) \cap W$, we then have
\[
  d(w, w_0)
  \leq k \cdot \bigl(d(w, v) + d(v, w_0)\bigr)
  \leq 2k,
\]
and hence
\[
  \ball(v, 1) \cap W \subseteq \ball(w_0, 2k) \cap W.
\]
Now, \Cref{eq:FiniteDimensionalSpaceCoveringEstimateProofStep1} yields $w_1, \dots, w_{2^{\ell-1}} \in W$ with
\begin{equation*}
  \begin{aligned}
  \ball(v, 1) \cap W \subseteq \ball(w_0, 2k) \cap W &\subseteq w_0 + \bigl(\ball(0, 2k) \cap W\bigr) \\
  &\subseteq \bigcup_{j=1}^{2^{\ell-1}} \ball\left(w_j + w_0, \frac{\eps}{2k}\right).
  \end{aligned}
\end{equation*}
Thus, $\extcn\bigl( B(v,1)\cap W, \frac{\eps}{2k}\bigr) \le 2^{\ell-1}$.
By \cref{lem:InternalVSExternalCoveringNumbers}, this implies 
\[
  \begin{aligned}
      \cn( B(v,1)\cap W, \eps) 
      &\le \extcn\Bigl( B(v,1)\cap W, \frac{\eps}{2k}\Bigr)
    \leq 2^{\ell-1} \leq 2^{1 + n \log_2 \left( (2k)^2 C_1^2 8^{1/p} \big/ \eps \right)} \\
    &= 2 \cdot \left( \frac{(2k)^2 C_1^2 8^{1/p}}{\eps} \right)^n 
    \stackrel{(\ast)}{\leq} \left( \frac{8^{1 + \frac{1}{p}} k^2 C_1^2 \cdot R}{\eps} \right)^n 
    =: \left( C \cdot \frac{R}{\eps} \right)^n.
  \end{aligned}
\]
Here, we used at $(\ast)$ that $n \geq 1$ and $R = 1$.
This completes the proof of the lemma.
\end{proof}

\begin{proof}[Proof of \Cref{prop:PolynomiallyBoundedNTermApproximationControlledComplexity}]
Since $P$ is a polynomial, there exist $m, C_1 \in \N$ with $P(n) \leq C_1 \cdot n^m$ for all $n \in \N$.
Let $x \in \XX$, and $R , \alpha > 0$.

For each subset $I_0 \subset \N$ with $\# I_0 \leq n$, the space
\[
  U_{I_0} := \linspan \bigl\{g_k \,\,:\,\, k \in I_0\bigr\}
\]
is a subspace of $\XX$ of dimension at most $n$.
This means that independent of the considered ``base field'' $\mathbb{K} \in \{ \R, \CC \}$,
the real dimension of $U_{I_0}$ satisfies $\dim_{\R}(U_{I_0}) \leq 2 n$.
In case of $\dim(U_{I_0})\ge 1$, \Cref{lem:CoveringNumberEstimatesForBallsInFiniteDimensionalSpaces}
therefore shows for $\eps \in (0,R]$ that
\[
  \cn \bigl(\ball(x,R) \cap U_{I_0}, \eps\bigr)
  \leq \left( C_2 \cdot \frac{R}{\eps} \right)^{\dim_{\R}(U_{I_0})}
  \le \left( C_2 \cdot \frac{R}{\eps} \right)^{2 n}
\]
with a constant $C_2 = C_2(\|\cdot\|^{-1})$.
In case of $\dim(U_{I_0}) = 0$, we have $U_{I_0} = \{0\}$
and thus ${\cn\bigl(B(x,R)\cap U_{I_0}, \eps\bigr) \le 1 \le (C_2 \cdot R/\eps)^{2n}}$.
It is easy to see
\[
  \cn \left(\,\, \bigcup_{j=1}^M M_j, \eps \right)
  \leq \sum_{j=1}^M \cn(M_j, \eps) \text{ for arbitrary subsets } M_j \subset \XX.
\]
Finally, since each set $\emptyset \ne I_0 \subset \{1, \dots, C_1 \cdot n^m\}$ with $\# I_0 \leq n$ can be
specified by listing the $n$ elements it contains (with repetitions in case of $\num I_0 <n$), we have
\begin{align*}
  M
  &:= \# \bigl\{ I_0 : I_0 \subset \{k \in \N : k \leq P(n)\}, \# I_0 \leq n \bigr\} \\
  &\leq \# \bigl\{ I_0 : I_0 \subset \{1, \dots, C_1 \cdot n^m\}, \# I_0 \leq n \bigr\} \\
  &\leq \bigl(C_1 \cdot n^m\bigr)^n+1 
  \le (C_3 \cdot n^m)^n,
\end{align*}
where $C_3 := C_1 +1$.

Thus, we overall get for $n \geq n_0 := \lceil C_2 C_3 \max\{1,R\} \rceil$
and $C := 1 + 2\alpha + m$ that
\begin{align*}
  &\cn\left( \CalA_n \cap \ball(x,R), n^{-\alpha} \right)\\
  &\le \cn
     \Biggl(\,\,
       \bigcup_{\substack{I_0 \subset \{k \in \N : k \leq P(n)\} , \\ \# I_0 \leq n}}
         \ball(x,R) \cap \linspan \bigl\{g_k \,\,:\,\, k \in I_0\bigr\}, \quad
         \min\{R,n^{-\alpha}\}
     \Biggr) \\
  &\leq \sum_{\substack{I_0 \subset \{k \in \N : k \leq P(n)\} , \\ \# I_0 \leq n}}
  \cn\bigl(\ball(x,R) \cap U_{I_0}, \quad \min\{R, n^{-\alpha}\}\bigr) \\
  &\leq \bigl(C_3 \cdot n^m\bigr)^n \cdot \bigl(C_2 R \cdot \max\{R^{-1},n^{\alpha}\}\bigr)^{2 n} \\
  &\leq \bigl(C_2 C_3 R \max\{1,R^{-1}\}\bigr)^n \cdot n^{(2 \alpha + m)n}
   \overset{(\ast)}{\leq} n^n \cdot n^{(2 \alpha + m)n} \\
  &= n^{(1 + 2\alpha + m)n}
   = n^{Cn}
   \leq \exp(\fcc(n))
\end{align*}
for $\fcc(n) := C \cdot n \cdot \ln(e n)$.
Here, the step marked with $(\ast)$ used that $n \geq n_0 \geq C_2 C_3 R$.
We have thus verified Condition \eqref{eq:ControlledComplexity}.
It is easy to see that $\frac{\fcc(n)}{n^{1+\delta}} \xrightarrow[n\to\infty]{} 0$ for every $\delta>0$,
so that Condition \eqref{eq:ZetaGrowthCondition} is satisfied as well.
\end{proof}

\subsection{
  Proof of \texorpdfstring{
             \Cref{prop:ShearletsOptimalForC2Boundary}
           }{
              Proposition~\ref{prop:ShearletsOptimalForC2Boundary}
           }
}
\label{sec:ShearletProof}

\begin{proof}[Proof of \Cref{prop:ShearletsOptimalForC2Boundary}]
  The idea of the proof is to identify a set $\SC_{00}$ satisfying {\ConditionZetaText}
  and a map $\Psi\colon \SC_{00} \to \SC$ satisfying a kind of ``lower H{\"o}lder bound''
  which will allow us to transfer a critical measure for $\SC_{00}$ to a critical measure for $\SC$,
  using \Cref{prop:transfer_principle}.

  \medskip{}

  \textbf{Step 1} \emph{(Constructing a first "transfer map" $\Phi$):}
  In this step, we construct a set $\SC_0$ and a "transfer map" $\Phi : \SC_0 \to \SC$
  which will later be used to construct the final "transfer map" $\Psi$.

  For any non-trivial interval $I \subset \R$, let $C^{1,1}(I)$
  denote the set of bounded functions $f \in C^1 (I)$
  for which the derivative $f'$ is Lipschitz continuous.
  For $f \in C^{1,1}(I)$, define
  \[
    \| f \|_{C^{1,1}} 
    := \norm{f}_{C^{1,1}(I)}
    := \max\biggl\{ \| f \|_{L^\infty(I)} , \sup_{x,y \in I, x \neq y} \frac{|f'(x) - f'(y)|}{|x-y|}\biggr\}
    .
  \]
  Note that if $f \in C^2 (I)$, then $f$ belongs to $C^{1,1}(I)$ if and only if $f''$ is bounded
  (which is automatic, if $I$ is a compact interval),
  and then $\| f \|_{C^{1,1}(I)} = \max \{ \| f \|_{L^\infty(I)} , \| f'' \|_{L^\infty(I)} \}$;
  this is an easy consequence of the mean-value theorem.

  Now, with the set $C_{2\pi}^2 (\R)$ as in \Cref{def:StarShapedSets}, set
  \[
    \SC_0
    \coloneqq \left\{
                f \in C^2_{2\pi}(\R)
                \,\,\colon\,\,
                \|f\|_{C^{1,1}(\R)} \leq \frac{1}{8}
              \right\}
    ,
  \]
  let $x^{(0)} := (\frac{1}{2}, \frac{1}{2})^T \in \R^2$, and consider the map
  \[
    \Phi \colon \quad
    \SC_0 \to \SC , \quad
    f \mapsto \indicator_{x^{(0)} + S_{f + \frac{3}{8}}}
    ,
  \]
  where $S_{f + \frac{3}{8}}$ is defined as in \Cref{eq:StarShapedSet}.

  We first show that $\Phi$ is indeed well-defined, i.e., that $\Phi(f) \in \SC$ for $f \in \SC_0$.
  To see this, let $f \in \SC_0$; then $-\frac{1}{8} \leq f \leq \frac{1}{8}$,
  so that $\varrho_0 \leq \frac{1}{4} \leq f + \frac{3}{8} \leq \frac{1}{2} \leq \varrho_1$.
  Moreover, $f + \frac{3}{8} \in C_{2\pi}^2 (\R)$ with
  $\| (f + \frac{3}{8})'' \|_{L^\infty} \leq \| f \|_{C^{1,1}} \leq \frac{1}{8} \leq \nu$.
  Finally, for $x = (x_1, x_2)^T \in S_{f + \frac{3}{8}}$, we have $x_1 = r \cos(\varphi)$
  for some $\varphi \in [0, 2 \pi]$ and $0 \leq r \leq f(\varphi) + \frac{3}{8} \leq \frac{1}{2}$.
  Therefore, $|x_1| \leq \frac{1}{2}$, and one can similarly show that $|x_2| \leq \frac{1}{2}$.
  Hence, $x + x^{(0)} \in [0, 1]^2$, and this shows $x^{(0)} + S_{f + \frac{3}{8}} \subset [0, 1]^2$.
  Overall, this proves that $x^{(0)} + S_{f + \frac{3}{8}} \in \mathrm{STAR}^2 (\nu; \varrho_0, \varrho_1)$
  and hence $\Phi(f) \in \SC$.

  Next, we show that $\Phi$ satisfies
  $\norm{\Phi(f)-\Phi(g)}_{L^2(\RR^2)} \asymp \norm{f-g}_{L^1([0,2\pi])}^{\frac{1}{2}}$
  for $f,g\in \SC_0$.
  As the first step, we note for $f,g \in \SC_0$ that
  \begin{equation}
    \begin{aligned}
      & \bigl\| \Phi (f) - \Phi (g) \bigr\|_{L^2(\R^2)}^2 \\
      &= \Bigl\|
           \indicator_{x^{(0)} + S_{f + \frac{3}{8}}}
           - \indicator_{x^{(0)} + S_{g + \frac{3}{8}}}
         \Bigr\|_{L^2}^2 
       = \Bigl\|\indicator_{S_{f + \frac{3}{8}}} - \indicator_{S_{g + \frac{3}{8}}} \Bigr\|_{L^2}^2 \\
      &= \int_{\R^2}
           \biggl|
             \indicator_{S_{f + \frac{3}{8}}}(x) - \indicator_{S_{g + \frac{3}{8}}}(x)
           \biggr|^2
         \, dx \\
      &\overset{(\ast)}{=}
         \int_0^\infty
           r
           \int_0^{2\pi}
             \biggl|
               \indicator_{S_{f + \frac{3}{8}}}\left( r \begin{pmatrix} \cos\varphi \\ \sin\varphi \end{pmatrix} \right)
               - \indicator_{S_{g + \frac{3}{8}}}\left( r \begin{pmatrix} \cos\varphi \\ \sin\varphi \end{pmatrix} \right)
             \biggr|^2
           \, d\varphi
         \, dr \\
      &\overset{(\dagger)}{=}
         \int_0^{2\pi}
           \int_0^\infty
             r \cdot \Bigl|
                        \indicator_{0 \leq r \leq f(\varphi) + \frac{3}{8}}
                        - \indicator_{0 \leq r \leq g(\varphi) + \frac{3}{8}}
                     \Bigr|^2
           \, dr
         \, d\varphi \\
      &\overset{(\ddagger)}{=}
         \int_0^{2\pi}
           \int_0^\infty
             r \cdot \indicator_{(\min\{f(\varphi), g(\varphi)\} + \frac{3}{8} ,\,\, \max\{f(\varphi), g(\varphi)\} + \frac{3}{8}]}(r)
           \, dr
         \, d\varphi
      .
    \end{aligned}
    \label{eq:ShearletStuffPhiIsometryStep1}
  \end{equation}
  Here, the step marked with $(\ast)$ introduced polar coordinates and
  the step marked with $(\dagger)$ used Tonelli's theorem and the definition of the sets $S_{\varrho}$
  from \Cref{eq:StarShapedSet}.
  Finally, the step marked with $(\ddagger)$ used the identity
  \[
    \bigl| \indicator_{[0, a]}(r) - \indicator_{[0, b]}(r) \bigr|^2
    = \indicator_{(\min\{a,b\}, \max\{a,b\}]} (r)
    \qquad \text{ for } r \in \R \text{ and } a,b \geq 0
    .
  \]
  To prove this identity, assume without loss of generality that $a \leq b$
  and distinguish the four cases \emph{(i)} $r < 0$, \emph{(ii)} $0 \leq r \leq a$,
  \emph{(iii)} $a < r \leq b$, and \emph{(iv)} $r > b$.

  Now, we first derive a lower bound for $\bigl\| \Phi (f) - \Phi (g) \bigr\|_{L^2(\R^2)}^2$.
  To this end, first note for $a,b,c \geq 0$ that
  \begin{align*}
    &\int_{\R} \indicator_{(\min \{ a,b \} + c, \max \{ a,b \} + c]} (r) \, d r \\
    &= (\max \{ a,b \} + c) - (\min \{ a,b \} + c)
     = \max \{ a,b \} - \min \{ a,b \}
     = |a - b|
    .
  \end{align*}
  Note furthermore that if
  $\indicator_{(\min\{f(\varphi), g(\varphi)\} + \frac{3}{8} ,\,\, \max\{f(\varphi), g(\varphi)\} + \frac{3}{8}]}(r) \neq 0$,
  then $r \geq \frac{3}{8} \geq \frac{1}{4}$.
  Hence, we get in view of \Cref{eq:ShearletStuffPhiIsometryStep1} that
  \begin{equation}
    \begin{aligned}
      \bigl\| \Phi (f) - \Phi (g) \bigr\|_{L^2(\R^2)}^2
      &\geq \frac{1}{4}
            \int_0^{2\pi}
              \int_{\R}
                \indicator_{(\min\{f(\varphi), g(\varphi)\} + \frac{3}{8} ,\,\, \max\{f(\varphi), g(\varphi)\} + \frac{3}{8}]}(r)
              \, d r
            \, d \varphi \\
      &=   \frac{1}{4}
           \int_{0}^{2 \pi} |f(\varphi) - g(\varphi)| \, d \varphi
       =   \frac{1}{4} \,
           \| f - g \|_{L^1 ([0, 2\pi])}
       .
    \end{aligned}
    \label{eq:ShearletStuffPhiIsometryLowerBound}
  \end{equation}

  The upper bound is shown similarly:
  Note that if
  ${\indicator_{\frac{3}{8} + (\min\{f(\varphi), g(\varphi)\} ,\,\, \max\{f(\varphi), g(\varphi)\}]}(r) \neq 0}$,
  then $r \leq \frac{1}{2}$.
  Therefore, \Cref{eq:ShearletStuffPhiIsometryStep1} implies
  \[
    \begin{aligned}
      \bigl\| \Phi (f) - \Phi (g) \bigr\|_{L^2(\R^2)}^2
      &\leq \frac{1}{2}
            \int_0^{2\pi}
              \int_{\R}
                \indicator_{(\min\{f(\varphi), g(\varphi)\} + \frac{3}{8} , \max\{f(\varphi), g(\varphi)\} + \frac{3}{8}]}(r)
              \, d r
            \, d \varphi \\
      &=   \frac{1}{2} \,
           \| f - g \|_{L^1 ([0, 2\pi])}
       .
    \end{aligned}
  \]
  In particular, this implies that $\Phi$ is continuous
  (with respect to the $L^1$ topology on $\SC_0$ and the $L^2$ topology on $\SC$)
  and hence Borel measurable.

  \medskip{}

  \textbf{Step 2} \emph{(Constructing a suitable measure to transfer):}
  Write $C_b^2 (\R)$ for the set of all functions in $C^2 (\R)$ with $f,f',f''$ bounded,
  and equip this space with the norm $\| \cdot \|_{C^{1,1}(\RR)}$.
  Moreover, for a function $f : \R \to \R$, let $\supp (f)$
  denote the closure of the set $\{ x \in \R \,\,:\,\, f(x) \neq 0 \}$.
  It is well-known\footnote{This follows for instance by combining the very general result
  \cite[Chapter~17, Theorem~7.11]{SmithPrimerOfModernAnalysis} with a multiplication by a cutoff function
  and by noting that the $\| \cdot \|_{C^{1,1}}$ norm is equivalent to the norm
  ${\| f \|_{L^\infty} + \| f' \|_{L^\infty} + \| f'' \|_{L^\infty}}$
  by \cite[Lemma~F.1]{PetersenVoigtlaenderOptimalApproximationUsingNNs}.
  However, an easy direct proof is also possible since we are in the setting of
  an interval in dimension $d = 1$.}
  that there exists a bounded linear extension operator
  \[
    E : \quad L^1 ([\pi - 1, \pi + 1]) \to L^1 (\R)
  \]
  such that the restriction
  \[
    E: \quad C^{1,1}([\pi-1, \pi+1]) \to C^{1,1}(\R)
  \]
  is well-defined and bounded, with $(Ef)|_{[\pi-1, \pi+1]} = f$
  and $\supp (Ef) \subseteq [1, 2\pi - 1]$
  and with $E(C^2 ([\pi-1, \pi+1])) \subseteq C^2_b (\R)$.
  We let $E_{2 \pi} f$ be the $2\pi$-periodic extension of $(Ef)|_{[0, 2\pi)}$,
  noting that if $f \in C^2 ([\pi-1, \pi+1])$, then $E_{2 \pi}f \in C_{2\pi}^2(\R)$
  due to the condition $\supp (Ef) \subseteq [1, 2\pi-1]$.

  It is easy to see that there exists $C \geq 1$ with
  \[
    \| E f \|_{C^{1,1} (\R)} \leq C \cdot \| f \|_{C^{1,1}([\pi-1,\pi+1])}
    \quad \text{and} \quad
    \| E_{2\pi} f \|_{C^{1,1}(\R)} \leq C \cdot \| f \|_{C^{1,1}([\pi-1,\pi+1])}
  \]
  for all $f \in C^{1,1} ([\pi-1, \pi+1])$ and such that
  \[
    \| E_{2 \pi} f \|_{L^1 ([0, 2 \pi])}
    \leq C \cdot \| f \|_{L^1 ([\pi - 1, \pi + 1])}
    \qquad \forall \, f \in L^1 ([\pi - 1, \pi + 1])
    .
  \]

  Now, define
  \[
    \SC_{00}
    := \left\{
         f \in C^2 ([\pi-1, \pi+1])
         \,\,:\,\,
         \|f\|_{C^{1,1}} \leq \frac{1}{8 C}
       \right\}
    .
  \]
  We note that
  \[
    \overline{\SC_{00}}^{\|\cdot\|_{L^1}}
    \supseteq \SC_{00}^{1,1}
    := \left\{
         f \in C^{1,1}([\pi-1, \pi+1])
         \,\,:\,\,
         \|f\|_{C^{1,1}} \leq \frac{1}{16 \, C^2}
       \right\}.
  \]
  Indeed, for $f \in \SC_{00}^{1,1}$ we have $Ef \in C^{1,1}(\R)$ with
  $\|E f\|_{C^{1,1}} \leq 1/(16 \, C)$.
  Now, pick a function $\varphi \in C_c^\infty((-1,1))$ with $\varphi \geq 0$ and $\int_{-\infty}^\infty \varphi (x) dx = 1$
  and set $\varphi_\eps := \eps^{-1} \varphi(x/\eps)$ and $f_\eps := (Ef) * \varphi_\eps$
  (where the symbol "$\ast$" denotes convolution).
  Then, denoting the smallest Lipschitz constant of a function $g$ by
  $\Lip (g) = \sup_{x \neq y} \frac{|g(x) - g(y)|}{|x-y|}$,
  a combination of \cite[Section 4.13 and Theorem~4.15]{AltLinearFunctionalAnalysis}
  and \cite[Proposition~9.3]{FollandRealAnalysis} shows
  \(
    \|f_\eps\|_{C^{1,1}}
    = \max \{ \|f_\eps\|_{L^\infty}, \Lip (f'_\eps) \}
    \leq \|Ef\|_{C^{1,1}}
    \leq \frac{1}{16 \, C},
  \)
  as well as $f_{\eps} \in C^{\infty}(\RR)$,
  so that $f_\eps \in \SC_{00}$, and $f_\eps \xrightarrow{L^1 ([\pi-1,\pi+1])} (Ef)|_{[\pi-1, \pi+1]} = f$.

  Thus, noting that \emph{(i)} each $\eps$-covering (consisting of finitely many, closed $L^1$-balls)
  for $\SC_{00}$ is an $\eps$-covering for $\overline{\SC_{00}}^{\norm{\cdot}_{L^1}}$
  and hence for $\SC_{00}^{1,1}$, we see in view of
  \emph{(ii)} \Cref{lem:InternalVSExternalCoveringNumbers} and
  \emph{(iii)} \cite[Theorem~3]{ClementsEntropiesOfSeveralSetsOfFunctions} that
  \[
    \log_2 \Bigl( \cn(\SC_{00}, \|\cdot\|_{L^1}, \eps) \Bigr)
    \overset{(i)}{\geq} \log_2 \Bigl( \extcn (\SC_{00}^{1,1}, \|\cdot\|_{L^1}, \eps) \Bigr)
    \overset{(ii)}{\geq} \log_2 \Bigl( \cn\left(\SC_{00}^{1,1}, \|\cdot\|_{L^1}, 2\eps\right) \Bigr)
    \overset{(iii)}{\gtrsim} \eps^{- 1/2}
  \]
  for all sufficiently small $\eps > 0$.

  By definition, this shows $\lmd(\SC_{00}) \geq \frac{1}{2} > 0$.
  Moreover, on $C^2([\pi-1,\pi+1])$, the norm $\norm{\cdot}_{C^{1,1}}$ is equivalent to $\norm{\cdot}_{C^2}$.
  Hence $\SC_{00}$ is a ball around the origin in a Banach space and hence satisfies {\ConditionZetaText}
  by \cref{continuously_embedded_balls_are_zeta}.
  Thus, \Cref{thm:MainResultCriticalMeasureExistence} and \Cref{restriction_inherits_criticality}
  yield a Borel probability measure $\mu_0$ on $\SC_{00}$
  (with $\SC_{00}$ considered as a subset of $L^1 ([\pi-1, \pi+1])$)
  such that for each $0\le s < \frac{1}{2}$,
  there exist $\eps_0(s) > 0$ and $c(s) > 0$ with
  \[
    \mu_0 \bigl(\SC_{00} \cap \ball (f, \eps\mid \| \cdot \|_{L^1})\bigr)
    \leq \exp\bigl(-c(s) \cdot \eps^{-s}\bigr)
    \quad \forall \ \eps \in (0, \eps_0(s)) \text{ and } f \in \SC_{00} .
  \]

  \medskip

  \textbf{Step~3} \emph{(Transferring the measure to $\SC$):}
  We now define $\mu$ to be the push-forward of $\mu_0$ under the (Borel measurable) map 
  \[
    \Psi := \Phi \circ E_{2 \pi} : \quad \SC_{00} \subset L^1([\pi-1,\pi+1]) \to \SC \subset L^2 (\R^2).
  \]
  Note that indeed $E_{2\pi}(\SC_{00}) \subset \SC_{0}$ and $\Phi : \SC_0 \to \SC$,
  so that $\Psi$ is well-defined and $\mu$ is a Borel probability measure on $\SC$.
  Here, we use that $E_{2\pi} : L^1 ([\pi-1,\pi+1]) \to L^1([0, 2 \pi])$
  is continuous with $E_{2 \pi}(\SC_{00}) \subset \SC_{0}$ and that
  $\Phi : \SC_0 \to \SC$ is continuous, with respect to the $L^1([0, 2 \pi])$-topology
  on $\SC_0$ and the $L^2 (\R^2)$-topology on $\SC$.
  Hence, $\Psi : \SC_{00} \to \SC$ is continuous and hence Borel measurable,
  with respect to the $L^1 ([\pi-1, \pi+1])$-topology on $\SC_{00}$ and the $L^2(\R^2)$-topology
  on $\SC$.

  Moreover, the map $\Psi$ satisfies a sub-H{\"o}lder condition as in \cref{prop:transfer_principle}.
  Indeed, for ${g,h \in \SC_{00}}$, we have $E_{2 \pi}(g)|_{[\pi-1, \pi+1]} = g$
  (and similarly for $h$).
  It follows that
  \begin{align*}
    \| g - h \|_{L^1([\pi-1, \pi+1])}
    &\leq \| E_{2 \pi}(g) - E_{2 \pi}(h) \|_{L^1([0, 2 \pi])} \\
    \overset{\text{Eq. \eqref{eq:ShearletStuffPhiIsometryLowerBound}}}&{\leq}
          4 \,\, \| \Phi(E_{2 \pi}(g)) - \Phi(E_{2\pi}(h)) \|_{L^2(\R^2)}^2.
  \end{align*}
  Now \cref{prop:transfer_principle} (with $\beta=\frac{1}{2}$ and applied once for every $s<\frac{1}{2}$)
  implies that $\mu$ satisfies the small-ball condition with parameter $\sigma$ for arbitrary $\sigma<1$.
  By \Cref{cor:SmallBallCondition_Dominated_by_LowerMinkowskiDim},
  this implies $\lmd(\SC) \geq 1$.

  \medskip

  \textbf{Step 4} \emph{(Completing the proof):}
  Finally, \cite[Theorem~E.10]{VoigtlaenderPeinAnalysisVSSynthesisSparsityAlphaShearlets}
  implies that there exists a polynomial $P$ 
  and a suitably enumerated cone-adapted shearlet system $\CalG = (g_n)_{n \in \N}$ 
  such that $(\CalA_n)_{n \in \N} = \bigl(\Sigma_n(\CalG, P)\bigr)_{n \in \N}$ satisfies
  for arbitrary $\delta \in (0,1)$ that
  \[
    \sup_{f \in \SC} \dist_{L^2}(f, \CalA_n) \leq C_\delta \cdot n^{-(1 - \delta)}
    \qquad \forall \, n \in \N
    .
  \]
  By \Cref{thm:PhaseTransitionForFamiliesWithControlledComplexity}
  and \Cref{prop:PolynomiallyBoundedNTermApproximationControlledComplexity},
  this implies $\umd(\SC) \leq 1$.

  Overall, we see that $\md(\SC) = 1$, that the measure $\mu$ is critical for $\SC$
  (equipped with the $L^2 (\R^2)$ metric),
  and that the family $(\CalA_n)_{n \in \N} = \bigl(\Sigma_n(\CalG, P)\bigr)_{n \in \N}$ 
  is optimal for $\SC$ in the class of families of controlled complexity.
  We note that the constructed measure $\mu$ is a critical Borel probability measure
  on $\SC \subset L^2(\R^2)$; this can be changed to a Borel probability measure
  on $L^2 (\R^2)$ that is critical for $\SC$ using \Cref{cor:IntrinsicToExtrinsic}.
\end{proof}

%% file: parts/technical_proofs_quasi_metric.tex
\section{Proofs of technical results in the setting of quasi-metric spaces}
\label{sec:QuasiMetricTechnicalProofs}

\subsection{Proof of \texorpdfstring{\cref{lem:FrostmanLemmaEasy}}{Lemma~\ref{lem:FrostmanLemmaEasy}}}
\label{sec:proof_of_lem:FrostmanLemmaEasy}

\begin{proof}[Proof of \cref{lem:FrostmanLemmaEasy}]
  This result already appears in \cite[Proposition~3.4]{kloecknerGeneralizationHausdorffDimension2012},
  but only for the setting of genuine metric spaces, and without a proof.
  For the sake of completeness, we provide the proof here.

  If $s = 0$, the claim of the lemma is trivial.
  Thus, let $s > 0$.
  Note that $\SC \neq \emptyset$, since $\mu^{\ast}(\SC) > 0$.
  Let $c,r_0$ as in \cref{eq:SmallBallConditionPrecise}.
  In the following, we use the interpretation $\exp\bigl(-c \cdot (1/r)^s\bigr) = 0$ for $r = 0$.
  By letting $r \downarrow 0$ in \cref{eq:SmallBallConditionPrecise}, we then see that
  \cref{eq:SmallBallConditionPrecise} remains valid for $r = 0$ as well.
  Here, we use that $\om{\mu}\bigl(B(x,0)\bigr) \le \om{\mu}\bigl(B(x,r)\bigr)$ for all $r>0$.

  Assume towards a contradiction that $\hausd(\SC) < s$.
  By definition (see \cref{eq:PEHausdorffDimensionDefinition}), this means $\CalH_{\exp}^{\sigma} (\SC) = 0$
  for some $\sigma \in (\hausd(\SC),s)$.
  Since $\sigma < s$, there exists $r^\ast = r^\ast (s, \sigma, c) > 0$ such that
  \begin{equation}
    c \cdot (1/r)^s \geq (1/r)^\sigma
    \qquad \forall \, 0 < r \le r^\ast
    .
    \label{eq:FrostmannLemmaEasyPowerEstimate}
  \end{equation}

  Now, by definition of $\CalH_{\exp}^{\sigma}$, see \cref{sub:power-exp_Hausdorff_dimension},
  there exists a covering $\SC \subset \bigcup_{i \in I} M_i$,
  with a countable (or finite) index set $I$, with sets $M_i \subset \SC$
  (without loss of generality, we may assume that $M_i\subset \SC$ instead of $M_i \subset\XX$,
  by switching from $M_i$ to $M_i \cap \SC$, if necessary)
  satisfying $\diam(M_i) \leq \min \{ \frac{r_0}{2}, r^\ast \}$,
  and such that $\sum_{i \in I} f_\sigma (\diam(M_i)) < \om{\mu}(\SC)$,
  with $f_\sigma$ as in \cref{eq:PEHausdorffHelperFunction}.
  By discarding empty sets, we can assume that $M_i \neq \emptyset$ for all $i \in I$.

  Now, for each $i \in I$, set $r_i := \diam (M_i)$ and choose $x_i \in M_i$.
  By definition of the diameter and since we consider "closed" balls,
  we then have $M_i \subset \ball (x_i, r_i)$.
  Overall, we thus get
  \begin{align*}
    \om{\mu}(\SC)
      &  = \mu^\ast \left(\bigcup_{i \in I} M_i\right) \\
      & \leq \sum_{i \in I} \mu^\ast (M_i)
        \leq \sum_{i \in I} \mu^\ast (\ball (x_i, r_i)) \\
      & \overset{\eqref{eq:SmallBallConditionPrecise}}{\leq}
          \sum_{i \in I} \exp\bigl(- c \cdot (1/r_i)^s\bigr)
        \overset{\eqref{eq:FrostmannLemmaEasyPowerEstimate}}{\leq}
          \sum_{i \in I} \exp\bigl(- (1/r_i)^\sigma\bigr) \\
      & = \sum_{i \in I} f_\sigma (\diam(M_i))
      < \om{\mu}(\SC)
      ,
  \end{align*}
  which is the desired contradiction.
\end{proof}

\subsection{Proof of \texorpdfstring{\cref{lem:CompressionRateVSMinkowskiDimension}}{Lemma~\ref{lem:CompressionRateVSMinkowskiDimension}}}
\label{sec:proof_of_lem:CompressionRateVSMinkowskiDimension}

\begin{proof}[Proof of \cref{lem:CompressionRateVSMinkowskiDimension}]
  Note that the claim is trivial in case of $\# \SC < 2$,
  since in this case $\# \SC = 1$ and this easily implies $\CR^\ast (\SC, \XX) = \infty$
  and $\umd(\SC) = 0$.
  Hence, in the remainder of the proof, we can assume that $\# \SC \geq 2$.
  For the entire proof, let $k \geq 1$ be a triangle constant for $d$.
  We separately prove "$\leq$" and "$\geq$".

  \medskip{}

  "$\leq$":
  This is trivial if $\CR^\ast (\SC, \XX) = 0$.
  Thus, assume $\CR^\ast (\SC, \XX) > 0$, and let $0 < \CR < \CR^\ast (\SC, \XX)$ be arbitrary.
  It easily follows from the definition of $\CR^\ast(\SC,\XX)$ that there exist $C \geq 1$
  and a codec sequence $( (E_n, D_n))_{n \in \N}$ for $\SC$ with
  $\delta(E_n, D_n) \leq C \cdot n^{-\CR}$ for all $n \in \N$.
  Notice that for every $n\in \NN$, the range of $D_n$ consists of at most $2^n$ points in $\XX$.
  By definition of $\delta$, it thus follows that $\SC$ can be covered by $2^n$ closed balls with radius $C\cdot n^{-\CR}$.
  Hence 
  \begin{equation}
      \extcn(\SC, \XX, C\cdot n^{-\CR}) \le 2^n \quad \forall n \in \NN.
      \label{eqn_102_1}
  \end{equation}

  We will now show that \eqref{eqn_102_1} is sufficient to give suitable upper bounds
  for $\cn(\SC, \eps)$ for all sufficiently small $\eps>0$.
  Let $\eps \in (0,1]$ be arbitrary.
  Choose $n \in \N_{\geq 2}$ with
  \begin{equation}
    2Ck \cdot n^{-\CR} < \eps \leq 2 C k \cdot (n-1)^{-\CR}
    .
      \label{eqn_102_2}
  \end{equation}
  Note that this implies because of $n \geq 2$ that
  \[
      n^{\CR}
    = \bigl( (n-1) + 1\bigr)^{\CR}
    \leq \bigl(2 (n-1)\bigr)^{\CR}
    \leq \frac{2^{\CR+1} C k}{\eps}
  \]
  and thus $n \leq (2^{\CR+1} C k)^{1/\CR} \eps^{-1/\CR}$.
  
  Using the monotonicity of $\cn$ together with \eqref{eqn_102_2} at (a),
  \cref{lem:InternalVSExternalCoveringNumbers} at (b), and \eqref{eqn_102_1} at (c), we infer that
  \begin{equation*}
     \cn(\SC,\eps)
     \overset{(a)}{\le}
     \cn(\SC, 2Ck \cdot n^{-\CR})
     \overset{(b)}{\le}
 \extcn(\SC,\XX, C\cdot n^{-\CR})
 \overset{(c)}{\le} 2^n.
  \end{equation*}
  We have thus shown
  \[
    \log_2\bigl(\cn (\SC, \eps)\bigr)
    \leq n
    \leq (2^{\CR+1} C k)^{1/\CR} \eps^{-1/\CR}
  \]
  and thus
  \[
    \log_2 \bigl(\log_2 (\cn (\SC, \eps))\bigr)
    \leq \log_2 \bigl( (2^{\CR+1} C k)^{1/\CR} \eps^{-1/\CR}\bigr)
    =    \frac{1}{\CR} \log_2 (2^{\CR+1} C k) + \frac{1}{\CR} \log_2 (1/\eps)
  \]
  for all $\eps \in (0,1]$.

  Thus, we finally see
  \[
    \umd (\SC)
    = \limsup_{\eps \downarrow 0}
        \frac{\log_2 \log_2 \cn(\SC, \eps)}{\log_2 (1/\eps)}
    \leq \limsup_{\eps \downarrow 0}
         \left(
           \frac{\frac{1}{\CR} \log_2 (2^{\CR+1} C k)}{\log_2 (1/\eps)}
           + \frac{1}{\CR}
         \right)
    =    \frac{1}{\CR}
    <    \infty
    .
  \]
  Hence, $\CR \leq \bigl(\umd (\SC)\bigr)^{-1}$.
  Since $0 < \CR < \CR^\ast (\SC, \XX)$ was arbitrary, this proves "$\leq$".

  \medskip{}

  "$\geq$":
  In case of $\umd(\SC) = \infty$, the claim is trivial.
  Hence, assume $\umd(\SC) < \infty$, and let $0 < \CR < \bigl(\umd(\SC)\bigr)^{-1}$
  be arbitrary, so that $\umd(\SC) < \frac{1}{\CR}$.
  By definition of $\umd(\SC)$, see \cref{sub:power-exp_Minkowski_dimension},
  this implies existence of $\eps_0 \in (0,1)$ with
  \[
    \frac{\log_2 \log_2 \cn(\SC, \eps)}{\log_2 (1/\eps)} \leq \frac{1}{\CR}
    \qquad \forall \, 0 < \eps \leq \eps_0
    .
  \]
  Hence,
  \begin{equation}
    \cn(\SC, \eps) \leq 2^{(1/\eps)^{1/\CR}}
    \qquad \forall \, 0 < \eps \leq \eps_0
    .
    \label{eq:CompressionRateVSMinkowskiCoveringBound}
  \end{equation}

  Let $n_0 := \lceil (1/\eps_0)^{1/\CR} \rceil - 1 \in \N$.
  For $n \in \N$ with $n > n_0$, we have $n \geq (1/\eps_0)^{1/\CR}$, meaning that
  $\eps_n := n^{-\CR}$ satisfies $\eps_n \leq \eps_0$.
  Hence, $\cn(\SC,\eps_n) \leq 2^{(1/\eps_n)^{1/\CR}} = 2^n$ by
  \cref{eq:CompressionRateVSMinkowskiCoveringBound}.
  We can thus choose points $x_1,\dots,x_{2^n} \in \SC$ (not necessarily pairwise distinct) 
  such that $\{ x_1,\dots,x_{2^n} \}$ is an $\eps_n$-net of $\SC$.
  Choose a surjective map $D_n : \{ 0,1 \}^n \to \{ x_1,\dots,x_{2^n} \}$.
  For every $x \in \SC$, there then exists $E_n (x) \in \{ 0,1 \}^n$ with
  $d\bigl(x, D_n(E_n(x))\bigr) \leq \eps_n = n^{-\CR}$.
  This completes the construction of the codec $(E_n, D_n)$ for $n > n_0$.

  Finally, for $n \in \N$ with $n \leq n_0$, let $D_n : \{ 0,1 \}^n \to \SC \subset \XX$
  and $E_n : \SC \to \{ 0,1 \}^n$ be arbitrary.
  Then, we have for $C := \max \bigl\{ 1, \diam(\SC) \cdot n_0^\CR \bigr\}$ that
  \[
    \delta(E_n, D_n)
    \leq \begin{cases}
           n^{-\CR} \leq C \cdot n^{-\CR},                            & \text{if } n > n_0, \\
           \diam(\SC) \leq C \cdot n_0^{-\CR} \leq C \cdot n^{-\CR} , & \text{if } n \leq n_0 .
         \end{cases}
  \]
  Here, we note because of $\cn(\SC, \eps) < \infty$ for $0 < \eps \leq \eps_0$ that
  $\SC$ is totally bounded and hence bounded, so that $\diam(\SC) < \infty$.

  Overall, by definition of $\CR^\ast (\SC, \XX)$, we see $\CR^\ast (\SC,\XX) \geq \CR$.
  Since $0 < \CR < \bigl(\umd(\SC)\bigr)^{-1}$ was arbitrary, this proves "$\geq$".
\end{proof}

\subsection{Proof of \texorpdfstring{\cref{eq:PEHausdorffDimensionDefinition}}{Equation~(\ref{eq:PEHausdorffDimensionDefinition})}}
\label{sec:proof_of_eq:PEHausdorffDimensionDefinition}

In order to show that \cref{eq:PEHausdorffDimensionDefinition} holds, we first need the following lemma.
Its proof follows \cite[Theorem~4.7]{mattilaGeometrySetsMeasures1995}.
See also \cite[Lemma~2.2]{kloecknerGeneralizationHausdorffDimension2012}.
Since we use a different scale than in \cite[Theorem~4.7]{mattilaGeometrySetsMeasures1995},
since \cite[Lemma~2.2]{kloecknerGeneralizationHausdorffDimension2012} is stated without a proof,
and since we deal with quasi-metric spaces, we provide a proof.

\begin{lemma}\label{lem:ComparisionHausdorffMeasures_s_t}
    Let $(\XX,d)$ be a quasi-metric space, let $M\subset \XX$, and let $0<s<t<\infty$.
    \begin{enumerate}[label=(\roman*)]
        \item If $\hdm^s(M) <\infty$, then $\hdm^t(M)=0$.
        \item If $\hdm^t(M) >0$, then $\hdm^s(M)=\infty$. 
    \end{enumerate}
\end{lemma}

\begin{proof}
    Let $\delta_0 \in (0,1)$ be arbitrary.
    Let $\delta \in (0,\delta_0)$. 
    Since $\hdm^s(M)<\infty$, it follows from the definition of $\hdm^s$,
    see \cref{sub:power-exp_Hausdorff_dimension}, that there exists a sequence $(M_i)_{i\in \NN}$
    of subsets $M_i\subset \XX$ such that $\diam(M_i)\le \delta$ for all $i\in \NN$,
    such that $M \subset \cup_{i\in\NN} M_i$, and such that
    \begin{equation}
        \sum_{i\in \NN} \exp\Bigl( -\Bigl(\frac{1}{d_i}\Bigr)^s\Bigr) 
        \le \hdm^s(M) + 1,
        \label{eqn_101_1}
    \end{equation}
    where we set $d_i := \diam(M_i)$ for $i\in \NN$ and where we identify $\exp(-(1/r)^s)=0$ if $r=0$.
    Moreover, since $\frac{\dif}{\dif x} (x^{-s} - x^{-t}) = x^{-1} ( tx^{-t} - sx^{-s}) \geq 0$
    for $x\le 1$, we have
    \begin{equation}
        \exp\bigl( - d_i^{-t} + d_i^{-s} \bigr)
        \le  \exp\bigl( - \delta^{-t} + \delta^{-s} \bigr).
        \label{eqn_101_2}
   \end{equation}

   Using \eqref{eqn_101_1} and \eqref{eqn_101_2}, we get
    \begin{align*}
        \hdm^{t,\delta_0}(M)
        &\le \sum_{i\in \NN} \exp\Bigl(-\Bigl(\frac{1}{d_i}\Bigr)^t\Bigr)
        =\sum_{i\in \NN} \exp\Bigl(-\Bigl(\frac{1}{d_i}\Bigr)^s\Bigr) 
        \cdot \exp\bigl( - d_i^{-t} + d_i^{-s} \bigr)\\
        &\le \bigl(\hdm^{s}(M) +1 \bigr) \cdot \exp\bigl( -\delta^{-t} + \delta^{-s} \bigr).
    \end{align*}
    Passing to the limit $\delta\downarrow0$, we obtain that $\hdm^{t,\delta_0}(M) =0$.
    Passing to the limit $\delta_0\downarrow 0 $, we obtain $\hdm^{t}(M)=0$.
    This shows (i).
    The second claim (ii) is just the contraposition of the first claim and is thus also proved.
\end{proof}

With \cref{lem:ComparisionHausdorffMeasures_s_t} in place, we may proceed to prove \cref{eq:PEHausdorffDimensionDefinition}.

\begin{lemma}
  Let $(\XX,d)$ be a quasi-metric space and let $M\subset \XX$. 
  Then the four expressions from \cref{eq:PEHausdorffDimensionDefinition} coincide.
\end{lemma}

\begin{proof}
    Let $S_{<\infty} := \bigl\{s >0 : \hdm^s(M)<\infty\bigr\}$
    and define $S_{=0}$, $S_{=\infty}$, and $S_{>0}$ analogously.

    We first deal with the edge cases. 
    If $S_{<\infty}= \emptyset$, then $S_{=0}=\emptyset$, $S_{=\infty}=S_{>0}=(0,\infty)$.
    In that case, all four quantities in \cref{eq:PEHausdorffDimensionDefinition} are equal to $+\infty$.
    If $S_{>0}=\emptyset$, then $S_{=\infty}=\emptyset$, and $S_{=0}=S_{<\infty}=(0,\infty)$.
    In that case, all four quantities in \cref{eq:PEHausdorffDimensionDefinition} are equal to $0$.

    Assume in the following that there exist $s_0,t_0 \!\in\! (0,\infty)$
    such that $\hdm^{s_0}(M) \!>\! 0$ and $\hdm^{t_0}(M) \!<\! \infty$.
    It follows from \cref{lem:ComparisionHausdorffMeasures_s_t}
    that $\hdm^{s_0/2}(M)=\infty$ and that $\hdm^{2t_0}(M)=0$.
    In particular, all four sets defined above are non-empty.

    From the subset relations, we immediately get 
    \begin{equation*}
        \inf S_{<\infty} \le \inf S_{=0}
        \quad \text{and}\quad
        \sup S_{=\infty} \le \sup S_{>0}.
    \end{equation*}

    We claim that $\inf S_{=0} \le \sup S_{=\infty}$.
    Indeed, if $\inf S_{=0} =0$, then the inequality is trivial. 
    If $\inf S_{=0}>0$, let $0<s<\inf S_{=0}$ be arbitrary and take $s<t<\inf S_{=0}$.
    Then $\hdm^{t}(M)>0$ and thus, by \cref{lem:ComparisionHausdorffMeasures_s_t}, that $\hdm^{s}(M)=\infty$. 
    Hence $s\le \sup S_{=\infty}$.
    Since this holds for all $0 < s< \inf S_{=0}$, we deduce the claim also in that case.

    We claim that $\sup S_{>0} \le \inf S_{<\infty}$.
    Indeed, if $s \in S_{>0}$ and $t \in S_{<\infty}$, then $\hdm^{s}(M)>0$ and $\hdm^{t}(M)<\infty$.
    It follows that $s\le t$, as otherwise $s>t$ would contradict \cref{lem:ComparisionHausdorffMeasures_s_t}.

    Combining all four inequalities, we deduce that all quantities in \cref{eq:PEHausdorffDimensionDefinition} must be equal.
\end{proof}

\subsection{Proof of \texorpdfstring{\cref{eq:PEHausdorffDominatedByPELowerMinkowski}}{Equation~(\ref{eq:PEHausdorffDominatedByPELowerMinkowski})}}
\label{sec:proof_of_eq:PEHausdorffDominatedByPELowerMinkowski}

\begin{lemma}\label{lem:PEHausdorffDominatedByPEMinkowski}
  Let $\SC$ be a non-empty quasi-metric space.
  Then we have
  \[
    \hausd (\SC) \leq \lmd (\SC)
    .
  \]
\end{lemma}

\begin{proof}
  The proof largely follows the one of
  \cite[Proposition~3.2]{kloecknerGeneralizationHausdorffDimension2012},
  which is the same result, but only for the case of genuine metric spaces.

  If $\lmd(\SC) = \infty$, we are done.
  In the following, we thus assume that $\lmd(\SC) < \infty$,
  which in particular implies that $\SC$ is totally bounded.
  In case of $\# \SC < 2$, it is straightforward to see that $\hausd(\SC) = 0 \leq \lmd(\SC)$.
  Hence, we can assume that $\# \SC \geq 2$.

  Let $t > s > \lmd(\SC)$.
  Since $s > \lmd(\SC)$, there exist $c > 0$ and a sequence $(\eps_n)_{n\in\NN} \subset (0,\infty)$
  such that $\lim_{n\to\infty} \eps_n = 0$ and such that
  \begin{equation*}
      \cn(\SC,\eps_n) 
      < \exp\Bigl( c \cdot \Bigl(\frac{1}{\eps_n}\Bigr)^s\Bigr) 
      < \infty
      \quad \text{ for all $n\in \NN$.}
  \end{equation*}
  For $n\in \NN$ let $N(n):= \cn(\SC,\eps_n)$ and let $(B_i^{(n)})_{i=1}^{N(n)}$ be a covering
  of $\SC$ by balls of radius $\eps_n$.
  Using the quasi-triangle inequality with constant $k$,
  it follows that $d_i^{(n)}:=\diam(B_i^{(n)}) \le 2 k \eps_n$ for every $i=1,\dots,N(n)$.
  We infer that
  \begin{equation*}
      \hdm^{t,2k\eps_n}(\SC)
      \le \sum_{i=1}^{N(n)} \exp\Bigl(- \Bigl(\frac{1}{d_i^{(n)}}\Bigr)^t\Bigr)
      \le \exp\Bigl( c\cdot \Bigl(\frac{1}{\eps_n}\Bigr)^s\Bigr)  
      \cdot \exp\Bigl(- \Bigl(\frac{1}{2k\eps_n}\Bigr)^t\Bigr).
  \end{equation*}
  It follows that
  \begin{equation*}
      \hdm^{t}(\SC) 
      = \lim_{n\to \infty}\hdm^{t,2k\eps_n}(\SC)
      \le \liminf_{n\to \infty}\exp\Bigl( c\cdot \Bigl(\frac{1}{\eps_n}\Bigr)^s\Bigr)  
      \cdot \exp\Bigl(- \Bigl(\frac{1}{2k\eps_n}\Bigr)^t\Bigr) =0.
  \end{equation*}
  Hence $\hausd(\SC) \le t$.
  Since this holds for all $t>\lmd(\SC)$, we deduce the claim.
\end{proof}

\subsection{Proof of \texorpdfstring{\cref{measure_SC_n_is_well_def}}{Lemma~\ref{measure_SC_n_is_well_def}}}
\label{sec:proof_of_measure_SC_n_is_well_def}

\begin{proof}[Proof of \cref{measure_SC_n_is_well_def}]
    \emph{Step 1:}
    We first show that $\mu$ is a well-defined probability measure.
    Since every $\SC_n$ is finite and non-empty, the uniform distribution on $\SC_n$ exists
    and is a Borel measure on $\XX$.
    By \cite[Proposition 10.6.1]{cohnMeasureTheory2013}, the sequence of independent Borel-measurable
    random vectors $Y_n \colon \Omega \to \XX$, $n\in \NN$, exists
    (i.e., there exists a probability space $(\Omega,\CalA)$ on which a sequence $(Y_n)_{n\in \NN}$
    of random variables is defined, such that the $Y_n$ are independent and each $Y_n$
    is uniformly distributed on $\SC_n$).

    Now define $S_n := \sum_{i=1}^n \frac{Y_i}{C_0\cdot i^{\alpha}}$.
    We show that $S_n\colon \Omega \to \XX$ is Borel-measurable on $\XX$.
    Let \mbox{$A\in \Borel(\XX)$} be arbitrary and let
    \begin{equation*}
      \widetilde{A}:=
      \biggl\{
        (y_1,\dots,y_n) \in \SC_1\times \dots \times \SC_n
        \,\,:\,\,
        \sum_{i=1}^n \frac{y_i}{C_0\cdot i^{\alpha}} \in A
      \biggr\}.
    \end{equation*}
    Then $\widetilde{A}$ is finite, because $\SC_1 \times \dots \times \SC_n$ is.
    Moreover, we have
    \begin{align}
        S_n^{-1}(A) 
        &= \bigcup_{(y_1,\dots,y_n) \in \widetilde{A}} 
        \bigl\{ \omega \in \Omega : \text{for all $i=1,\dots,n$: $Y_i(\omega)=y_i$}\bigr\}
        \nonumber\\
        &= \bigcup_{(y_1,\dots,y_n) \in \widetilde{A}}  \,\,
             \bigcap_{i=1,\dots,n} Y_i^{-1}\bigl(\{y_i\}\bigr).
        \label{eqn_100_1}
    \end{align}
    Since every $Y_i$ is Borel-measurable, it follows
    that each $Y_i^{-1}(\{y_i\})$ is a measurable subset of $\Omega$.
    By \eqref{eqn_100_1}, the set $S_n^{-1}(A)$ is thus a finite union
    of finite intersections of measurable sets.
    Hence $S_n^{-1}(A)$ is measurable, meaning that $S_n$ is a Borel-measurable function.

    Now define
    \begin{equation*}
        Y \colon \quad \Omega \to \XX,\quad
        Y(\omega) := \sum_{n=1}^{\infty} \frac{Y_n(\omega)}{C_0\cdot n^{\alpha}}.
    \end{equation*}
    Since $\SC$ satisfies \ConditionZetaTextAlphaC{(\alpha,C_0)},
    the series converges for every $\omega\in \Omega$ to some element in $\SC$.
    Furthermore, since $\XX$ is metrizable by \cref{thm:Aoki_Rolewicz},
    and since $Y$ is the pointwise limit of the sequence of Borel-measurable functions
    $(S_n)_{n\in \NN}$, it follows from \cite[Proposition 8.1.10]{cohnMeasureTheory2013}
    that $Y$ is also Borel measurable.
    Hence $\mu$ given in \cref{definition_measure_SC_n} is well-defined.

    \medskip{}
    \emph{Step 2:}
    We show that $\om{\mu}(\XX\setminus \SC)=0$.

    To begin, note that each $\SC_n$ is finite, and hence a separable, completely metrizable space,
    i.e., a Polish space; see \cite[Section~8.1]{cohnMeasureTheory2013}.
    Thus, by \cite[Proposition~8.1.4]{cohnMeasureTheory2013}, the space $\CalZ := \prod_{n=1}^{\infty} \SC_n$
    (with the product topology) is a Polish space as well.
    Moreover, \cite[Proposition~8.1.7]{cohnMeasureTheory2013} shows that the product
    $\sigma$-algebra on $\CalZ$ agrees with the Borel $\sigma$-algebra on $\CalZ$.

    \smallskip{}

    Next, let $\widehat{\XX}$ be the completion of $\XX$ (as a metric space, where we choose any metric
    inducing the topology of $\XX$; such a metric exists by the Aoki-Rolewicz theorem
    (\Cref{thm:Aoki_Rolewicz})).
    Note that each projection
    \[
      \pi_n : \quad
      \CalZ \to \SC_n \subset \SC \subset \XX \subset \widehat{\XX}, \quad
      z = (z_\ell)_{\ell \in \N} \mapsto z_n
    \]
    is continuous, meaning that each of ``partial sum'' maps
    \[
      S_m : \quad
      \CalZ \to \XX \subset \widehat{\XX}, \quad
      S_m \bigl( (z_\ell)_{\ell \in \N}\bigr) = \sum_{n=1}^{m} \frac{z_n}{C_0 \, n^\alpha}
    \]
    is continuous as well (since it only depends on finitely many coordinates).
    Now, define
    \[
      \CalP
      := \bigcup_{m \in \N} S_m (\CalZ)
      \subset \XX
      \subset \widehat{\XX}
      ,
    \]
    and let $\overline{\CalP}$ be the closure of $\CalP$ in $\widehat{\XX}$.
    Note that since each of the sets $\SC_n$ is finite, so is each set $S_m (\CalZ)$,
    which implies that $\CalP$ is countable, so that $\overline{\CalP}$ is separable,
    and hence a Polish space.

    Now, since $\SC$ satisfies \ConditionZetaTextAlphaC{(\alpha,C_0)}, the map
    \[
      S : \quad
      \CalZ \to \SC \subset \XX \subset \widehat{\XX} , \quad
      z = (z_\ell)_{\ell \in \N}
      \mapsto \sum_{n=1}^{\infty} \frac{z_n}{C_0 \, n^\alpha}
            = \lim_{m \to \infty} S_m (z)
    \]
    is well-defined.
    Because of $S_m (z) \in S_m (\CalZ) \subset \CalP \subset \overline{\CalP}$,
    we get that
    \[
      S(z) = \lim_{m \to \infty} S_m (z) \in \overline{\CalP}
      \qquad \text{for all } z \in \CalZ
      .
    \]
    Since $\widehat{\XX}$ and hence also $\overline{\CalP}$ is metrizable,
    \cite[Proposition~8.1.10]{cohnMeasureTheory2013} implies that
    $S : \CalZ \to \overline{\CalP}$ is (Borel) measurable,
    as the pointwise limit of a sequence of measurable functions.
    Since $\CalZ$ and $\overline{\CalP}$ are Polish spaces,
    it now follows by \cite[Theorem~13.2.1(b')]{DudleyRealAnalysisProbability} that
    $S(\CalZ) \subset \overline{\CalP}$ is an analytic set.

    Next, we define a Borel measure $\nu$ on $\overline{\CalP}$ via $\nu := \gamma \circ S^{-1}$
    (the push-forward of $\gamma$ via $S$, interpreted as a map $S : \CalZ \to \overline{\CalP}$),
    where $\gamma := \prod_{n=1}^{\infty} \gamma_n$ is the product measure of the normalized
    counting measures $\gamma_n$ on $\SC_n$.
    We note that this product measure exists by \cite[Section~8.2]{DudleyRealAnalysisProbability}
    (or alternatively, \cite[Theorem~3.5.1]{BogachevMeasureTheory}).

    \smallskip{}

    Since $S(\CalZ) \subset \overline{\CalP}$ is an analytic set,
    \cite[Theorem~8.4.1]{cohnMeasureTheory2013}
    implies that $S(\CalZ)$ is $\nu$-measurable, meaning that there exist Borel sets
    $\widehat{A}, \widehat{B} \subset \overline{\CalP}$ with
    $\widehat{A} \subset S(\CalZ) \subset \widehat{B}$ and
    $\nu (\widehat{B} \setminus \widehat{A}\,) = 0$.
    Since $S(\CalZ) \subset \widehat{B}$, we have $S^{-1}(\widehat{B}) = \CalZ$ and thus
    \[
      \nu(\widehat{B}\,)
      = \gamma(S^{-1}(\widehat{B}\,))
      = \gamma(\CalZ)
      = 1,
    \]
    and therefore also $\nu(\widehat{A}\,) = 1$, where $\widehat{A} \subset S(\CalZ) \subset \SC$.

    To complete the proof of Step~2, note that
    $\widehat{A} \subset \overline{\CalP} \subset \widehat{\XX}$
    is a Borel set, so that $A := \widehat{A} \cap \XX$ is a Borel subset of $\XX$,
    by \cite[Lemma~6.2.4]{BogachevMeasureTheory}.
    Note that in fact $\widehat{A} \subset S(\CalZ) \subset \SC \subset \XX$,
    so that $\widehat{A} = \widehat{A} \cap \XX = A$ is a Borel subset of $\XX$.
    Finally, the sequence of random variables $(Y_n)_{n \in \N}$ from \Cref{definition_measure_SC_n}
    follows the distribution $\gamma$.
    This implies
    \begin{align*}
      \mu(A)
      &= \PP \biggl(\,\sum_{n=1}^{\infty} \frac{Y_n}{C_0 \, n^\alpha} \in A\biggr)
       = \PP \Bigl(S \bigl( (Y_n)_{n \in \N} \bigr) \in A \Bigr)
       = \PP \bigl( (Y_n)_{n \in \N} \in S^{-1}(A)\bigr) \\
      &= \PP \bigl( (Y_n)_{n \in \N} \in S^{-1}(\widehat{A} \,)\bigr)
       = \gamma \bigl( S^{-1}(\widehat{A} \,) \bigr)
       = \nu (\widehat{A} \,)
       = 1
      .
    \end{align*}

    Finally, since $A = \widehat{A} \subset S(\CalZ) \subset \SC$,
    we see that $\XX \setminus \SC \subset \XX \setminus A$ and hence
    \[
      \mu^\ast (\XX \setminus \SC)
      \leq \mu^\ast (\XX \setminus A)
      = \mu (\XX \setminus A)
      = \mu(\XX) - \mu(A)
      = 0
      .
    \]

    \medskip{}
    \emph{Step 3:}
    By Step~2, we have $\mu^\ast (\XX \setminus \SC) = 0$ and hence
    \[
      1
      = \mu (\XX)
      = \mu^\ast(\XX)
      \leq \mu^\ast (\SC) + \mu^\ast (\XX \setminus \SC)
      = \mu^\ast (\SC)
      \leq \mu^\ast (\XX)
      =    \mu(\XX)
      =    1,
    \]
    which easily implies $\om{\mu}(\SC) = 1$.
\end{proof}

\subsection{Proof of \texorpdfstring{\cref{restriction_inherits_growht_order}}{Lemma~\ref{restriction_inherits_growht_order}}
}
\label{sec_proof_of_restriction_inherits_growht_order}

In order to prove \cref{restriction_inherits_growht_order},
we need the following two elementary results regarding restrictions
of measures and of sigma-algebras.

The following \cref{restrict_measure_to_subset_of_full_outer_measure} states that
for subsets of $\SC$, the outer measure induced by the restriction of $\mu$ to $\SC$
agrees with the outer measure induced by $\mu$.
The proof is given in \cref{sub:proof_of_restrict_measure_to_subset_of_full_outer_measure} below.

\begin{lemma}\label{restrict_measure_to_subset_of_full_outer_measure}
    Let $(\XX,\mathcal{F})$ be a measurable space and let $\mu\colon \mathcal{F}\to [0,1]$
    be a probability measure on $\XX$.
    Let $\SC \subset \XX$ be such that $\om{\mu}(\SC)=1$.
    Then the measure $\nu$ defined via
    \begin{equation}
        \label{eqn_definition_restricted_measure}
       \nu(A \cap \SC) := \mu(A)
       \quad \text{ for $A\in \mathcal{F}$}
    \end{equation}
    is a well-defined probability measure on $\mathcal{F}\sigmacap\SC$.
    Furthermore,
    \begin{equation}
        \label{eqn_outer_measure_and_restriction}
        \om{\nu}(B) = \om{\mu}(B)
        \quad \text{ for all $B\subset \SC$}.
    \end{equation}
\end{lemma}

The following \cref{Borel_intrinsic_is_Borel_trace} shows that the intrinsic Borel sigma-algebra
on $\SC$ is equal to the trace of the Borel sigma-algebra on $\XX$.
Here, for a quasi-metric space $(\XX,\varrho)$, we write $\Borel(\XX,\varrho)$
for the Borel $\sigma$-algebra generated by the topology induced by $\varrho$.
The proof is given in \cref{sub:proof_of_Borel_intrinsic_is_Borel_trace} below.

\begin{lemma}\label{Borel_intrinsic_is_Borel_trace}
   Let $(\XX,\metr)$ be a quasi-metric space and let $\emptyset \neq \SC \subset \XX$.
   Then 
   \begin{equation*}
       \Borel\bigl( \SC,\metr\vert_{\SC\times \SC}\bigr) 
       = \Borel\bigl(\XX,\metr\bigr) \sigmacap \SC.
   \end{equation*}
\end{lemma}

Using the previous two auxiliary results, the proof of \cref{restriction_inherits_growht_order} is now easy.

\begin{proof}[Proof of \cref{restriction_inherits_growht_order}]
    First of all, \cref{restrict_measure_to_subset_of_full_outer_measure} shows that
    the measure $\nu$ is a well-defined probability measure on $\mathcal{F}\sigmacap \SC$.

    \medskip{}

    \emph{Part (i):}
    By \cref{Borel_intrinsic_is_Borel_trace}, we have
    $\Borel(\XX,\metr)\sigmacap \SC = \Borel(\SC,\metr\vert_{\SC\times \SC}) = \Borel(\SC)$.
    Part (i) follows from this by noticing that $\mathcal{F} = \Borel(\XX, \metr) = \Borel(\XX)$
    in the definition of $\nu$.

    \medskip{}

    \emph{Part (ii):}
    Assuming \eqref{eqn_restrictions_growth_condition_mu},
    we show \eqref{eqn_restrictions_growth_condition_nu}.
    Using \eqref{eqn_outer_measure_and_restriction} at (a)
    and the monotonicity of the outer measure at (b), we get
    \begin{equation*}
        \om{\nu}\bigl(\ball\bigl(x,r\mid \SC\bigr)\bigr)
        = \om{\nu}\bigl(\ball\bigl(x,r\mid \XX\bigr)\cap \SC\bigr)
        \overset{(a)}{=} \om{\mu}\bigl(\ball\bigl(x,r\mid \XX\bigr)\cap \SC\bigr)
        \overset{(b)}{\le} \om{\mu}\bigl(\ball\bigl(x,r\mid \XX\bigr)\bigr).
    \end{equation*}
    Now the growth condition \eqref{eqn_restrictions_growth_condition_mu} for $\mu$
    implies the growth condition \eqref{eqn_restrictions_growth_condition_nu} for $\nu$.
\end{proof}

\subsubsection{Proof of \texorpdfstring{\cref{restrict_measure_to_subset_of_full_outer_measure}}{Lemma~\ref{restrict_measure_to_subset_of_full_outer_measure}}}
\label{sub:proof_of_restrict_measure_to_subset_of_full_outer_measure}

\begin{proof}[Proof of \cref{restrict_measure_to_subset_of_full_outer_measure}]
    In order to show the well-definedness, let $A,A' \in \mathcal{F}$ with $A\cap \SC = A'\cap \SC$.
    We have to show that $\mu(A) = \mu(A')$.
    We have
    \[
      (A\setminus A') \cap \SC
      = A \cap (A')^{c} \cap \SC
      = (A\cap \SC) \cap (A')^{c}
      = (A'\cap \SC) \cap (A')^{c}
      \subset A' \cap (A')^{c}
      = \emptyset.
    \]
    Hence $(A\setminus A') \subset \SC^{c}$.
    This implies that $\SC \subset (A\setminus A')^{c}$. 
    Thus,
    \[
      \mu\bigl((A\setminus A')^{c}\bigr)
      =\om{\mu}\bigl((A\setminus A')^{c}\bigr)
      \geq \om{\mu}(\SC)
      =1.
    \]
    Therefore, since $\mu$ is a measure, we have $\mu(A\setminus A')=0$.
    We deduce that
    \begin{equation*}
       \mu(A) 
       = \mu(A\cap A') + \mu(A\setminus A')
       = \mu(A\cap A').
    \end{equation*}
    Since the argument is symmetric in $A,A'$,
    swapping the roles of $A$ and $A'$, we infer that we also have $\mu(A')=\mu(A\cap A')$.
    Therefore, we have $\mu(A) = \mu(A')$.
    Hence $\nu$ is well-defined.

    We now prove \eqref{eqn_outer_measure_and_restriction}.
    Let $B\subset \SC$.
    Then we have for every $A\in \mathcal{F}$ that $B \subset A$ if and only if $B\subset A\cap \SC$.
    Using the definition of the outer measures $\om{\mu}$ at (a),
    the identity \eqref{eqn_definition_restricted_measure} at (b),
    the just proven equivalence at (c),
    the definition of $\mathcal{F} \sigmacap \SC$ at (d)
    and the definition of the outer measure $\om{\nu}$ at (e),
    we infer that
    \begin{align*}
        \om{\mu}(B) 
        &\overset{(a)}{=} \inf \bigl\{ \mu(A) : A\in \mathcal{F},\, A\supseteq B\bigr\}\\
        &\overset{(b)}{=} \inf \bigl\{ \nu(A \cap \SC) : A\in \mathcal{F},\, A\supseteq B \bigr\}\\
        &\overset{(c)}{=} \inf \bigl\{ \nu(A \cap \SC) : A\in \mathcal{F},\, A \cap \SC\supseteq B \bigr\}\\
        &\overset{(d)}{=} \inf
                          \bigl\{
                            \nu\bigl(\widetilde{A}\bigr)
                            \,\,:\,\,
                            \widetilde{A} \in \mathcal{F} \sigmacap \SC,\, \widetilde{A} \supseteq B
                          \bigr\}\\
        &\overset{(e)}{=} \om{\nu}(B).
        \qedhere
    \end{align*}
\end{proof}

\subsubsection{Proof of \texorpdfstring{\cref{Borel_intrinsic_is_Borel_trace}}{Lemma~\ref{Borel_intrinsic_is_Borel_trace}}}
\label{sub:proof_of_Borel_intrinsic_is_Borel_trace}

In order to prove \cref{Borel_intrinsic_is_Borel_trace},
we first gather some results on traces of quasi-metrics, topologies, and $\sigma$-algebras.
Let us introduce some notation.
For a topology $\tau$ on a space $\XX$, we will write $\Borel(\XX)$, $\Borel(\tau)$,
or $\Borel(\XX,\tau)$ to denote the Borel $\sigma$-algebra on $\XX$.
For a quasi-metric space $(\XX,\varrho)$, we will write $\tau(\XX,\varrho)$
for the topology generated by $\varrho$.

The following \cref{technical_lemma_traces_Borel_sigma_algebra} establishes
that the trace Borel $\sigma$-algebra is the $\sigma$-algebra induced by the trace topology.
See \cite[Lemma 6.2.4]{BogachevMeasureTheory} for a proof.

\begin{lemma}\label{technical_lemma_traces_Borel_sigma_algebra}
    Let $(\XX,\tau)$ be a topological space and let $\SC$ be a subset equipped with the trace topology $\tau\sigmacap \SC$.
    Then
    \begin{equation*}
       \Borel(\XX,\tau) \sigmacap \SC 
       =  \Borel(\SC, \tau\sigmacap \SC).
    \end{equation*}
\end{lemma}

That the trace of the topology induced by a metric agrees with the topology induced
by the trace of the metric is stated in the following \cref{lem:trace_of_topology_induced_by_metric}.
See \cite[Theorem 5.1 in chapter IX]{DugundjiTopology}.

\begin{lemma}\label{lem:trace_of_topology_induced_by_metric}
   Let $(\XX,d)$ be a metric space and let $\emptyset\ne\SC \subset \XX$.
   Then
   \[
       \tau\big(\SC, d\big\vert_{\SC\times \SC}\big)
       = \tau(\XX,d) \sigmacap\SC.
   \]
\end{lemma}

In order to lift the previous result to quasi-metric spaces,
we will use that there always exists a metric that induces the same topology.
The following \cref{quasi_metric_Heinonen_result} is a weaker version
of \cite[Proposition 14.5]{heinonenLecturesAnalysisMetric2001}.

\begin{lemma} \label{quasi_metric_Heinonen_result}
 Let $(\XX,\varrho)$ be a quasi-metric space.
 Then there exists $\eps>0$ and a metric $d$ on $\XX$
 such that $d$ is bi-Lipschitz equivalent to $\varrho^{\eps}$.
\end{lemma}

The following \cref{lem:quasi_metric_bi_Lipschiz_equiv_same_topo}
and \cref{lem:quasi_metric_to_some_power_same_topo} state that for a quasi-metric $\varrho$,
the metric $d$ found in \cref{quasi_metric_Heinonen_result} induces the same topology.

\begin{lemma}
    \label{lem:quasi_metric_bi_Lipschiz_equiv_same_topo}
    Let $\XX$ be a set and let $\varrho,\tilde{\varrho}$ be two quasi-metrics on $\XX$.
    Assume that $\varrho$ and $\tilde{\varrho}$ are bi-Lipschitz equivalent.
    Then
    \begin{equation*}
        \tau(\XX,\varrho) = \tau(\XX,\tilde{\varrho})
    \end{equation*}
\end{lemma}

\begin{proof}
    We prove ``$\subseteq$'' and ``$\supseteq$'' separately.
    To prove ``$\subseteq$'', let $V \in \tau(\XX, \varrho)$ and let $x\in V$.
    Then, there exists $r>0$ such that $\{y\in \XX: \varrho(y,x)<r\} \subset V$.
    By assumption, there exists $C>0$ such that $\varrho(x,y) \leq C \cdot \tilde{ \varrho}(x,y)$
    for all $x,y\in X$.
    Hence
    \(
      \{y\in \XX : \tilde{\varrho}(x,y) < \frac{r}{C}\}
      \subset \{y\in \XX: \varrho(x,y)<r\} \subset V
      .
    \)
    Hence $V \in \tau(\XX,\tilde{\varrho})$.
    Thus $\tau(\XX,\varrho) \subset \tau(\XX,\tilde{\varrho})$.
    The converse set inclusion ``$\supseteq$'' follows by symmetry from ``$\subseteq$''
    by swapping the roles of $\varrho$ and $\tilde{\varrho}$.
\end{proof}

\begin{lemma}\label{lem:quasi_metric_to_some_power_same_topo}
   Let $(\XX,\varrho)$ be a quasi-metric space and let $\eps>0$.
   Then $\varrho^{\eps}$ is a quasi-metric on $\XX$ and $\tau(\XX,\varrho) = \tau(\XX,\varrho^{\eps})$.
\end{lemma}

\begin{proof}
    It is clear that $\varrho^{\eps}$ is symmetric and definite.
    To show the quasi-triangle inequality for $\varrho^{\eps}$,
    let $x,y,z \in \XX$, and let $C$ be a triangle constant for $\varrho$.
    Then, using the elementary estimate
    $(a + b)^\eps \leq (2 \max \{ a, b \}) \leq 2^\eps (a^\eps + b^\eps)$
    for $a,b \geq 0$, we see
    \[
        \varrho^{\eps}(x,y) 
        \le \big( C \varrho(x,z) + C \varrho(z,y)\big)^{\eps}
        \le C^{\eps} \cdot 2^{\eps} \cdot \big( \varrho^{\eps}(x,z) + \varrho^{\eps}(z,y)\big).
    \]
    Hence $\varrho^{\eps}$ is a quasi-metric.

    We now show that $\varrho$ and $\varrho^{\eps}$ induce the same topology.
    For all $x\in \XX$ and all $r>0$, we have
    \begin{equation*}
        \big\{ y \in \XX : \varrho(x,y) <r \big\}
        = \big\{ y \in \XX : \varrho(x,y)^{\eps} < r^{\eps} \big\}.
    \end{equation*}
    Hence $\varrho$ and $\varrho^{\eps}$ generate the same family of ``open'' balls,
    and this implies $\tau(\XX,\varrho) =\tau(\XX,\varrho^{\eps})$.
\end{proof}

With \cref{quasi_metric_Heinonen_result}, \cref{lem:quasi_metric_bi_Lipschiz_equiv_same_topo},
and \cref{lem:quasi_metric_to_some_power_same_topo} in place,
we can now lift the result from \cref{lem:trace_of_topology_induced_by_metric}
to the setting of quasi-metric spaces.

\begin{lemma}\label{lemma_trace_topology_induced_by_trace_quasi_metric}
  Let $(\XX,\varrho)$ be a quasi-metric space and let $\emptyset \ne\SC\subset \XX$ be a subset.
  Then
  \begin{equation*}
      \tau\big(\SC, \varrho\big\vert_{\SC\times \SC}\big)
      = \tau(\XX,\varrho) \sigmacap\SC. 
  \end{equation*}
\end{lemma}

\begin{proof}
    By \cref{quasi_metric_Heinonen_result}, there exists a metric $d$ and $\eps>0$
    such that $\varrho^{\eps}$ and $d$ are bi-Lipschitz equivalent.
    By checking the definition, it follows that also
    $\varrho^{\eps}\big\vert_{\SC\times \SC}$ and $d\big\vert_{\SC\times \SC}$
    are bi-Lipschitz equivalent.
    Furthermore, we have
    \(
      \varrho^{\eps}\big\vert_{\SC\times \SC}
      = \big( \varrho\big\vert_{\SC\times \SC}\big)^{\eps}
      .
    \)
    Using \cref{lem:quasi_metric_to_some_power_same_topo} at (a),
    \cref{lem:quasi_metric_bi_Lipschiz_equiv_same_topo} at (b),
    and that $d$ is a metric together with \cref{lem:trace_of_topology_induced_by_metric} at (c),
    we obtain
    \begin{equation*}
        \tau\big(\SC,\varrho\big\vert_{\SC\times \SC}\big)
        \overset{(a)}{=}\tau\big(\SC,\varrho^{\eps}\big\vert_{\SC \times \SC}\big)
        \overset{(b)}{=}\tau\big(\SC,d\big\vert_{\SC\times \SC}\big)
        \overset{(c)}{=}\tau\big(\XX,d\big)\sigmacap\SC 
        \overset{(b)}{=}\tau\big(\XX,\varrho^{\eps}\big)\sigmacap\SC 
        \overset{(a)}{=}\tau(\XX,\varrho)\sigmacap \SC.
        \qedhere
    \end{equation*}
\end{proof}

With \cref{technical_lemma_traces_Borel_sigma_algebra}
and \cref{lemma_trace_topology_induced_by_trace_quasi_metric} established,
the proof of \cref{Borel_intrinsic_is_Borel_trace} is now straightforward.

\begin{proof}[Proof of \cref{Borel_intrinsic_is_Borel_trace}]
    Using \cref{lemma_trace_topology_induced_by_trace_quasi_metric} at (a)
    and \cref{technical_lemma_traces_Borel_sigma_algebra} at (b), we deduce that
    \begin{equation*}
     \Borel\bigl( \tau(\SC,\metr\vert_{\SC\times \SC})\bigr) 
     \overset{(a)}{=} \Borel\bigl(\tau(\XX,\varrho) \sigmacap \SC\bigr)
     \overset{(b)}{=} \Borel\bigl(\tau(\XX,\metr)\bigr) \sigmacap \SC.
     \qedhere
    \end{equation*}
\end{proof}

%% file: parts/technical_proofs_function_spaces.tex
\section{Proofs of technical results on function spaces}
\label{sec:proofs_of_technical_results_on_function_spaces}

We believe that the results in this section are folklore, but we provide proofs since we could not locate a convenient reference.

\subsection{Entropy numbers and Minkowski dimension}

Following \cite[Definition 1 in Section 1.3.1]{edmundsFunctionSpacesEntropy1996}, for two quasi-Banach spaces $U$ and $V$, a linear map $T\colon U\to V$, and a number $k\in \NN$, the $k$-th \emph{(dyadic) entropy number of $T$} is defined as
\begin{equation}
    \en_k(T) 
    := \en_k(T\colon U\to V)
    := \inf \biggl\{ \eps>0 : \begin{array}{c}
            \text{there exist $(v_i)_{i=1}^{2^{k-1}} \subset V$}\\
            \text{such that $T(\ball_U) \subset \bigcup_{i=1}^{2^{k-1}} (v_i + \eps \ball_V)$} 
   \end{array}
       \biggr\}
       \label{eq:def_entropy_numbers}
\end{equation}
with $\ball_U$ denoting the closed unit ball in $U$ (and similar for $B_{V}$).
The following \cref{entropy_numbers_basic_properties} is contained in \cite[Lemma 1 in Section 1.3.1]{edmundsFunctionSpacesEntropy1996}.

\begin{lemma}
    \label{entropy_numbers_basic_properties}
 Let $U,V,W$ be quasi-Banach spaces and let $S\colon U\to V$ and $T\colon V\to W$ be bounded linear maps.
 \begin{enumerate}[label=(\roman*)]
     \item We have $\en_1(T) \le \norm{T}$.
    \item For all $k,\ell\in \NN$ we have $\en_{k+\ell-1}(TS) \le \en_{k}(T) \en_{\ell}(S)$.
 \end{enumerate}
\end{lemma}

\begin{lemma}
    \label{entropy_numbers_to_minkowski_dim_upper}
    Let $U,V$ be quasi-Banach spaces and let $T\colon U\to V$ be a compact linear map and let $t\in(0,\infty)$.
    Assume that 
    \begin{equation*}
       \limsup_{k\to \infty} \bigl( k^{t} \cdot \en_k(T)\bigr) <\infty.
    \end{equation*}
    Then $\bigl(T(\ball_U),\, \norm{\cdot \mid V}\bigr)$ is totally bounded and we have $\umd\bigl(T(\ball_U)\bigr) \le \frac{1}{t}$.
\end{lemma}

\begin{proof}
    Let $\SC:= T(\ball_U)$
    and let $K$ be the modulus of concavity of $V$.
    Since $T$ is compact, $\SC$ is totally bounded.

    We start the proof of the remaining claim with the following observation, which relates entropy numbers to covering numbers.
    Let $k\in \NN$ and $\eps>0$ such that 
    \begin{equation*}
        \frac{\eps}{2K} > \en_k(T).
    \end{equation*}
    Then, by definition of entropy numbers, there exist $v_1,\dots, v_{2^{k-1}} \in V$ such that
    \[
        \SC \subset \bigcup_{i=1}^{2^{k-1}} \ball(v_i,\eps/2K).
    \]
    By \cref{lem:InternalVSExternalCoveringNumbers}, we get
    \begin{equation*}
       \cn(\SC,\eps) 
       \le \extcn\Bigl( \SC, V, \frac{\eps}{2K}\Bigr) \le 2^{k-1}.
    \end{equation*}
    
    Using the just derived relation, we now prove the claim.
    By assumption, there exist $k_0\in \NN$ and $C>0$ such that $\en_k(T) < C\cdot k^{-t}$ for all $k\ge k_0$.
    For $\eps>0$ let
    \begin{equation}
        \label{eqn_11_1}
        k:= k(\eps) := \left\lceil \Bigl( \frac{2CK}{\eps}\Bigr)^{\frac{1}{t}}\right\rceil.
    \end{equation}
    Then, for every sufficiently small $\eps > 0$ (depending on $C,K,t$), we have $k(\eps)\ge k_0$.
    Furthermore, we have
    \begin{equation*}
        \frac{\eps}{2K} \ge C\cdot k^{-t} > \en_k(T).
    \end{equation*}
    Therefore, by the preliminary argument, we have $\cn(\SC,\eps)\le 2^{k-1}$.
    Using \eqref{eqn_11_1}, we get 
    \begin{equation*}
        2^{k-1} 
        =\exp(\ln(2)(k-1))
        \le \exp(\tilde{c} \cdot \eps^{-1/t})
    \end{equation*}
    for some $\tilde{c}>0$ (depending on $C,K,t$).
    We deduce that
    \begin{equation*}
        \cn(\SC,\eps) \le \exp(\tilde{c}\cdot \eps^{-1/t}) 
        \quad \text{for all sufficiently small $\eps > 0$}.
    \end{equation*}
    This easily implies the claim.
\end{proof}

\begin{lemma}
    \label{entropy_numbers_to_minkowski_dim_lower}
    Let $U,V$ be quasi-Banach spaces and let $T\colon U\to V$ be a compact linear map and let $t\in(0,\infty)$.
    Assume that 
    \begin{equation*}
       \liminf_{k\to \infty} \bigl( k^{t} \cdot \en_k(T)\bigr) >0.
    \end{equation*}
    Then $\bigl(T(\ball_U), \norm{\cdot \mid V}\bigr)$ is totally bounded and we have $\lmd\bigl(T(\ball_U)\bigr) \ge \frac{1}{t}$.
\end{lemma}

\begin{proof}
    Let $\SC:= T(\ball_U)$.
    Since $T$ is compact, $\SC$ is totally bounded.

    We start the proof of the remaining claim with the following observation. 
    By definition of entropy numbers, if $\eps < \en_k(T)$,
    then $\SC$ cannot be covered by $2^{k-1}$ balls of radius $\eps$ (centered anywhere in $V$),
    so that $\cn(\SC,\eps) > 2^{k-1}$.
    
    Using this relation, we now prove the claim.
    By assumption, there exists $k_0\in \NN$ and $c>0$ such that $\en_k(T) > c\cdot k^{-t}$ for all $k\ge k_0$.
    For $\eps>0$ let
    \begin{equation}
        \label{eqn_11_2}
        k:= k(\eps) := \left\lfloor \Bigl( \frac{c}{\eps}\Bigr)^{\frac{1}{t}}\right\rfloor.
    \end{equation}
    Then, for all sufficiently small $\eps > 0$ (depending on $c,t$), we have $k(\eps)\ge k_0$.
    Furthermore, we have
    \begin{equation*}
        \en_k(T) > c \cdot k^{-t} \ge c \cdot \Bigl( (c/\eps)^{1/t} \Bigr)^{-t} = \eps.
    \end{equation*}
    Therefore, by the preliminary argument, we have $\cn(\SC,\eps) > 2^{k-1}$.
    Using \eqref{eqn_11_2}, we get 
    \begin{equation*}
        2^{k-1} 
        =\exp(\ln(2)(k-1))
        \ge \exp(\tilde{c} \cdot \eps^{-1/t})
    \end{equation*}
    for some $\tilde{c}>0$ and all sufficiently small $\eps > 0$ (both depending on $c,t,k_0$).
    We deduce that
    \begin{equation*}
        \cn(\SC,\eps) \ge \exp(\tilde{c}\cdot \eps^{-1/t}) 
        \quad \text{for all sufficiently small $\eps > 0$}.
    \end{equation*}
    This easily implies the claim.
\end{proof}

\subsection{Embeddings for spaces on domains}

Let $\big(V(\RR^d), \norm{\cdot \mid V(\RR^d)}\big)$ be a quasi-normed space
where $V(\RR^d)$ is a subset of $\td(\RR^d)$.
Let ${\emptyset \ne \Omega \subset \RR^d}$ be an open subset.
We define the space
\begin{equation*}
        V(\Omega) := \big\{ f\in \dis(\Omega) :
        \text{ there exists $g \in V(\RR^d)$ with $g\big\vert_{\Omega}=f$} \big\},
\end{equation*}
and the quasi-norm
\begin{equation*}
    \norm{f \mid V(\Omega)}
    := \inf\big\{ \norm{g \mid V(\RR^d)} : \text{$g\in V(\RR^d)$ with $g\big\vert_{\Omega}=f$}\big\}.
\end{equation*}

\begin{lemma}\label{transfer_embeddings_RR^d_to_Omega_1}
    If the embedding $V_1(\RR^d) \embeds V_2(\RR^d)$ is continuous,
    then the embedding \\$V_1(\Omega)\embeds V_2(\Omega)$ is also continuous.
\end{lemma}

\begin{proof}
    Let $f \in V_1(\Omega)$.
    Let $\eps>0$ be arbitrary.
    By definition of $\norm{\cdot \mid V_1(\Omega)}$,
    there exists $\bar{f} \in V_1(\RR^d)$ with $\bar{f}\big\vert_{\Omega}=f$
    and $\norm{\bar{f}\mid V_1(\RR^d)} \le \norm{ f \mid V_1(\Omega)}+\eps$.
    We have $\bar{f} \in V_1(\RR^d) \subset V_2(\RR^d)$.
    Hence, by definition of $V_2(\Omega)$, we have $f = \bar{f}\big\vert_{\Omega} \in V_2(\Omega)$.
    Using the definition of $\norm{\cdot \mid V_2(\Omega)}$ at (a),
    and the embedding assumption at (b), we infer that
    \begin{equation*}
        \norm{f\mid V_2(\Omega)} 
        \overset{(a)}{\le} \norm{ \bar{f} \mid V_2(\RR^d)}
        \overset{(b)}{\lesssim} \norm{\bar{f} \mid V_1(\RR^d)}
        \le \norm{f \mid V_1(\Omega)} + \eps
    \end{equation*}
    with constants independent of $f$. 
    Taking $\eps\to 0$, we deduce the claim.
\end{proof}

\subsection{Proof of \texorpdfstring{\cref{minkowski_dim_isotropic}}{Lemma~\ref{minkowski_dim_isotropic}}}
\label{sec_proof_of_minkowski_dim_isotropic}

\begin{lemma}
    \label{embeddings_isB_isF_Omega}
 Let $\emptyset \ne \Omega\subset \RR^d$ be an open subset.
 Let $s\in \RR$, let $0<p<\infty$ and let $0<q\le \infty$.
 Then we have the continuous embeddings
 \begin{equation*}
     B^{s}_{p, \min\{p,q\}}(\Omega)
         \embeds F^s_{p,q}(\Omega)
         \embeds B^s_{p,\max\{p,q\}}(\Omega).
 \end{equation*}
\end{lemma}
\begin{proof}
In the case $\Omega =\RR^d$, this was proved in \cite[Section 2.3.2 Proposition 2]{triebelTheoryFunctionSpaces1983}.
    The claim then follows by applying \cref{transfer_embeddings_RR^d_to_Omega_1}.
\end{proof}

\begin{lemma}\label{entropy_numbers_isotropic}
Let $\emptyset \ne \Omega \subset \RR^d$ be an open and bounded subset.
Let $A_1,A_2\in \{B,F\}$.
Let $s_1,s_2\in \RR$ and let $0<p_1,q_1,p_2,q_2 \le \infty$ (with $p_i<\infty$ if $A_i=F$).
Assume that 
\begin{equation*}
   s_1 > s_2 
   \quad \text{and} \quad
    s_1 - \frac{d}{p_1} > s_2 -\frac{d}{p_2}.
\end{equation*}
Then the embedding $\iota \colon (A_1)^{s_1}_{p_1,q_1}(\Omega) \embeds (A_2)^{s_2}_{p_2,q_2}(\Omega)$
is compact and there exist constants $C_1, C_2>0$ such that the dyadic entropy numbers of $\iota$ satisfy
    \begin{equation*}
        C_1 \cdot k^{-\frac{s_1-s_2}{d}} 
        \le \en_k(\iota) 
        \le C_2 \cdot k^{-\frac{s_1-s_2}{d}} 
        \quad 
        \text{ for all $k\in \NN$}.
    \end{equation*}
\end{lemma}

\begin{proof}
    In the case $A_1=A_2=B$, this is \cite[Theorem 1.97]{triebelTheoryFunctionSpaces2006}.
    The general case follows from this case using \cref{entropy_numbers_basic_properties}
    and \cref{embeddings_isB_isF_Omega}, as we will now show. 
    
    We first show the upper bound.
    By \cref{embeddings_isB_isF_Omega} there exist $\tilde{q}_1, \tilde{q}_2 \in (0,\infty]$
    such that we have the embeddings
    \begin{equation*}
        \iota_1 \colon (A_1)^{s_1}_{p_1,q_1}(\Omega) 
        \embeds B^{s_1}_{p_1,\tilde{q}_1}(\Omega) 
        \quad \text{and} \quad
        \iota_2 \colon B^{s_2}_{p_2,\tilde{q}_2}(\Omega) 
        \embeds (A_2)^{s_2}_{p_2,q_2}(\Omega).
    \end{equation*}
    Consider the sequence of embeddings
    \begin{equation*}
        \iota_A \colon (A_1)^{s_1}_{p_1,q_1}(\Omega) 
        \overset{\iota_1}{\embeds} B^{s_1}_{p_1,\tilde{q}_1}(\Omega) 
        \overset{\iota_B}{\embeds} B^{s_2}_{p_2,\tilde{q}_2}(\Omega) 
        \overset{\iota_2}{\embeds} (A_2)^{s_2}_{p_2,q_2}(\Omega).
    \end{equation*}
    Then, using \cref{entropy_numbers_basic_properties}, we infer that
    \begin{equation*}
       \en_k(\iota_A) 
       = \en_k(\iota_2 \iota_B \iota_1)
       \le \en_k(\iota_2 \iota_B) \en_1(\iota_1)
       \le \en_1(\iota_2) \en_k(\iota_B) \en_1(\iota_1)
       \le \norm{\iota_2} \norm{\iota_1} \en_k(\iota_B).
    \end{equation*}
    The upper bound for $\en_k(\iota_B)$ then implies the upper bound for $\en_k(\iota_A)$.

    We show the lower bound. 
    By \cref{embeddings_isB_isF_Omega} there exist $\tilde{q}_1, \tilde{q}_2 \in (0,\infty]$
    such that we have the embeddings
    \begin{equation*}
        \iota_1 \colon B^{s_1}_{p_1,\tilde{q}_1}(\Omega) 
        \embeds (A_1)^{s_1}_{p_1,q_1}(\Omega) 
        \quad \text{and} \quad
        \iota_2 \colon 
        (A_2)^{s_2}_{p_2,q_2}(\Omega) 
        \embeds B^{s_2}_{p_2,\tilde{q}_2}(\Omega) 
    \end{equation*}
    Consider the sequence of embeddings
    \begin{equation*}
        \iota_B \colon B^{s_1}_{p_1,\tilde{q}_1}(\Omega) 
        \overset{\iota_1}{\embeds} (A_1)^{s_1}_{p_1,q_1}(\Omega)
        \overset{\iota_A}{\embeds} (A_2)^{s_2}_{p_2,q_2}(\Omega) 
        \overset{\iota_2}{\embeds} B^{s_2}_{p_2,\tilde{q}_2}(\Omega).
    \end{equation*}
    Then, using \cref{entropy_numbers_basic_properties}, we infer that
    \begin{equation*}
       \en_k(\iota_B) 
       = \en_k(\iota_2 \iota_A \iota_1)
       \leq \en_k(\iota_2 \iota_A) \en_1(\iota_1)
       \leq \en_1(\iota_2) \en_k(\iota_A) \en_1(\iota_1)
       \leq \norm{\iota_2} \norm{\iota_1} \en_k(\iota_A).
    \end{equation*}
    The lower bound for $\en_k(\iota_B)$ then implies the lower bound for $\en_k(\iota_A)$
    (using also that $\iota_1, \iota_2 \ne 0$).
\end{proof}

\begin{proof}[Proof of \cref{minkowski_dim_isotropic}]
    The spaces are quasi-Banach spaces, see \cite[Remark 1.96]{triebelTheoryFunctionSpaces2006}.
    The remaining claim follows from \cref{entropy_numbers_isotropic}, \cref{entropy_numbers_to_minkowski_dim_lower}, and \cref{entropy_numbers_to_minkowski_dim_upper}.
\end{proof}

\subsection{Proof of \texorpdfstring{\cref{minkowski_dim_dms}}{Lemma~\ref{minkowski_dim_dms}}}
\label{sec_proof_of_minkowski_dim_dms}

\begin{lemma}
    \label{embeddings_dmsB_dmsF_Omega}
 Let $\emptyset \ne \Omega \subset \RR^d$ be an open subset.
 Let $s\in \RR$, let $0<p<\infty$ and let $0<q\le \infty$.
 Then we have the continuous embeddings
 \begin{equation*}
     S^{s}_{p,\min\{p,q\}}B(\Omega)
     \embeds S^{s}_{p,q}F(\Omega)
     \embeds S^{s}_{p,\max\{p,q\}}B(\Omega).
 \end{equation*}
\end{lemma}
\begin{proof}
    In the case $\Omega = \RR^d$, this is shown in \cite[Proposition 1.15]{triebelFunctionSpacesDominating2019}. 
    The claim now follows from \cref{transfer_embeddings_RR^d_to_Omega_1}.
\end{proof}

\begin{lemma}\label{entropy_numbers_dms}
    Let $\emptyset \ne \Omega \subset \RR^d$ be an arbitrary bounded domain.
    Let $A_1,A_2 \in \{B,F\}$.
    Let $s_1,s_2 \in \RR$ and let $0<p_1,q_1,p_2,q_2\le\infty$ (with $p_i <\infty$ if $A_i = F$).
    Assume that
    \begin{equation*}
       s_1 > s_2
       \quad \text{and}\quad
       s_1 - \frac{1}{p_1} > s_2 - \frac{1}{p_2}.
    \end{equation*}
    Then the embedding $\iota\colon S^{s_1}_{p_1,q_1}A_1(\Omega) \embeds S^{s_2}_{p_2,q_2}A_2(\Omega)$
    is compact and there exist constants $C_1,C_2>0$ and $u_1,u_2\ge0$ such that
    \begin{equation*}
        C_1 \cdot k^{-(s_1-s_2)}\cdot (\ln (e k))^{u_1}
        \le \en_k(\iota)
    \le C_2 \cdot k^{-(s_1-s_2)} \cdot (\ln (e k))^{u_2}
    \quad \text{for all $k\in \NN$.}
    \end{equation*}
\end{lemma}
\begin{proof}
    If $d=1$, we have $S^{s}_{p,q}A(\Omega)=A^{s}_{p,q}(\Omega)$ for $A\in \{B,F\}$
    and so the claim follows from \cref{entropy_numbers_isotropic}.
    Assume in the following that $d\ge 2$.
    In the case $A_1=A_2=B$, the claim is a consequence of
    \cite[Theorem 4.11]{vybiralFunctionSpacesDominating2006}.
    Indeed, the lower bound follows from
    \cite[Theorem 4.11 Part (ii)]{vybiralFunctionSpacesDominating2006}.
    The upper bound follows from \cite[Theorem 4.11 Part (iii)]{vybiralFunctionSpacesDominating2006},
    if ${\theta := s_1 - s_2 - \frac{1}{q_1} + \frac{1}{q_2} > 0}$, and if $\theta \leq 0$,
    it follows from \cite[Theorem 4.11 Part (iv)]{vybiralFunctionSpacesDominating2006}
    (for, e.g., $\eps:=1$).
    (We remark that strictly speaking, \cite[Theorem~4.11]{vybiralFunctionSpacesDominating2006} only
    provides bounds for the entropy numbers for the case $k \geq 2$,
    but it is easy to see that the bounds we state then hold for all $k \in \N$;
    this makes it necessary to use $\ln(e k)$ instead of $\ln(k)$.)

    Using the same arguments as in the proof of \cref{entropy_numbers_isotropic},
    the general case now follows from this case using \cref{entropy_numbers_basic_properties}
    and \cref{embeddings_dmsB_dmsF_Omega}. 
\end{proof}

\begin{proof}[Proof of \cref{minkowski_dim_dms}]
    Let $0<\delta<s_1-s_2$.
    It follows from \cref{entropy_numbers_dms} that
    \begin{equation*}
        k^{-(s_1-s_2)}
        \lesssim \en_k(\iota)
        \lesssim k^{-(s_1-s_2-\delta)}.
    \end{equation*}
    It follows from \cref{entropy_numbers_to_minkowski_dim_upper}
    that the generalized upper Minkowski dimension of the corresponding unit ball
    is upper bounded by $1/(s_1-s_2-\delta)$. 
    Taking $\delta\downarrow 0$, we infer that it is upper bounded by $1/(s_1-s_2)$.
    Similarly, using \cref{entropy_numbers_to_minkowski_dim_lower},
    we deduce that the generalized lower Minkowski dimension of the unit ball
    is lower bounded by $1/(s_1-s_2)$.
    Combining the upper and lower bound, we conclude that the Minkowski dimension of the unit ball
    is equal to $1/(s_1-s_2)$.
\end{proof}

%% file: parts/appendix.tex
\section{A compact convex set with distinct upper and lower power-exponential Minkowski dimensions}
\label{sec:example_lmd<umd}

In this section, we show that in infinite-dimensional spaces there exist convex and compact
subsets for which the upper and lower power-exponential Minkowski dimension differ.
More precisely, let ${0 < s < t < \infty}$.
Below, we provide an explicit construction of a set $\SC\subset \ell^{\infty}(\NN)$
satisfying ${\lmd(\SC) = s}$ and ${\umd(\SC) = t}$.

Let
\begin{equation*}
    \SC 
    := \bigl\{ x \in \ell^{\infty}(\NN) : \abs{x_i} \le a_i \text{ for all $i \in \NN$}\bigr\}
    = \prod_{i=1}^{\infty} [-a_i,a_i]
\end{equation*}
where the sequence $(a_i)_{i\in\NN}$ is defined as follows:
Let $a:= \frac{t}{s}$.
Then $a>1$.
Let
\begin{equation*}
    N_k := \lfloor 2^{t a^k} \rfloor
    \quad \text{for $k\in \NN$.}
\end{equation*}
We have
\begin{equation*}
    N_{k+1}-N_{k} = \lfloor 2^{t a^{k+1}} \rfloor - \lfloor 2^{t a^k}\rfloor
    \ge 2^{t a^{k+1}} - 1 - 2^{ta^k}
    = \bigl(2^{ta^k}\bigr)^a - 2^{ta^k} -1 \xrightarrow{k\to\infty} \infty.
\end{equation*}
Hence there exists $k_0\in \NN$ such that $N_{k+1}>N_k$ for all $k\ge k_0$.
We now define the sequence $(a_i)_{i\in\NN}$ as follows.
Let
\begin{equation*}
   a_i := \begin{cases}
       1,\quad &\text{for } i=1,\dots,N_{k_0},\\
       2^{-a^k},\quad &\text{for } k\ge k_0+1 \text{ and } N_{k-1}+1 \le i \le N_k.
   \end{cases}
\end{equation*}
Since $a>1$, the sequence $(a_i)_{i\in\NN}$ is non-increasing.
In addition, we have $a_i\to 0$ as $i\to \infty$.

Our main result in this section is the following \cref{umd_lmd_of_hilbert_cube_s_t}.
Its proof is deferred to the end of the section.
\begin{proposition}
    \label{umd_lmd_of_hilbert_cube_s_t}
    Let $\SC$ be as defined above.
    Then, considering $\SC$ as a subset of $\ell^{\infty}(\NN)$, we have
 \begin{equation*}
    \lmd(\SC) = s 
    \qquad \text{and} \qquad 
    \umd(\SC)=t.
 \end{equation*}
\end{proposition}

In our example, the quantity with the highest impact on the covering numbers
is the number of ``coordinates'' that need to be considered when trying to cover $\SC$
by balls of radius $\eps$. 
To fix notation, for every $\eps\in (0,1)$ let
\[
  n(\eps) := \max\{ n\in \NN : a_n > \eps\}.
\]
All remaining ``coordinates'', that is $i \in \NN$ with $i>n(\eps)$ can be dealt with
using a single ball of radius $\eps$ centered at the origin;
see \cref{bounds_for_covering_numbers_Hilbert_cube} for the details.

Before proving \cref{umd_lmd_of_hilbert_cube_s_t}, we first recall some results,
namely \cref{bounds_for_covering_numbers_Hilbert_cube} and \cref{Hilbert_cube_compact_convex}
on Hilbert cubes.
We believe these results to be folklore, but could not locate a reference. 
For the reader's convenience, we provide proofs.

\begin{lemma}
    \label{bounds_for_covering_numbers_Hilbert_cube}
    Let $\SC$ be as defined at the start of \cref{sec:example_lmd<umd} and
    consider $\SC$ as a subset of $\ell^{\infty}(\NN)$.
    Then for every $\eps \in (0,1)$ we have
    \begin{equation*}
        \prod_{i=1}^{n(\eps)} \frac{a_i}{\eps}
        \le \cn(\SC,\eps) 
        \le \prod_{i=1}^{n(\eps)} \Bigl(\frac{a_i}{\eps} +1\Bigr)  .
 \end{equation*}
\end{lemma}

\begin{proof}
    Fix $\eps\in (0,1)$.
    Because $(a_i)_{i\in\NN}$ is a null sequence, the set $\{ i\in \NN : a_i > \eps\}$ is finite.
    Furthermore, since $a_1=1>\eps$, the set is non-empty.
    Hence $n:=n(\eps)\in \NN$ is well-defined.
    
    We first show the upper bound.
    For each coordinate $i \in \{1, \dots, n\}$, the interval $[-a_i, a_i]$ has a length of $2a_i$.
    We can cover this interval using $k_i$ closed sub-intervals of radius $\eps$
    (which have length $2\eps$), where
    \[
      k_i
      = \left\lceil \frac{2a_i}{2\eps} \right\rceil
      = \left\lceil \frac{a_i}{\eps} \right\rceil \le \frac{a_i}{\eps} + 1.
    \]
    By taking the Cartesian product of these 1-dimensional covers,
    we obtain a covering of the finite-dimensional truncated box $B_n = \prod_{i=1}^n [-a_i, a_i]$
    by $\prod_{i=1}^n k_i$ hypercubes in $\mathbb{R}^n$, each having an $\ell^\infty$-radius of $\eps$. 
    
    For the remaining tail coordinates $i > n$, since $a_i \le \eps$, the projection of $\SC$
    onto these coordinates is entirely contained within a single $\ell^\infty$ ball of radius $\eps$
    centered at the origin.
    By padding the centers of our $\mathbb{R}^n$ hypercubes with zeros for all $i > n$,
    we can thus construct a valid $\eps$-cover for the entire set $\SC$ in $\ell^\infty(\NN)$.
    The total number of balls required is bounded by the number of centers:
    \[
        \cn(\SC, \eps) \le \prod_{i=1}^n k_i \le \prod_{i=1}^{n(\eps)} \left( \frac{a_i}{\eps} + 1 \right).
    \]
    
    We now show the lower bound using a volume argument.
    Let $N:= \cn(\SC,\eps)$, meaning that $\SC$ can be covered by $N$ $\ell^{\infty}$-balls of radius $\eps$.
    Let $\pi_n : \ell^\infty(\NN) \to \mathbb{R}^n$ denote the projection map onto the first $n$ coordinates.
    The projection of $\SC$ under $\pi_n$ is exactly the box $B_n = \prod_{i=1}^n [-a_i, a_i]$.
    Furthermore, the projection of any $\ell^\infty(\NN)$-ball of radius $\eps$ is an $\ell^\infty$-ball in $\mathbb{R}^n$ of radius $\eps$.
    
    Since the $N$ original balls cover $\SC$, their projections must completely cover $B_n$.
    In $\mathbb{R}^n$, the Lebesgue measure (volume) of each projected $\eps$-ball is $(2\eps)^n$.
    The volume of the target set $B_n$ is $\prod_{i=1}^n (2a_i) = 2^n \prod_{i=1}^n a_i$. 
    Hence,
    \[
        N \cdot (2\eps)^n \ge \operatorname{Vol}(B_n) = 2^n \prod_{i=1}^n a_i.
    \]
    Rearranging terms, we establish the lower bound for $N= \cn(\SC, \eps)$.
    This completes the proof.
\end{proof}

\begin{lemma}
    \label{Hilbert_cube_compact_convex}
    Let $\SC$ be as defined at the start of \cref{sec:example_lmd<umd}.
    Then $\SC$ is a convex and compact subset of $\ell^{\infty}(\NN)$.
\end{lemma}
\begin{proof}
    The convexity is straightforward.
    Since uniform convergence implies pointwise convergence
    and all intervals in the Cartesian product are closed, the set $\SC$ is closed.
    By \cref{bounds_for_covering_numbers_Hilbert_cube}, the set $\SC$ is also totally bounded.
    Hence $\SC$ is closed and totally bounded and thus compact.
\end{proof}

Now that we have established bounds on the covering numbers of $\SC$
(see \cref{bounds_for_covering_numbers_Hilbert_cube}),
we can proceed to prove \cref{umd_lmd_of_hilbert_cube_s_t}.

\begin{proof}[Proof of \cref{umd_lmd_of_hilbert_cube_s_t}]
 We will repeatedly make use of the following inequalities. 
 For every $k\ge k_0$ we have
 \begin{equation}
     \label{eqn_6_1}
     2^{ta^k}
     \ge N_k = \lfloor 2^{ta^k}\rfloor 
     \ge 2^{ta^k}-1
     = 2^{ta^k}\cdot \bigl(1- 2^{-ta^k}\bigr)
     \ge 2^{ta^k}\cdot c_1
 \end{equation}
 where $c_1 := 1 - 2^{-ta^{k_0}}>0$.

 \medskip{}

 \textbf{Step 1:}
 In the first part of the proof, we consider two particular sequences in order to provide
 an upper bound for $\lmd(\SC)$ and a lower bound for $\umd(\SC)$.
 Since $a>1$, we have 
 \begin{equation*}
     \frac{2^{-a^k}}{2^{-a^{(k+1)}}} 
     = 2^{a^k(a-1)} 
     \ge 2^{a^{k_0+1}(a-1)}
     >1
     \quad \text{for all $k\ge k_0+1$.}
 \end{equation*}
 It follows that there exists $\delta \in (0,1)$ such that
 \begin{equation*}
     \frac{2^{-a^k}}{2^{-a^{(k+1)}}} > \frac{1+\delta}{1-\delta} >1 \quad \text{for all $k\ge k_0+1$.}
 \end{equation*}
 We define 
 \begin{equation*}
 \eps_k := (1-\delta) \cdot 2^{-a^k}
 \quad \text{and} \quad
 \eps'_k := (1+\delta) \cdot 2^{-a^k}
 \quad \text{ for $k\ge k_0+1$.}
 \end{equation*}
 By the previous computation, we have $\eps_k > \eps'_{k+1}$ for all $k\ge k_0+1$.
 The idea behind the definition of $\eps_k$ and of $\eps_k'$ is that they are slightly smaller
 and slightly larger than the value $\eps= 2^{-a^k}$,
 at which the number $n(\eps)$ jumps from $N_{k-1}$ to $N_k$.

 \medskip{}

 \textbf{Step 1a:} We show that $\umd(\SC) \ge t$.
 Consider the sequence $(\eps_k)_{k\ge k_0+1}$ and let $k\ge k_0+1$ be arbitrary.
 For all $i \ge N_k+1$ we have $a_i\le 2^{-a^{k+1}} < \eps'_{k+1} \le \eps_k$.
 In addition, for all $i\le N_k$, we have $a_i \ge 2^{-a^k}>\eps_k$.
 Hence $n(\eps_k) = N_k$.
 By \cref{bounds_for_covering_numbers_Hilbert_cube}, we get
 \begin{equation*}
    \cn(\SC,\eps_k)
    \ge \prod_{i=1}^{n(\eps_k)} \frac{a_i}{\eps_k}
    \ge \prod_{i=1}^{N_k} \frac{2^{-a^k}}{(1-\delta)\cdot 2^{-a^k}}
    = \Bigl( \frac{1}{1-\delta}\Bigr)^{N_k}.
 \end{equation*}
 Hence
 \begin{equation}
     \label{eqn_6_2}
    \log_2\bigl(\cn(\SC,\eps_k)\bigr)
    \ge N_k \log_2\Bigl(\frac{1}{1-\delta}\Bigr).
 \end{equation}
 Using \eqref{eqn_6_2} and \eqref{eqn_6_1}, we obtain
 \begin{align}
    \log_2\log_2\bigl(\cn(\SC,\eps_k)\bigr)
    &\overset{\eqref{eqn_6_2}}{\ge} \log_2(N_k) + \log_2\log_2\Bigl(\frac{1}{1-\delta}\Bigr)\nonumber\\
    &\overset{\eqref{eqn_6_1}}{\ge} ta^k +\log_2(c_1) + \log_2\log_2\Bigl(\frac{1}{1-\delta}\Bigr).
    \label{eqn_6_3}
 \end{align}
 On the other hand, we have
 \begin{equation}
     \log_2\Bigl(\frac{1}{\eps_k}\Bigr) 
     = \log_2\Bigl( \frac{2^{a^k}}{1-\delta}\Bigr)
     = a^k + \log_2\Bigl(\frac{1}{1-\delta}\Bigr).
    \label{eqn_6_4}
 \end{equation}
 Inserting \eqref{eqn_6_4} into \eqref{eqn_6_3}, we obtain
 \begin{equation}
     \label{eqn_6_5}
    \log_2\log_2\bigl(\cn(\SC,\eps_k)\bigr)
    \ge t\log_2\Bigl(\frac{1}{\eps_k}\Bigr)
        -t \log_2\Bigl(\frac{1}{1-\delta}\Bigr)
        +\log_2(c_1) 
        + \log_2\log_2\Bigl(\frac{1}{1-\delta}\Bigr).
 \end{equation}
 Using that $\eps_k\to 0$ as $k\to \infty$ and that \eqref{eqn_6_5} holds for all $k\ge k_0+1$ at (a), we deduce that
 \begin{equation*}
     \umd(\SC)
     =\limsup_{\eps\to 0} \frac{\log_2\log_2(\cn(\SC,\eps))}{\log_2(1/\eps)}
     \ge \limsup_{k\to \infty}\frac{\log_2\log_2(\cn(\SC,\eps_k))}{\log_2(1/\eps_k)}
     \overset{(a)}{\ge} t.
 \end{equation*}
 This is the desired lower bound for $\umd(\SC)$.

 \medskip{}

 \textbf{Step 1b:}
 We show that $\lmd(\SC) \le s$.
 Consider the sequence $(\eps'_k)_{k\ge k_0+1}$ and let\\ $k\ge k_0+2$ be arbitrary.
 For all $i\le N_{k-1}$ we have $a_i \ge 2^{-a^{k-1}} > 2^{-a^{k-1}}(1-\delta) \ge 2^{-a^{k}}(1+\delta)= \eps'_k$.
 In addition, for all $i\ge N_{k-1}+1$, we have $a_i\le 2^{-a^k} < (1+\delta)2^{-a^{k}} =\eps'_k$.
 This shows that $n(\eps'_k) = N_{k-1}$.
 By \cref{bounds_for_covering_numbers_Hilbert_cube}
 and because of $a_i>\eps_k'$ for all $1\le i \le N_{k-1}$, we get
 \begin{equation*}
    \cn(\SC,\eps'_k)
    \le \prod_{i=1}^{N_{k-1}} \Bigl(\frac{a_i}{\eps_k'} +1\Bigr)
    \le \Bigl(\frac{2}{\eps_k'}\Bigr)^{N_{k-1}}.
 \end{equation*}
 Hence, using \eqref{eqn_6_1} at (a) and $a=t/s$ at (b), we infer that
 \begin{equation*}
 \log_2\bigl(\cn(\SC,\eps_k')\bigr)
 \le N_{k-1} \cdot \log_2\Bigl(\frac{2}{\eps_k'}\Bigr)
 \overset{(a)}{\le} 2^{ta^{k-1}} \cdot \log_2\Bigl(\frac{2}{\eps_k'}\Bigr)
 \overset{(b)}{=} 2^{sa^k}\cdot \log_2\Bigl(\frac{2}{\eps_k'}\Bigr).
 \end{equation*}
 Hence
 \begin{equation}
 \log_2\log_2\bigl(\cn(\SC,\eps_k')\bigr)
 \le sa^k + \log_2\log_2\Bigl(\frac{2}{\eps_k'}\Bigr).
     \label{eqn_6_6}
 \end{equation}
 By definition of $\eps'_k$ we have
 \begin{equation}
     \log_2\Bigl(\frac{1}{\eps_k'}\Bigr)
     = \log_2\Bigl(\frac{2^{a^k}}{1+\delta}\Bigr)
     = a^k + \log_2\Bigl(\frac{1}{1+\delta}\Bigr).
     \label{eqn_6_7}
 \end{equation}
 Inserting \eqref{eqn_6_7} into \eqref{eqn_6_6}, we obtain
 \begin{equation}
    \log_2\log_2\bigl(\cn(\SC,\eps_k')\bigr)
    \leq s \log_2\Bigl(\frac{1}{\eps_k'}\Bigr) 
         - s \log_2\Bigl(\frac{1}{1+\delta}\Bigr)
         + \log_2\log_2\Bigl(\frac{2}{\eps_k'}\Bigr).
    \label{eqn_6_8}
 \end{equation}
 Using that $\eps_k'\to 0$ as $k\to \infty$ and that \eqref{eqn_6_8} holds
 for all $k\ge k_0+2$ at (a), we deduce that
 \begin{equation*}
    \lmd(\SC)
    = \liminf_{\eps\to 0} \frac{\log_2\log_2\bigl(\cn(\SC,\eps)\bigr)}{\log_2(1/\eps)}
    \le \liminf_{k\to \infty}\frac{\log_2\log_2\bigl(\cn(\SC,\eps_k')\bigr)}{\log_2(1/\eps_k')}
    \le s.
 \end{equation*}
 This is the desired upper bound for $\lmd(\SC)$.

 \medskip{}

 \textbf{Step 2:} In the second step of the proof, we derive the lower bound for $\lmd(\SC)$
 and the upper bound for $\umd(\SC)$. 
 In order to give a lower bound for a liminf and an upper bound for a limsup,
 we have to consider all sufficiently small $\eps>0$ rather than just specific sequences
 as we did in Step~1 of the proof.
 We will distinguish two cases.

 \medskip{}

 \textbf{Step 2.a:}
 In the first case, we consider intervals for $\eps$ where $n(\eps)$ is constant.
 To be precise, we consider all $k\ge k_0+1$ and all $\eps$ satisfying
 \begin{equation}
     \eps_{k+1}'= (1+\delta)\cdot 2^{-a^{(k+1)}} \le \eps \le (1-\delta) \cdot 2^{-a^{k}} = \eps_k.
     \label{eqn_6_15}
 \end{equation}
 We have for all $i\le N_k$ that $a_i \ge 2^{-a^k} > (1-\delta)\cdot 2^{-a^k} \ge \eps$ and for all $i\ge N_{k}+1$ that
 $a_i \le 2^{-a^{k+1}} < (1+\delta) \cdot 2^{-a^{k+1}} \le \eps$.
 Hence $n(\eps) = N_k$.

 For the logarithms of the inverse of $\eps$, we have the bounds
 \begin{equation}
     \log_2\Bigl(\frac{1}{\eps}\Bigr)
     \le \log_2\Bigl( \frac{2^{a^{k+1}}}{1+\delta}\Bigr)
     =  a^{k+1} - \log_2(1+\delta)
     \le a^{k+1}
     \label{eqn_6_9}
 \end{equation}
 and
 \begin{equation}
    \log_2\Bigl(\frac{1}{\eps}\Bigr)
\ge \log_2\Bigl( \frac{2^{a^k}}{1-\delta}\Bigr)
= a^{k} - \log_2(1-\delta)
\ge a^{k}.
     \label{eqn_6_10}
 \end{equation}

 For the covering numbers, \cref{bounds_for_covering_numbers_Hilbert_cube} implies the lower bound
 \begin{equation*}
    \cn(\SC,\eps)
    \ge \prod_{i=1}^{N_k} \frac{a_i}{\eps}
    \ge \prod_{i=1}^{N_k} \frac{2^{-a^k}}{(1-\delta) \cdot 2^{-a^k}} 
    = \Bigl( \frac{1}{1-\delta}\Bigr)^{N_k}.
 \end{equation*}
 Hence, using \eqref{eqn_6_1} at (a), we get
 \begin{equation*}
    \log_2\bigl(\cn(\SC,\eps)\bigr)
    \ge N_k \cdot \log_2\Bigl(\frac{1}{1-\delta}\Bigr)
    \overset{(a)}{\ge} 2^{ta^k} \cdot c_1 \cdot \log_2\Bigl(\frac{1}{1-\delta}\Bigr).
 \end{equation*}
 Hence, using \eqref{eqn_6_9} at (a), and recalling that $a=t/s$, we get
 \begin{align}
    \log_2\log_2\bigl(\cn(\SC,\eps)\bigr)
    &\ge t a^k + \log_2(c_1) + \log_2\log_2\Bigl(\frac{1}{1-\delta}\Bigr)\nonumber\\
    &\overset{(a)}{\ge} t \cdot \frac{1}{a} \log_2\Bigl(\frac{1}{\eps}\Bigr)  
+ \log_2(c_1) + \log_2\log_2\Bigl(\frac{1}{1-\delta}\Bigr)\nonumber\\
    &= s \log_2(1/\eps) + c_2,
    \label{eqn_6_12}
 \end{align}
 where 
 \begin{equation*}
     c_2 := \log_2(c_1) + \log_2\log_2\Bigl(\frac{1}{1-\delta}\Bigr) \in \RR.
 \end{equation*}

 For the covering numbers, \cref{bounds_for_covering_numbers_Hilbert_cube} also implies the upper bound
 \begin{equation*}
    \cn(\SC,\eps)
\le \prod_{i=1}^{N_k} \Bigl(\frac{a_i}{\eps} +1\Bigr)
\le \Bigl(\frac{2}{\eps}\Bigr)^{N_k},
 \end{equation*}
 where we also used that $a_i\ge \eps$ for $1\le i \le N_k$.
 Hence, by \eqref{eqn_6_1}, we have
 \begin{equation*}
    \log_2\Bigl(\cn(\SC,\eps)\Bigr)
    \le N_k \cdot \log_2\Bigl(\frac{2}{\eps}\Bigr)
    \le 2^{ta^k}\cdot \log_2\Bigl(\frac{2}{\eps}\Bigr).
 \end{equation*}
 Hence, using also \eqref{eqn_6_10} at (a), we get
 \begin{align}
    \log_2\log_2\bigl(\cn(\SC,\eps)\bigr)
&\le ta^k + \log_2\log_2\Bigl(\frac{2}{\eps}\Bigr)\nonumber\\
&\overset{(a)}{\le} t \log_2(1/\eps) + \log_2\log_2\Bigl(\frac{2}{\eps}\Bigr).
\label{eqn_6_11}
 \end{align}

 \medskip{}

 \textbf{Step 2.b:}
 In the second case, we consider the values of $\eps$ around the jumps of $n(\eps)$.
 More precisely, we consider all $k\ge k_0+2$ and all $\eps$ satisfying
 \begin{equation}
   \eps_k' = (1+\delta)\cdot 2^{-a^{k}} \ge \eps \ge (1-\delta) \cdot 2^{-a^{k}} = \eps_{k}.
   \label{eqn_6_16}
 \end{equation}
 Notice that $\eps_k'$ is the lower endpoint of the $(k-1)$-th interval and $\eps_k$
 is the upper endpoint of the $k$-th interval considered in \eqref{eqn_6_15}.
 In particular, since here we assume $k\ge k_0+2$, the inequalities \eqref{eqn_6_12}
 and \eqref{eqn_6_11} apply to $\eps=\eps_k'$ and to $\eps=\eps_k$.
 The idea is now to approximate $\cn(\SC,\eps)$ and $\log_2(1/\eps)$
 by the corresponding values for $\eps_k$ and $\eps_k'$.

 For the logarithms of the inverses of $\eps$, we have for all $u>0$
 the following bounds for the approximation errors:
 \begin{equation}
     \abs{ \log_2 \Bigl(\frac{u}{\eps}\Bigr) - \log_2\Bigl(\frac{u}{\eps_k'}\Bigr)} 
     =\log_2\Bigl(\frac{\eps_k'}{\eps}\Bigr) 
     \le \log_2\Bigl( \frac{\eps_k'}{\eps_k}\Bigr)
     = \log_2\Bigl(\frac{1+\delta}{1-\delta}\Bigr)
     \label{eqn_6_13}
 \end{equation}
 and
 \begin{equation}
     \abs{ \log_2 \Bigl(\frac{u}{\eps}\Bigr) - \log_2\Bigl(\frac{u}{\eps_k}\Bigr)} 
     =\log_2\Bigl(\frac{\eps}{\eps_k}\Bigr) 
     \le \log_2\Bigl( \frac{\eps_k'}{\eps_k}\Bigr)
     = \log_2\Bigl(\frac{1+\delta}{1-\delta}\Bigr).
     \label{eqn_6_14}
 \end{equation}
 
 For the covering numbers, we will use the monotonicity of $\cn(\SC,\eps)$ in $\eps$.
 For the upper bound, we use the monotonicity of $\cn(\SC,\eps)$ at (a),
 then \eqref{eqn_6_11} for $\eps_k$ instead of $\eps$ at (b),
 and \eqref{eqn_6_14} at (c), resulting in
 \begin{align}
    \log_2\log_2\bigl(\cn(\SC,\eps)\bigr)
    &\overset{(a)}{\le} \log_2\log_2\bigl( \cn(\SC, \eps_k)\bigr)\nonumber\\
    &\overset{(b)}{\le}  t \log_2(1/\eps_k) +  \log_2\log_2\Bigl(\frac{2}{\eps_k}\Bigr)\nonumber\\
    &\overset{(c)}{\le} t \log_2(1/\eps)
                        + t\log_2\Bigl(\frac{1+\delta}{1-\delta}\Bigr)
                        + \log_2
                          \Bigl(
                            \log_2(2/\eps)
                            + \log_2\Bigl(\frac{1+\delta}{1-\delta}\Bigr)
                          \Bigr).
    \label{eqn_6_18}
 \end{align}
 For the lower bound, we use the monotonicity of $\cn(\SC,\eps)$ at (a),
 followed by \eqref{eqn_6_12} for $\eps_{k}'$ instead of $\eps$ at (b),
 and \eqref{eqn_6_13} at (c), resulting in
 \begin{align}
    \log_2\log_2\bigl(\cn(\SC,\eps)\bigr)
    &\overset{(a)}{\ge} \log_2\log_2\bigl(\cn(\SC,\eps_k')\bigr)\nonumber\\
    &\overset{(b)}{\ge} s \log_2(1/\eps_k') + c_2\nonumber\\
    &\overset{(c)}{\ge} s \log_2(1/\eps) - s \log_2\Bigl(\frac{1+\delta}{1-\delta}\Bigr) + c_2.
    \label{eqn_6_17}
 \end{align}

 \medskip{}

 \textbf{Step 2.c:} We combine Step~2.a and Step~2.b.
 Notice that the unions of all the intervals considered in \eqref{eqn_6_15}
 and in \eqref{eqn_6_16} cover $(0,\eps_{k_0+1}]$.
 Combining \eqref{eqn_6_12} and \eqref{eqn_6_17}, we infer that
 \begin{equation*}
    \log_2\log_2\bigl(\cn(\SC,\eps)\bigr)
    \ge s \log_2(1/\eps) - s \log_2\Bigl(\frac{1+\delta}{1-\delta}\Bigr) + c_2  
 \end{equation*}
 for all $\eps\in(0,\eps_{k_0+1}]$.
 It follows that $\lmd(\SC)\ge s$. This is the desired lower bound for $\lmd(\SC)$.

 Combining \eqref{eqn_6_11} and \eqref{eqn_6_18}, we obtain
 \begin{equation*}
    \log_2\log_2\bigl(\cn(\SC,\eps)\bigr)
    \leq t \log_2(1/\eps)
         + t\log_2\Bigl(\frac{1+\delta}{1-\delta}\Bigr)
         + \log_2\Bigl( \log_2(2/\eps) + \log_2\Bigl(\frac{1+\delta}{1-\delta}\Bigr)\Bigr)
 \end{equation*}
 for all $\eps\in(0,\eps_{k_0+1}]$.
 It follows that $\umd(\SC) \le t$.
 This is the desired upper bound for $\umd(\SC)$.

 Combining the upper and lower bounds that we derived in step 1 and step 2, we deduce the main claim.
\end{proof}
 
\section{Standard technical lemmas}
\label{sec:standard_technical_lemmas}

\begin{lemma}
    \label{lemma_technical_log_log_asymp}
    Let $\eps_f>0$, let $f\colon (0,\eps_f)\to \NN_{\ge2}$, and let $s>0$.
    If \begin{equation*}
       \liminf_{\eps\downarrow 0} \frac{\log_2\log_2 f(\eps)}{\log_2\big(\frac{1}{\eps}\big)} \ge s,
    \end{equation*}
    then for every $t\in (0,s)$ there exists $\eps_0\in (0,\eps_f)$
    such that $f(\eps) \ge 2^{\eps^{-t}}$ for all $\eps\in (0,\eps_0)$.
\end{lemma}

\begin{proof}
    Let $t \in (0,s)$.
    By definition of $\liminf$, there exists $\eps_0\in (0,\eps_f)$
    such that for all $\eps\in (0,\eps_0)$ we have
    \begin{equation*}
       \frac{\log_2\log_2 f(\eps)}{\log_2\big(\frac{1}{\eps}\big)} >t.
    \end{equation*}
    By rearranging terms and by applying the map $\nu \mapsto 2^{\nu}$ twice, we deduce the claim.
\end{proof}

\begin{lemma}\label{lemma_ball_inclusions}
 Let $(\XX,\metr)$ be a non-empty quasi-metric space
 with modulus of concavity $K\ge 1$, let $x\in \XX$, and let $r>0$.
 Then, for every $y \in \ball(x,r)$ we have
 \begin{equation*}
    \ball(x,r) \subset \ball(y,2Kr).
 \end{equation*}
\end{lemma}

\begin{proof}
    For $z\in \ball(x,r)$ we have $\metr(z,y) \le K \cdot \bigl( \metr(z,x) +\metr(x,y)\bigr)
       \le 2Kr$.
\end{proof}

\begin{lemma}\label{lem:InternalVSExternalCoveringNumbers}
  Let $(\XX, d)$ be a quasi-metric space with triangle constant $k \geq 1$.
  Then for any subset $\emptyset \neq \SC \subseteq \XX$ and every $\eps > 0$, we have
  \[
    \cn(\SC, \eps)
    \leq \extcn\left(\SC, \XX, \frac{\eps}{2k}\right).
  \]
\end{lemma}

\begin{proof}
  Let $N := \extcn\left(\SC, \XX, \frac{\eps}{2k}\right)$.
  If $N= \infty$, the claim is trivial; hence, we can assume that $N<\infty$.
  By definition of $N$, there exist $x_1, \dots, x_N \in \XX$ with $\SC \subseteq \bigcup_{i=1}^N \ball\left(x_i, \frac{\eps}{2k}\right)$.
  Now, for $i \in \{1, \dots, N\}$, there are two cases:
  \begin{enumerate}[label=(\roman*)]
    \item if $\SC \cap \ball\left(x_i, \frac{\eps}{2k}\right) \neq \emptyset$,
          choose $y_i \in \SC \cap \ball\left(x_i, \frac{\eps}{2k}\right)$;

    \item if $\SC \cap \ball\left(x_i, \frac{\eps}{2k}\right) = \emptyset$,
          let $y_i \in \SC$ be arbitrary.
  \end{enumerate}
  We claim that $\{y_1, \dots, y_N\} \subseteq \SC$ is an $\eps$-net for $\SC$.
  Indeed, let $x \in \SC$ be arbitrary.
  Then there exists $i \in \{1, \dots, N\}$ such that
  $x \in \SC \cap \ball\left(x_i, \frac{\eps}{2k}\right) \neq \emptyset$,
  and hence $y_i \in \SC \cap \ball\left(x_i, \frac{\eps}{2k}\right)$.
  This implies
  \begin{align*}
    d(x, y_i) &\leq k \cdot \bigl(d(x, x_i) + d(x_i, y_i)\bigr) \\
              &\leq k \cdot \left(\frac{\eps}{2k} + \frac{\eps}{2k}\right)
               =    \eps.
  \end{align*}
  Hence, we have shown $\cn(\SC, \eps) \leq N = \extcn\left(\SC, \XX, \frac{\eps}{2k}\right)$.
\end{proof}

The following results are proven in the case of metric spaces in \cite[p. 278]{buragoCourseMetricGeometry2001}.
The same proofs apply (verbatim) in the case of quasi-metric spaces.

\begin{lemma}\label{maximal_separated_subsets_exist}
  Let $(\SC,\metr)$ be a non-empty quasi-metric space and let $\eps>0$.
  Then a maximal $\eps$-separated subset of $\SC$ exists.
\end{lemma}

\begin{lemma}
    \label{maximal_separated_subsets_are_nets}
 Let $(\SC,\metr)$ be a non-empty quasi-metric space and let $\eps>0$.
 Let $P\subset \SC$ be a maximal $\eps$-separated subset.
 Then $P$ is an $\eps$-net for $\SC$.
\end{lemma}